\documentclass[a4paper,reqno]{amsart}

\usepackage{amsmath}
\usepackage{amsthm}
\usepackage{amssymb}
\usepackage{amscd} %diagrams

\usepackage{xcolor}
\usepackage{bbm} %bold numbers
\usepackage{mathtools,mathrsfs} %nice calligraphy

\usepackage[normalem]{ulem} %COMMENT THIS LINE FOR THE FINAL VERSION
\usepackage{verbatim} %DO NOT comment this line (comment environment)

\usepackage{enumitem}

\usepackage{xy}
\xyoption{matrix} \xyoption{arrow} \xyoption{arc} \xyoption{color}

\newtheoremstyle{newremark}
  {5pt}
  {5pt}
  {\rmfamily}
  {}
  {\rmfamily\bf}
  {.}
  {.5em}
  {}

\newtheorem{theorem}{Theorem}
\newtheorem{lemma}[theorem]{Lemma}
\newtheorem{corollary}[theorem]{Corollary}
\newtheorem{proposition}[theorem]{Proposition}

\theoremstyle{newremark}
\newtheorem{remark}[theorem]{Remark}
\newtheorem{definition}[theorem]{Definition}

\newtheorem*{definition*}{Definition} %no numbering for Theorem*
\newtheorem*{notations*}{Notations}

\numberwithin{theorem}{section}
\numberwithin{equation}{section}

\newcommand{\R}{\mathbb{R}} %real numbers

\newcommand{\bbS}{\mathbb{S}}

\DeclareMathOperator{\rmdiam}{\mathrm{diam}} %diameter

\DeclareMathOperator{\rmLip}{\mathrm{Lip}} %Lipschitz constant

\def\Xint#1{\mathchoice
{\XXint\displaystyle\textstyle{#1}}%
{\XXint\textstyle\scriptstyle{#1}}%
{\XXint\scriptstyle\scriptscriptstyle{#1}}%
{\XXint\scriptscriptstyle
\scriptscriptstyle{#1}}%
\!\int}
\def\XXint#1#2#3{{%
\setbox0=\hbox{$#1{#2#3}{\int}$}
\vcenter{\hbox{$#2#3$}}\kern-.5\wd0}}

\def\dashint{\Xint-}

\renewcommand{\leq}{\leqslant}
\renewcommand{\geq}{\geqslant}
\renewcommand{\subset}{\subseteq}
\renewcommand{\supset}{\supseteq}

\newcommand{\eps}{\varepsilon}

\newcommand{\abs}[1]{\left| #1 \right|} % for absolute value
 
\newcommand{\trans}{\mathsf{t}}
\newcommand{\trac}{{\rm tr}}

\begin{document}

%=================
% TITLE AND AUTHOR
%=================

\title[Torus-like solutions for the LdG model]{Torus-like  solutions for the Landau-de Gennes model. \\ Part I: the Lyuksyutov regime}

\author{Federico Dipasquale}
\address{Dipartimento di Matematica, Sapienza Universit\`{a} di Roma, P.le Aldo Moro 5, 00185 Roma, Italy}
\email{dipasquale@mat.uniroma1.it}

\author{Vincent Millot}
\address{LAMA, Univ Paris Est Creteil, Univ Gustave Eiffel, UPEM, CNRS, F-94010, Cr\'{e}teil, France}
\email{vincent.millot@u-pec.fr}

\author{Adriano Pisante}
\address{Dipartimento di Matematica, Sapienza Universit\`{a} di Roma, P.le Aldo Moro 5, 00185 Roma, Italy}
\email{pisante@mat.uniroma1.it}

%\date{\today}

%=========
% ABSTRACT
%=========

\begin{abstract}
We study global minimizers of a continuum Landau-de Gennes energy functional for nematic liquid crystals, in three-dimensional domains, under a Dirichlet boundary condition. In a relevant range of parameters (which we call {\sl Lyuksyutov regime}), the main result establishes the nontrivial topology of the biaxiality sets of minimizers for a large class of boundary conditions including  the homeotropic boundary data. To achieve this result, we first study minimizers subject to a physically relevant norm constraint (the {\sl Lyuksyutov constraint}), and show their regularity up to the boundary. From this regularity, we rigorously derive the norm constraint from the asymptotic Lyuksyutov regime. As a consequence,  isotropic melting  is avoided by unconstrained minimizers in this regime, which then allows us to analyse their biaxiality sets. In the case of a nematic droplet, it also implies that 
the radial hedgehog is an unstable equilibrium in the same regime of parameters. Technical results of  this paper will be largely employed 
in \cite{DMP1,DMP2}, where we  prove that biaxiality level sets are generically finite unions of tori for smooth  configurations minimizing the energy in restricted classes of axially symmetric maps satisfying a topologically nontrivial boundary condition.
\end{abstract}


%\thanks{{\it Aknowledgements.} }

\maketitle

%==================
% TABLE OF CONTENTS
%==================

\tableofcontents

%%%%%%%%%%%%%%%%%%%%%%%%%%%%%%%%%%%%%%%%%%%%%%%%%%%%%%%%%%
%%%%%%%%%%%%%%%%%%%%%%%%%%%%%%%%%%%%%%%%%%%%%%%%%%%%%%%%%%

\section{Introduction}\label{sec:1}
Nematic liquid crystals (NLC) are mesophases of matter between the liquid and the solid phases. Nematic molecules typically have elongated shape, approximately rod-like, and can flow freely, like in a liquid, which forces  their long axes  to align locally along some common direction. This feature is the key for the extreme responsivity of nematics to external stimuli, which in turn is the reason why they are so useful in technological applications. Macroscopic configurations of nematics are usually described by continuum theories, the most successful being the phenomenological Landau-de Gennes (LdG) theory \cite{Virga,dGPr,Ba,MoNe} which accounts for the most convincing description of the experimentally observed optical defects \cite{Kleman,LPZZ}. 
In the present article, the first in a series of three, we study minimizing configurations of the Landau-de Gennes energy functional in three dimensional domains under a  Dirichlet boundary condition (or {\sl strong anchoring  condition} in the NLC terminology \cite{Ba}). 
Our primary objective (and main result) in this first part  is to show the emergence of topological structures in minimizers according to the (topological) non triviality of the boundary condition (Theorem \ref{topology}). Here the topology is sought in the so-called biaxial surfaces, level sets of an indicator function, the {\sl signed biaxiality parameter}, associated with any (smooth, non vanishing) configuration, see \eqref{signedbiaxiality} and the discussion below. The non triviality of those surfaces provides the first mathematically rigorous result on the nature of defects which, at least in model geometries, are expected to be of torus type \cite{PeTr,KVZ,GaMk,KV}. Toroidal structures will be found in our companion articles \cite{DMP1,DMP2} where the LdG energy is minimized over a restricted class of symmetric configurations.  The secondary objective here is to prepare the analytical ground for~\cite{DMP1,DMP2}. Before going further, let us now  describe the mathematical setting in  details.
\vskip5pt

According to the LdG theory, configurations of NLC are represented by an order parameter which is a second order tensor called {\sl $Q$-tensor}. It takes values in the $5$-dimensional space 
$$\mathcal{S}_0:=\Big\{Q=(Q_{ij})\in\mathscr{M}_{3\times 3}(\R) : Q=Q^\trans \, , \; \trac(Q)=0 \Big\}\,,$$
where $\mathscr{M}_{3\times3}(\R)$ is the real vector space made of $3\times3$-matrices,  
$Q^\trans$ denotes the transpose of $Q$, and $\trac(Q)$ the trace of $Q$. The space $\mathcal{S}_0$ is endowed with the Hilbertian structure given by the usual (Frobenius) inner product. Since the matrices under consideration are symmetric, the inner product and the induced norm are given by  $P:Q:=\sum_{i,j}P_{ij}Q_{ij}=\trac(PQ)$ and $|Q|^2=\trac(Q^2)$. Upon the choice of an orthonormal basis, $\mathcal{S}_0$ can be identified with the Euclidean space $\R^5$. In particular, $\big\{Q\in\mathcal{S}_0: |Q|=1\big\}=\mathbb{S}^4$, the $4$-dimensional sphere. 

In this way, a NLC configuration contained in a domain $\Omega\subset\R^3$ is represented by a map ${\bf Q}:\overline\Omega\to\mathcal{S}_0$. 
At a given point $x\in\overline\Omega$, one can distinguish three mutually distinct phases: {\it (i)}~the {\sl isotropic phase}, ${\bf Q}(x)=0$; {\it (ii)} the {\sl uniaxial phase}, ${\bf Q}(x)$ has a double eigenvalue; {\it (iii)}~the {\sl biaxial phase}, ${\bf Q}(x)$ has three distinct eigenvalues. A convenient way to measure biaxiality  among configurations away from isotropic points has been introduced in \cite{KWH}. It relies on the (classical) biaxiality parameter $1-6\frac{\rm{tr}(Q^3)^2}{|Q|^6}\in[0,1]$ which vanishes exactly on the uniaxial phase. In turn, the value $1$ characterizes the {\sl maximal biaxiality} with maximal gap between the (normalized) eigenvalues.  A drawback of this parameter comes from the fact that it does not distinguish two different phases within the uniaxial phase (see e.g., \cite{DLR,KV,KVZ}): {\it (i)} the {\sl positive uniaxial phase} where the lowest eigenvalue is  double; {\it (ii)} the {\sl negative uniaxial phase} where the highest eigenvalue is  double. For this reason, we shall use a modified notion of biaxiality parameter that we now define.
\begin{definition}
\label{defsignedbiaxiality}
 For any $Q\in \mathcal{S}_0\setminus\{0\}$, we define the {\em signed biaxiality} parameter of $Q$ as
 \begin{equation}
\label{signedbiaxiality}
\widetilde{\beta}( Q):=\sqrt{6} \,\frac{\rm{tr}( Q^3)}{|Q|^{3}}  \in [-1,1] \, .
\end{equation}
\end{definition}
With this definition at hand, if a matrix $Q$ has a spectrum $\sigma(Q)=\{ \lambda_1 ,\lambda_2, \lambda_3 \} \subset \mathbb{R}$ with eigenvalues in increasing order, then $\widetilde{\beta}(Q)=\pm 1$ iff the minimal/maximal eigenvalue is double (purely positive/negative uniaxial phase), $\widetilde{\beta}(Q)=0$ iff $\lambda_2=0$ and $\lambda_1=-\lambda_3$ (maximal biaxial phase), and $Q=0$ iff $\lambda_1=\lambda_2=\lambda_3$ (isotropic phase). 
\vskip5pt

Let us now assume that the  occupied region $\Omega \subset \mathbb{R}^3$ is a bounded open set  with smooth boundary.  
We consider the Landau-de Gennes energy with the so-called one-constant approximation for the elastic energy density, see e.g. \cite{Ba}. In this case, it takes the form 
\begin{equation}
\label{LDGenergyInitial}
\mathcal{F}_{\rm LG}({\bf Q})=\int_\Omega \frac{L}{2} |\nabla {\bf Q}|^2+F_{\rm B}({\bf Q}) \, dx \, , 
\end{equation}
and it is defined for configurations ${\bf Q}$ in the Sobolev space $W^{1,2}(\Omega; \mathcal{S}_0)$. 
The parameter $L>0$ is a material-dependent elastic constant, and the bulk potential $F_{\rm B}$ is the quartic polynomial
\begin{equation}
\label{abc-potential}
F_{\rm B}(Q):=-\frac{a^2}{2} {\rm tr}(Q^2)-\frac{b^2}{3} {\rm tr} ( Q^3)+\frac{c^2}{4} \big({\rm tr} (Q^2) \big)^2 \, ,
\end{equation}
where $a,b$ and $c$ are also material-dependent strictly positive constants. As usual, it is convenient to subtract-off an additive constant and introduce
\begin{equation}
\label{FBtilde}
\widetilde{F}_{\rm B}(Q):=F_{\rm B}(Q)-\min_{\mathcal{S}_0} F_{\rm B} \, ,
\end{equation}
so that the new potential becomes nonnegative. 
It turns out that the potential is minimal when the signed biaxiality is maximal and the norm equals a characteristic value $s_+>0$ determined by $a$, $b$, and $c$. More precisely,  $\widetilde{F}_{\rm B}(Q)=0$ iff $Q \in \mathcal{Q}_{\rm min}$ where $\mathcal{Q}_{\rm min}$ the vacuum-manifold made of positive uniaxial matrices 
\begin{equation}
\label{vacuum}
\mathcal{Q}_{\rm min}:= \left\{ Q\in \mathcal{S}_0:  Q=s_+ \left( n \otimes n -\frac13 I \right) \, , \quad n \in \mathbb{S}^2  \right\} \, ,
\end{equation}
and
\begin{equation}
\label{splus}
 s_+:=\frac{b^2+\sqrt{b^4+24a^2 c^2}}{4c^2} \, 
 \end{equation}
is the positive root of the characteristic equation 
\begin{equation}
\label{spluseq}
2c^2 t^2 -b^2 t-3 a^2=0\,.
 \end{equation}
Notice that, up to a multiplicative constant, $\mathcal Q_{\rm min}$ is the representation of the real projective plane $\R P^2=\mathbb{S}^2/\{\pm1\}$ through the {\sl  Veronese immersion} into $\mathbb{S}^4$ (see e.g. \cite[p.~80]{BairdWood}). Therefore $\mathcal Q_{\rm min}$ has nontrivial topology, and there are nontrivial homotopy groups $\pi_2(\mathcal{Q}_{\rm min})=\mathbb{Z}$ and $\pi_1(\mathcal{Q}_{\rm min})=\mathbb{Z}_2$, which are relevant for the presence of topological defects. We replace $\mathcal{F}_{\rm LG}$ by the energy functional corresponding to the new potential
\begin{equation}\label{LDGenergytildeori}
\widetilde{\mathcal{F}}_{\rm LG}({\bf Q}):=\int_\Omega \frac{L}{2} |\nabla {\bf Q}|^2+\widetilde{F}_{\rm B}({\bf Q}) \, dx \, , 
\end{equation}
which  is now the sum of two nonnegative terms, one penalizing spatial variations, and the other deviations from the vacuum manifold $\mathcal{Q}_{\rm min}$.

To reduce the dependence on the parameters, we rescale tensor maps by setting 
\begin{equation}\label{normaltens}
{\bf Q}(x)=:s_+ \sqrt{\frac23} Q(x)\,.
\end{equation}
Under this normalization, the vacuum manifold becomes exactly the real projective plane $\mathbb{R}P^2$, where $ \mathbb{R}P^2\subset \mathbb{S}^4$ is precisely embedded (and from now on identified with) through the Veronese immersion (provided by \eqref{vacuum} with $\sqrt{3/2}$ in place of $s_+$). In turn, the energy functional rewrites
\begin{equation}
\label{LDGenergytildebold}
\widetilde{\mathcal{F}}_{\rm LG}({\bf Q})= \frac23  s_+^2 L\,  \mathcal{F}_{\lambda,\mu}(Q) \, ,
\end{equation}
with
\begin{equation}
\label{LDGenergy}
 \mathcal{F}_{\lambda,\mu}(Q):=\int_\Omega \frac{1}{2} |\nabla  Q|^2+\lambda  W(Q)+  \frac{\mu}{4} (1-|Q|^2)^2\, dx \,.
\end{equation}
The reduced parameters $\lambda$ and $\mu$ are given by 
$$ \lambda:=\sqrt{\frac{2}{3}}\frac{b^2s_+}{L}>0 \, , \quad\mu := \frac{a^2}{L}>0 \, , $$
and the reduced smooth potential $W:\mathcal{S}_0\to \R$ is nonnegative and vanishes exactly on~$\mathbb{R}P^2$. 
More precisely, in view of \eqref{splus}-\eqref{spluseq}, the potential $W$ is explicitly given by
\begin{equation}
\label{Newpotential1}
W(Q)=\frac{1}{3\sqrt{6}}\Big(|Q|^3-\sqrt{6}{\rm tr}(Q^3)\Big) +\frac{1}{12\sqrt{6}}\big(3|Q|^2+2|Q|+1\big)\big(|Q|-1\big)^2\,,
\end{equation}
or equivalently, 
\begin{equation}
\label{Newpotential2}
W(Q)=\frac{1}{4\sqrt{6}}|Q|^4-\frac{1}{3}{\rm tr}(Q^3)+\frac{1}{12\sqrt{6}}\,.
\end{equation}
The structure relations \eqref{LDGenergy} and \eqref{Newpotential1} suggests that, in a regime where $\mu$ is large compared to~$\lambda$, the energy $\mathcal{F}_{\lambda,\mu}$ favours rescaled configurations of approximatively unit norm. 

The functional $\widetilde{\mathcal{F}}_{\rm LG}$ has already been studied in several parameters regimes. 
We emphasize the articles \cite{BaZa,Ca,CoLa,HeMaPi,INSZ1,Maj,MaZa,NgZa} as somehow directly related to the present paper, and we refer to \cite{Ba,Gart} for further references. To the best of our knowledge, the reduction \eqref{LDGenergytildebold}-\eqref{LDGenergy} seems to be new, and in turn, the regime where $\mu$ is large compared to~$\lambda$ has not been addressed in the mathematical literature. This is precisely the range of parameters we want to focus on.
\vskip5pt

Following \cite{Ly} (see also \cite{Gart,KVZ,KV,PeTr} for further discussion on the physical ground), 
we first make the fundamental assumption that the norm of an admissible configuration ${\bf Q}$  is given by the constant value proper of the vacuum manifold, i.e.,
\begin{equation}
\label{Lyuk}
|{\bf Q}(x)| \equiv \sqrt{\frac23}\, s_+  \qquad (\hbox{Lyuksyutov constraint}) \, .
\end{equation}
Under the Lyuksyutov constraint, the energy functional takes the form 
$$\widetilde{\mathcal{F}}_{\rm LG}({\bf Q})=\frac23 s_+^2L\mathcal{E}_\lambda(Q)$$ 
for rescaled tensor maps $Q \in W^{1,2}(\Omega; \mathbb{S}^4)$, where
\begin{equation}
\label{LDGenergytilde}
\mathcal{E}_\lambda(Q):=\int_\Omega \frac{1}{2} |\nabla Q|^2+\lambda W(Q)\, dx \, .
\end{equation}
The restriction of the potential $W:\mathcal{S}_0\to\R$ to $\mathbb{S}^4$ is given by 
\begin{equation}
\label{redpotential}
W(Q)=\frac{1}{3\sqrt{6}}\Big(1-\widetilde{\beta}(Q)\Big) \qquad \forall Q\in \mathbb{S}^4\,,
\end{equation}
where $\widetilde{\beta}(Q)$ is the signed biaxiality from Definition \ref{defsignedbiaxiality}.
In particular, $W$ is nonnegative on $\mathbb{S}^4$, $\{W=0\}\cap\mathbb{S}^4=\mathbb{R}P^2$, and $\nabla_{\rm tan} W (Q)=0$ for any $Q \in \mathbb{R}P^2$. As a consequence, when further restricted to the subspace of positive uniaxial configurations $W^{1,2}(\Omega;\mathbb{R}P^2)$, the energy functional \eqref{LDGenergytilde} reduces to the Dirichlet integral, i.e., the Frank-Oseen energy in the one-constant approximation. For an account on the qualitative properties of defects in the Frank-Oseen model, we refer the interested reader to e.g. \cite{AlLi,BCL}.
\vskip5pt

A critical point $Q_\lambda\in W^{1,2}(\Omega;\mathbb{S}^4)$ of $\mathcal{E}_\lambda$ among $\mathbb{S}^4$-valued maps satisfies in the sense of distributions in $\Omega$ the Euler-Lagrange equation
\begin{equation}\label{MasterEq}
\Delta Q_\lambda+|\nabla Q_\lambda|^2Q_\lambda =\lambda \nabla_{\rm tan} W (Q_\lambda) \, ,
\end{equation}
with the tangential gradient of $W$ along $\mathbb{S}^4 \subset \mathcal{S}_0$ given by
$$\nabla_{\rm tan} W (Q)=- \Big(Q^2-\frac{1}{3}I - {\rm tr}(Q^3)Q\Big)  \,.$$
The left hand side of \eqref{MasterEq} is usually called {\sl tension field} of $Q$. It is a tangent field along $Q$ in~$\mathbb{S}^4$, and equation \eqref{MasterEq} is nothing else but a perturbed harmonic map equation for $\mathbb{S}^4$-valued map with the extra term $\lambda \nabla_{\rm tan} W (Q)$ as a source term. Any tensor field $Q$ which is weakly harmonic among $\mathbb{S}^4$-valued maps and lying in the subspace $W^{1,2}(\Omega;\mathbb{R}P^2)$ is also weakly harmonic among maps in $W^{1,2}(\Omega;\mathbb{R}P^2)$\footnote{Observe that the converse implication is not true in general, because the Veronese immersion is minimal but it is not totally geodesic, and the tension field of $Q$ in $\mathbb{S}^4$ could be purely orthogonal to $\mathbb{R}P^2$ but nonzero. Thus, if $Q$ is weakly harmonic among map in the space $W^{1,2}(\Omega;\mathbb{R}P^2)$, i.e., it is a critical point of the Frank-Oseen energy,
%it is a critical point of the energy \eqref{LDGenergytilde} among uniaxial configurations satisfying the norm constraint \eqref{Lyuk}
then  it does not solve \eqref{MasterEq} in general. Hence it is not a critical point of the Landau-de Gennes energy under norm constraint.},  and provides a solution to \eqref{MasterEq}. Since everywhere discontinuous weakly harmonic maps among maps in $W^{1,2}(\Omega;\mathbb{R}P^2)$ do exist (see \cite{Ri}), we expect smoothness of solutions to \eqref{MasterEq} to fail in general, and their regularity  should rely in an essential way on energy minimality. 
\vskip5pt

We consider the minimization of the energy functional $\mathcal{E}_\lambda$ among maps in $W^{1,2}(\Omega;\bbS^4)$ satisfying a Dirichlet boundary condition  in the sense of traces.  We fix a  smooth boundary trace $Q_{\rm b} :\partial \Omega \to \mathbb{S}^4$, and we consider the set of admissible configurations 
\begin{equation}
\label{admissibleconf}
\mathcal{A}_{Q_{\rm b}}(\Omega):=\Big\{Q\in W^{1,2}(\Omega;\mathcal{S}_0): Q_{|\partial\Omega}=Q_{\rm b}\,,\; |Q|=1\text{ a.e. in $\Omega$}\Big\} \subset W^{1,2}(\Omega;\mathbb{S}^4)\,, 
\end{equation}
which is nonempty by \cite{HaLi}. Hence, one can fix a reference extension $\bar{Q}_{\rm b} \in \mathcal{A}_{Q_{\rm b}}(\Omega)$, which, as a matter of fact, can be chosen in $C^0(\overline{\Omega};\mathbb{S}^4)$, or even smooth in the interior since  $\pi_2(\mathbb{S}^4)=0$ (so that density of smooth maps in $\mathcal{A}_{Q_{\rm b}}(\Omega)$ holds, see e.g. \cite{BeZh}). 
By the direct method in the Calculus of Variations, it is routine to show that there exist  minimizers $Q_\lambda \in \mathcal{A}_{Q_{\rm b}}(\Omega)$ of $\mathcal{E}_\lambda$. Concerning regularity, such minimizers are smooth in $\Omega$, and up to to the boundary if $\partial\Omega$ and $Q_{\rm b}$ are regular enough. The energy $\mathcal{E}_\lambda$ being a $0$-order perturbation the Dirichlet energy, regularity can be recovered through the well established theory of minimizing harmonic maps, starting from the pioneering papers \cite{SU1,SU2,SU3} where H\"older continuity up to the boundary for minimizers for a class of energies including \eqref{LDGenergytilde} was first established. For  details on this theory, we refer to the books \cite{GiaqMart,LiWa2,Mos,Simon}. The precise regularity statement we shall rely on is the object of the following theorem.

\begin{theorem}\label{thm:full-regularity}
Assume that $\partial\Omega$ is of class $C^3$ and $Q_{\rm b}\in C^{1,1}(\partial\Omega;\mathbb{S}^4)$. 
If $Q_\lambda$ is a minimizer of $\mathcal{E}_\lambda$ in the class $\mathcal{A}_{Q_{\rm b}}(\Omega)$, then $Q_\lambda \in C^{\omega}(\Omega)\cap C^{1,\alpha}(\overline\Omega)$ for every $\alpha\in(0,1)$. If in addition, 
$Q_{\rm b} \in C^{2,\delta}(\partial \Omega)$ for some $\delta>0$, then $Q_\lambda \in C^{2,\delta}(\overline{\Omega})$, and, finally, if $\Omega$ is a domain with analytic boundary and $Q_{\rm b} \in C^\omega(\partial\Omega)$, then $Q_\lambda \in C^\omega(\overline{\Omega})$.
\end{theorem}

Besides the fact that Theorem \ref{thm:full-regularity} cannot be truly considered as new, we shall present a detailed proof,   for essentially  two reasons. The first and main reason is that it gives us the opportunity to develop a full set of estimates (and identities) available for more general critical points of $\mathcal{E}_\lambda$,  keeping track of the data (domain, boundary condition, parameters). With this respect, it paves the way for our companion articles \cite{DMP1,DMP2}, where we consider minimizers of $\mathcal{E}_\lambda$ in a restricted class of symmetric maps for which  \cite{SU1,SU2,SU3} do not apply, and we perform some asymptotic analysis with respect to those data, see Remark~\ref{finalremark}. To effectively apply our estimates in \cite{DMP1,DMP2}, we had to  rely as less as possible on energy minimality, and we made  explicit  estimates coming from the regularity theory for stationary harmonic maps (see e.g. \cite{Ev,LiWa2,Mos})  which will be crucial to obtain compactness properties for the corresponding  solutions to \eqref{MasterEq}. Our second reason is to present a proof which is self-contained and elementary (even if rather long), aiming to popularize tools from harmonic maps theory, and hoping that it could be useful to the NLC community. 

The proof follows somehow a classical scheme, but it presents some differences we want to comment on. 
The crucial point is  to obtain Lipschitz continuity, as higher order regularity can be then deduced from linear elliptic theories. 
 For both  interior and  boundary regularity, the main steps are: {\sl (i)} monotonicity formulae; {\sl (ii)} strong compactness of blow-ups; {\sl (iii)} constancy of blow-up limits (Liouville property); {\sl (iv)} continuity under smallness of the scaled energy ($\eps$-regularity); {\sl (v)}~Lipschitz continuity.  
The monotonicity formula here is not obtained by inner variations, but instead by a (regularizing) penalty approximation for which we can use the classical Pohozaev multiplier argument (see e.g. \cite{ChenStruwe}, or \cite{MaZa} in the LdG context).  
More precisely, we relax the norm constraint, and passing to the limit in the monotonicity formulae for approximated problems, we obtain interior and boundary monotonicity formulae.
Strong compactness of blow-ups is obtained as usual by construction of comparison maps  based on the Luckhaus lemma  \cite{Lu}, see e.g. \cite{Simon}. 
The constancy of blow-up limits follows from  \cite{SU3} at interior points, and from \cite{Le} at boundary points. 
Our approach to $\eps$-regularity  treats in a unified way the interior and the boundary case, adapting for the latter the clever reflection trick devised in \cite{Scheven} for harmonic maps. H\"{o}lder-continuity under smallness of the scaled energy is not deduced from Hardy-BMO duality as in \cite{Ev}, or from integrability by compensation as in \cite{RiSt}. Here we adapt to our context the elementary iteration approach introduced in \cite{CWY}, as already done in \cite{PaPi} for a similar minimization problem.  
Finally, Lipschitz continuity is obtained using a harmonic replacement argument in the spirit of~\cite{Sch}.
\vskip5pt

With Theorem \ref{thm:full-regularity} at hand, we now remove the norm constraint \eqref{Lyuk}, and we consider the unrestricted energy functional \eqref{LDGenergy}. We minimize  $\mathcal{F}_{\lambda,\mu}$ over maps in $W^{1,2}(\Omega;\mathcal{S}_0)$ still satisfying a Dirichlet boundary condition. Given $Q_{\rm b} \in C^{1,1}(\partial \Omega; \mathbb{S}^4)$, existence of minimizers $Q_\lambda^\mu$ of $\mathcal{F}_{\lambda,\mu}$ in $ W_{Q_{\rm b}}^{1,2}(\Omega;\mathcal{S}_0)$ follows again from the direct method in the Calculus of Variations. In addition, the usual interior and boundary regularity for semilinear elliptic equations applied to the Euler-Lagrange equation satisfied by critical points of $\mathcal{F}_{\lambda,\mu}$ (see \eqref{ELeqQmu}),   
implies that  $Q^\mu_\lambda \in C^{1,\alpha}(\overline{\Omega};\mathcal{S}_0) \cap C^\omega (\Omega; \mathcal{S}_0)$ for every $\alpha \in (0,1)$.
At this stage, we are interested in the asymptotic behaviour of minimizers $Q^\mu_\lambda$ in the range of parameters (that we call {\em Lyuksyutov regime})
\begin{equation}\label{eq:lyuk-regime}
	\lambda=\sqrt{\frac{2}{3}} \frac{b^2s_+}{L}\ \equiv {\rm const} \, , \qquad \mu= \frac{a^2}{L} \to +\infty \, .
\end{equation}
Particular cases are given by  $a^2 \to \infty \, , \, \, b^2 \sim |a|^{-1}$ or $L\to 0 \, , \, \, b^2 \sim L$. These regimes resemble the low-temperature limit and the small elastic constant limit, respectively. For further discussions on this aspect and related asymptotic limits, we refer to Remark~\ref{rmk:rescaling} and \cite{Gart}.

Under these restrictions on the parameters, the last term in $\mathcal{F}_{\lambda,\mu}$ acts as a penalty approximation of the norm constraint \eqref{Lyuk}.   
The family $\{\mathcal{F}_{\lambda,\mu}\}_\mu$ converges to the functional $\mathcal{E}_\lambda$ (in the sense of $\Gamma$-convergence, see e.g. \cite{Bra}), and minimizers of $\mathcal{F}_{\lambda,\mu}$ converge to minimizers of $\mathcal{E}_\lambda$  in the energy space. Then Theorem \ref{thm:full-regularity} comes into play to prove that, in the Lyuksyutov asymptotic regime, the norm of minimizers of $\mathcal{F}_{\lambda,\mu}$ converges uniformly to one, hence providing a mathematical justification of the norm constraint  \eqref{Lyuk} originally introduced in \cite{Ly}. As a byproduct, minimizers do not exhibit the isotropic phase for $\mu$ large enough compared to $\lambda$, the fundamental point of our (upcoming) discussion.

\begin{theorem}\label{thm:noisotropicphase}
Assume that $\partial\Omega$ is of class $C^3$ and $Q_{\rm b}\in C^{1,1}(\partial\Omega;\mathbb{S}^4)$.   
There exist minimizers $Q^\mu_\lambda$  of $\mathcal{F}_{\lambda,\mu}$ in the class $W^{1,2}_{Q_{\rm b}}(\Omega;\mathcal{S}_0)$, and  any such $Q_\lambda^\mu$ belongs to $C^\omega(\Omega)\cap C^{1,\alpha}(\overline\Omega)$ for every $\alpha\in(0,1)$. 
In addition, as $\mu \to \infty$ with $\lambda$ constant (Lyuksyutov regime), the following holds:
\begin{enumerate}
\item there exist a (not relabeled) subsequence and $Q_\lambda \in W^{1,2}(\Omega;\mathbb{S}^4)$ minimizing $\mathcal{E}_\lambda$ in the class $\mathcal{A}_{Q_{\rm b}}(\Omega)$ such that $Q^\mu_\lambda \to Q_\lambda$ strongly in $ W^{1,2}(\Omega;\mathcal{S}_0)$;
\vskip5pt
\item  $\mathcal{F}_{\lambda,\mu}(Q_\lambda^\mu) \to \mathcal{E}_\lambda(Q_\lambda)$ and $\mu \int_\Omega(1-|Q_\lambda^\mu|^2)^2\,\mbox{d}x \to 0$;
\vskip5pt
\item  $|Q_\lambda^\mu| \to 1$  uniformly in $\overline{\Omega}$.
\end{enumerate}
\vskip3pt
\noindent In particular, for each $\lambda>0$, there exists a value $\mu_\lambda=\mu_\lambda(\lambda,\Omega,Q_{\rm b})>0$ such that for $\mu>\mu_\lambda$, any minimizer $Q_\lambda^\mu$ of $\mathcal{F}_{\lambda,\mu}$  satisfies $|Q_\lambda^\mu|>0$ in $\overline\Omega$, i.e., minimizers do not exhibit the isotropic phase.
\end{theorem}

This theorem is very much inspired  by the important paper \cite{MaZa}, where minimizers of $\widetilde{\mathcal{F}}_{\rm LG}$ (see \eqref{LDGenergytildeori}) are studied in the regime $L\to 0$, the other parameters being fixed. It is proved that they converge towards minimizing harmonic maps into $\mathcal{Q}_{\rm min}$ (see \eqref{vacuum}), hence recovering the Frank-Oseen model of NLC in the one-constant approximation. Under our normalization \eqref{normaltens}-\eqref{LDGenergytildebold}, the analysis in \cite{MaZa} corresponds to the regime $\lambda\to\infty$ and $\mu\to\infty$ with $\lambda\sim\mu$, and limits of minimizers are minimizing harmonic maps into $\R P^2$. The 
Lyuksyutov regime \eqref{eq:lyuk-regime} is thus different, and even if Theorem~\ref{thm:noisotropicphase} shares some features with \cite{MaZa}, it complements the result in \cite{MaZa}  giving in the limit another asymptotic theory. 

In Theorem ~\ref{thm:noisotropicphase}, claims (1) and (2)  can be seen as a standard consequence of the $\Gamma$-convergence of the family $\{\mathcal{F}_{\lambda,\mu}\}_\mu$ to $\mathcal{E}_\lambda$, although for the reader's convenience such notion is not explicitly used in the proof (but just mentioned here for readers familiar with it). As a matter of fact, the two claims rely on a sharp two-sided bound on the energies $ \{ \mathcal{F}_{\lambda,\mu}(Q_\lambda^\mu) \}_\mu$,  the lower semicontinuity property of the energy functionals,  the construction of trial sequences, 
and the standard weak compactness in $W^{1,2}$ coming from the equicoercivity of the energies.  
 Then minimum points strongly converge to minimum points in $W^{1,2}$, and the two claims follow as the upper and the lower bound mentioned above coincide.  
As already emphasized, claim (3) is the most important conclusion here as it guarantees that the isotropic phase is avoided  
in the Lyuksyutov regime (as proved in the different low-temperature regime in \cite{Ca} and  \cite{CoLa,HeMaPi}, in 2D and 3D respectively), and uniform convergence of the norm to one  provides a mathematical justification of the Lyuksyutov constraint. 
 The proof of claim (3) is reminiscent from Ginzburg-Landau theories as in  \cite{MaZa}.  
 It is crucially based on Theorem~\ref{thm:full-regularity} where the smoothness of the limiting minimizer $Q_\lambda$ together with the strong $W^{1,2}$-convergence yields  smallness of the scaled energy of $Q^\mu_\lambda$ at a sufficiently small scale. Then, elliptic regularity combined with monotonicity formulae  in a way similar to \cite[Propositions 4 and 6]{MaZa} leads to the uniform convergence of $|Q^\mu_\lambda|$. 
\vskip5pt

To illustrate our discussion so far, let us now consider the model case of a nematic droplet, i.e., when $\Omega=\{ |x|<1 \}$ is the unit ball. The outer unit normal to the boundary is $\overset{\rightarrow}{n}(x)=x/|x|$, and a natural boundary datum is the so-called {\sl radial anchoring}, namely 
\begin{equation}\label{hedgehogbdrydata}
Q_{\rm b}(x)=\sqrt{\frac32} \left( \frac{x}{|x|} \otimes \frac{x}{|x|}-\frac13 I \right)\,.
\end{equation}
Since $\overset{\rightarrow}{n} \colon \partial \Omega \to \mathbb{S}^2$ is harmonic, the homogeneous extension $\bar{H}(x)=Q_{\rm b}\left(x/|x|\right)$ (called the {\sl constant-norm hedgehog}) is a weakly harmonic map from $\Omega$ into $\mathbb{R}P^2$, and it is an energy minimizer of $\mathcal{E}_\lambda$ over $W_{Q_{\rm b}}^{1,2}(\Omega;\mathbb{R}P^2)$ by the lifting property of $W^{1,2}$-maps in $\mathbb{R}P^2$ in \cite{BaZa} and the celebrated result in \cite{BCL}. Moreover, a direct computation shows that $\bar{H}$ is also a weak solution to \eqref{MasterEq}, i.e., it is a critical point of $\mathcal{E}_\lambda$. As $\bar{H}$ is singular at the origin, Theorem \ref{thm:full-regularity} tells us that $\bar{H}$ is not  minimizing $\mathcal{E}_\lambda$ in the class $\mathcal{A}_{Q_{\rm b}}(\Omega)$. We shall  prove in Proposition \ref{consthedgehoginstability} that $\bar{H}$ is in fact strictly  unstable in many directions, employing an argument similar to \cite{SU3}, an explicit computation of the second variation of energy, and  a perturbation localized near the origin. 

Still in the case of a nematic droplet subject to radial anchoring, the energy functional $\mathcal{F}_{\lambda,\mu}$ has an $O(3)$-equivariant (radial) critical point commonly known as {\em radial hedgehog} 
\begin{equation}
\label{radialhedgehog}
H_\lambda^\mu(x):=s_\lambda^\mu(|x|) \left( \frac{x}{|x|} \otimes \frac{x}{|x|}-\frac13 I \right)  \, , \quad 0< |x|<1 \, .
\end{equation}
This solution is obtained from a unique function $s_\mu^\lambda(|x|)$ increasing from $0$ to $\sqrt{3/2}$ solving an ODE with the prescribed values at $|x|=0$ and $|x|=1$, see e.g. \cite{Maj,INSZ2} and the references therein. It turns out to be the unique uniaxial critical point of $\mathcal{F}_{\lambda,\mu}$ w.r.to arbitrary (not necessarily uniaxial) perturbations, see~\cite{Lam}. As the origin is an isotropic point, Theorem \ref{thm:noisotropicphase} shows that $H_\lambda^\mu$ does not minimize $\mathcal{F}_{\lambda,\mu}$ in the class $W^{1,2}_{Q_{\rm b}}(\Omega; \mathcal{S}_0)$, at least for $\mu$ large enough compared to $\lambda$. Hence minimizers cannot be purely uniaxial, and biaxial escape must occur. Using the strong convergence of $H_\lambda^\mu$ to $\bar H$ as $\mu\to\infty$, we pass to the limit in the second variation of $\mathcal{F}_{\lambda,\mu}$ at $H^\mu_\lambda$, and we deduce in Theorem \ref{hedgehoginstability}  the instability of $H_\lambda^\mu$ w.r.to biaxial perturbations for $\mu$ large enough. Both properties are the counterpart in the Lyuksyutov regime of the instability of the radial hedgehog in the low-temperature limit (essentially $a^2 \to \infty$) already proved in \cite{INSZ1} (see also \cite{GaMk2,Maj}) together with the (infinitesimal) biaxial escape phenomenon obtained there (see also Remarks \ref{instcomparison} and \ref{biaxialescape}).
\vskip5pt

Once the smoothness of $Q_\lambda$ and the absence of isotropic phase for $Q^\mu_\lambda$ are established, we can discuss for both cases the topological properties related to the presence of the biaxial phase, and the way they are connected with the topology of the vacuum manifold $\R P^2$. The starting point is that $Q_\lambda$ and $Q^\mu_\lambda$ are configurations satisfying 
\[ (HP_0)   \quad Q \in C^1(\overline{ \Omega}; \mathcal{S}_0 \setminus \{ 0\})\cap C^\omega(\Omega; \mathcal{S}_0) \, . 
  \]
The first assumption at the boundary that we impose on a configuration $Q:\overline\Omega\to\mathcal{S}_0\setminus\{0\}$ is the following
\[ (HP_1)  \qquad  \bar{\beta}:=  \min_{x \in \partial \Omega} \widetilde{\beta}\circ Q(x)>-1 \, . 
  \]
The case $\bar{\beta}=1$ occurs for the main and most natural example of positive uniaxial, i.e., $\R P^2$-valued, boundary condition, which  is 
\begin{equation}
\label{uniaxialbc}
Q_{\rm b} (x)=\sqrt{\frac32} \left(v(x) \otimes v(x)-\frac13 I \right) \quad \hbox{for all } x \in \partial \Omega \, , \quad v \in C^{1,1}(\partial \Omega ;\mathbb{S}^2)  \, .
\end{equation}
In particular, the choice $v(x)=\overset{\rightarrow}{n}(x)$ (the outer unit normal to the boundary~$\partial \Omega$) corresponds to the so-called {\sl homeotropic} boundary condition  (or {\sl radial anchoring}).

Since $\Omega \subset \mathbb{R}^3$ is a bounded open set with smooth boundary,  
we know that $\partial \Omega$ is a finite union of embedded smooth surfaces (in fact, $C^1$-regularity is enough). More precisely, $\partial \Omega = \cup_{i=1}^N S_i$ where the surfaces $S_i$ are smooth, embedded, connected, orientable, and boundaryless. The second (topological) assumption we make on $\Omega$ is
\[ (HP_2) \qquad \Omega \hbox{ is connected and simply connected.} 
 \]
 Under this assumption, each surface $S_i$ has zero genus, so it is an embedded sphere (see Lemma~\ref{simpledomains}). The domain $\Omega$ is thus a topological ball with finitely many disjoint closed balls  removed from its interior. By assumption $(HP_1)$, the maximal eigenvalue $\lambda_{\rm max}(x)$ of $Q(x)$ is simple for every $x \in \partial \Omega$. Hence there exists a corresponding well defined eigenspace map $V_{\rm max} \in C^1(\partial \Omega; \mathbb{R}P^2)$, and this map has a (nonunique) lifting $v_{\rm max} \in C^1(\partial \Omega; \mathbb{S}^2)$ since each surface $S_i$ has zero genus. To enforce the emergence of topology in the minimizers, we finally make a third assumption 
\[ (HP_3) \qquad \deg(v_{\rm max}, \partial \Omega)= \sum_{i=1}^N \deg(v_{\rm max},S_i)   \hbox{ is odd} \, .
 \]
 Notice that this property  only  depends on the map $V_{\rm max}$, and it does not depend on the choice of the lifting $v_{\rm max}$. In case of radial anchoring (i.e., $Q_{\rm b}$ of the form \eqref{uniaxialbc} with $v=\overset{\rightarrow}{n}=v_{\rm max}$), it is satisfied whenever $N$ is odd, that is whenever $\partial \Omega$ has an odd number of connected components (or, equivalently, if the domain $\Omega$ is a topological ball with an even number of disjoint closed ball removed from its interior). 
\vskip5pt

In order to emphasize the consequence of assumptions $(HP_0)$-$(HP_3)$ on a configuration $Q$ satisfying $Q=Q_{\rm b}$ on $\partial\Omega$, let us assume for a moment that $ \, Q_{\rm b}$ is $\R P^2$-valued. Then  
$Q_{\rm b}$ admits a lifting by $(HP_2)$, i.e., $Q_{\rm b}$ is of the form~\eqref{uniaxialbc}. Moreover, any lifting $v \in C^1(\partial \Omega; \mathbb{S}^2)$ of $Q_{\rm b}$ admits a extension $\bar{v}$ in $W^{1,2}(\Omega; \mathbb{S}^2)$ (see e.g. \cite{HaLi}), but no continuous extension because of  $(HP_3)$. As a consequence, $Q_{\rm b}$ admits an extension $\bar{Q}_{\rm b} \in W^{1,2}(\Omega; \mathbb{R}P^2)$ of the form 
\begin{equation}
\label{lifting} 
 \bar{Q}_{\rm b}(x)=\sqrt{\frac32}  \left( \bar{v}(x) \otimes \bar{v}(x) -\frac13 I \right) \, .
 \end{equation}
In view of \cite{BaZa} and $(HP_3)$, any extension of  $Q_{\rm b}$ in $W^{1,2}(\Omega; \mathbb{R}P^2)$  is in fact of the form \eqref{lifting} for a suitable (necessarily) discontinuous map $\bar{v} \in W^{1,2}(\Omega; \mathbb{S}^2)$.  The configuration $Q$ being smooth and without isotropic phase by  $(HP_0)$, it cannot be $\R P^2$-valued, i.e,  positive uniaxial, and biaxial escape must occur again for purely topological reasons.
\vskip5pt

To describe the way a configuration $Q$ encodes some topological information, we shall make use of the biaxiality function as follows.  

\begin{definition}
\label{defbiaxialityregions}
Given  $Q \in C^0(\overline{\Omega};\mathcal{S}_0 \setminus \{ 0\})$, we define its {\em biaxiality function} $\beta:=\widetilde\beta \circ Q$  and for each $t\in [-1,1]$ the associated {\em biaxiality regions} as the closed subsets of $\overline{\Omega}$ 
given by 
 \begin{equation}
 \label{biaxialityregions}
 \{\beta \leq t\}:=\big \{ x\in \overline{\Omega} : \widetilde{\beta} \circ Q(x) \leq t \big\}\quad\text{and}\quad\{ \beta\geq t\}:=\big\{  x\in \overline{\Omega} : \widetilde{\beta} \circ Q(x) \geq t \big\}\, , 
 \end{equation}
 where $\widetilde\beta$ is the signed biaxiality parameter  \eqref{signedbiaxiality}. The corresponding {\em biaxial surfaces} are defined as 
 \begin{equation}\label{eq:biaxial-surface}
  	\{\beta=t \}:=\big \{  x\in \overline{\Omega} : \widetilde{\beta} \circ Q(x) =t \big\}\,.
  \end{equation}
 \end{definition}
 
 Observe that if $t\in (-1,1)$ is a regular value of $\beta$, then biaxial surfaces are smooth surfaces inside $\overline{\Omega}$, possibly with boundary which is anyway smooth and contained in $\partial \Omega$. Moreover, the regions in \eqref{biaxialityregions} are homotopically equivalent to their interior $\{ \beta<t \}$ and $\{ \beta>t \}$, since the biaxial surfaces are actually smooth and serve as their common boundary.
 
We now introduce a notion of {\sl``mutual linking''}, a property that will (partially) encode the topological nontriviality of the biaxiality regions.
 \begin{definition}
 \label{linking}
 Let $A, B \subseteq \overline{\Omega}$ be two disjoint compact subsets. The sets $A$ and $B$ are said to be {\em mutually linked} \footnote{As an example, if $\Omega$ is the unit ball, $A$ is an unknotted embedded copy of $\bbS^1$ into $\Omega$, and $B= \overline{\Omega}\setminus A_\delta$ with $A_\delta$ a sufficiently small tubular neighborhood of $A$, then $A$ and $B$ are mutually linked. 
} if $A$ is not contractible in $\overline{\Omega}\setminus B$ and $B$ is not contractible in $\overline{\Omega}\setminus A$.
 \end{definition}

To illustrate this definition, let us discuss again the case of a nematic droplet. If $\Omega$ is the unit ball and $Q_{\rm b}$ is the hedgehog boundary data \eqref{hedgehogbdrydata}, we expect the minimizers  $Q_\lambda$ or  $Q^\mu_\lambda$ to be axially symmetric  around a fixed axis (in a sense made precise below). In particular, we expect their biaxiality regions \eqref{biaxialityregions} to be axially symmetric as well. More precisely, $\{ \beta < t\}$ with $t\in (-1,1)$ should be an increasing family of axially symmetric solid tori, and the complementary regions $\{ \beta>t \}$ should be kind of distance neighborhoods from the boundary $\partial \Omega$ with cylindrical neighborhoods of the symmetry axis added. In the extreme case $t=\pm 1$, we expect $\{\beta=-1\}$ to be a circle with axial symmetry, and $\{ \beta=1\}$ to be the sphere $\partial \Omega$ with the segment connecting the two antipodal points lying on the symmetry axis added. Clearly sub and superlevel of the biaxiality function should be mutually linked in the sense of Definition \ref{linking} above. This conjectural picture is supported by numerical simulations as already detailed in \cite{PeTr,KVZ,GaMk,KV}, where authors refer to it as the {\bf ``torus solution''} of the Landau-de Gennes model. For the nematic droplet with radial anchoring, the situation clearly reminds the one corresponding to the Hopf fibration 
 \[ \mathbb{C} \times \mathbb{C} \supset \mathbb{S}^3 \overset{\Phi}{\longrightarrow} \mathbb{S}^2 \subset \mathbb{C} \times \mathbb{R} \, , \qquad \Phi(z_1,z_2)=(2 z_1 \overline{z_2} , |z_1|^2-|z_2|^2) \,, \]
where the subsets $\{ |z_1|^2-|z_2|^2 > t \}$ and $\{ |z_1|^2-|z_2|^2 <t \}$ with $t\in (-1,1)$ form a decomposition of $\mathbb{S}^3$ into two disjoint mutually linked solid tori (a so-called Heegaard splitting).
\vskip5pt

Once again, Theorems \ref{thm:full-regularity} and \ref{thm:noisotropicphase} makes assumption  $(HP_0)$ available for $Q_\lambda$ and $Q^\mu_\lambda$ with $\mu$ larger than the constant $\mu_\lambda=\mu_\lambda(\lambda,\Omega,Q_{\rm b})$ (provided by Theorem \ref{thm:noisotropicphase}). It allows us to prove a weak counterpart of the conjectural picture described in the example above, which is therefore the main result of this paper.

\begin{theorem}\label{topology}
Assume that $\partial\Omega$ is of class $C^3$ and $Q_{\rm b}\in C^{1,1}(\partial\Omega;\mathbb{S}^4)$. Let $Q$ be either a minimizer of $\mathcal{E}_\lambda$ over $\mathcal{A}_{Q_{\rm b}}(\Omega)$, 
or  a minimizer of $\mathcal{F}_{\lambda,\mu}$ over $W^{1,2}_{Q_{\rm b}}(\Omega;\mathcal{S}_0)$ with $\mu>\mu_\lambda$ so that $(HP_0)$ holds. If assumptions $(HP_1)$-$(HP_3)$ also hold (e.g., $\Omega$ is connected and simply connected, $\partial \Omega$ has an odd number of connected components, and  $Q_{\rm b}(x)=\sqrt{3/2} (\overset{\rightarrow}{n}(x) \otimes \overset{\rightarrow}{n}(x)-\frac13 I)$ is the radial anchoring), then the biaxiality regions associated with the configuration $Q$ satisfy:
\begin{itemize}
\item[1)] the set of singular values of $\beta=\widetilde\beta\circ Q$ in $[-1,\bar{\beta}]$ is at most countable, and it can accumulate only at $\bar{\beta}$; moreover, for any regular value $-1<t<\bar{\beta}$ of $\beta$ the set $\{ \beta=t\} \subset \Omega$ is a smooth surface with a connected component of positive genus;
\vskip3pt
\item[2)] for any $-1\leq t_1<t_2 < \bar{\beta}$, the sets $\{ \beta \leq t_1\} \subset \Omega$ and $\{ \beta\geq t_2\}\subset \overline{\Omega}$ are nonempty, compact, and not simply connected; 
\vskip3pt
\item[3)] if in addition $Q \in C^\omega(\overline{\Omega})$ and $\bar{\beta}=1$, then the set of critical values is finite and $\{ \beta=1 \}\subset \overline{\Omega}$ is nonempty, compact, and not simply connected; in particular $\{ \beta=1 \}\cap \Omega$ is not empty;   
\vskip3pt
\item[4)] for any $-1\leq t_1<t_2 < \bar{\beta}$, if the interval $(t_1,t_2)$ contains no critical value, then $\{ \beta\leq t_1\}$ and $\{ \beta \geq t_2\}$ are mutually linked.
\end{itemize}
\end{theorem}
Claim 1) on discreteness of the set of singular values is a consequence of the analytic Morse-Sard theorem from \cite{SoSo}. The rest of the claim together with claim 2) is proved by contradiction using a degree-counting argument. The key observation  is that on each spherical component of a biaxial surface $\{ \beta=t\}$, the pull back bundle $E={v_{\rm max}}^* F$ of the tangent bundle $F=T\mathbb{S}^2\to \mathbb{S}^2$ under the lifting $v_{\rm max}$ of the eigenspace map $V_{\rm max}$ must be trivial (hence its Euler number vanishes). Then the contradiction coming essentially from $(HP_3)$ ensures that some $S_i$ has positive genus. The argument for 2) and 3) above holds for regular values $t \in (-1,\bar{\beta})$, and the extension to arbitrary values is based on the analytic regularity of $Q$ and the {\L}ojasiewicz retraction theorem \cite{Lo} (it is the only instance where this property is used). Finally, the linking property in 4) follows easily by contradiction using a deformation of the biaxial regions along the positive/negative gradient flow of $\beta$. 
We expect analogous properties to hold also for $t \in (\bar{\beta},1)$, but this case seems to be more subtle since the biaxial surfaces meet the boundary $\partial \Omega$, and we do not have rigorous result in this direction at present. 

As the conclusions of the theorem are  weak counterparts of the properties conjectured for the torus solution on a nematic droplet, we refer to such solutions on a general domain as {\bf ``torus-like solutions''}. It is a very challenging open problem to obtain a precise estimate on the genus of the surfaces $S_i$, if any. 
Any control on it should depend on a subtle role of the genus in giving a possible lower order correction term in the energy expansion of the minimizing configurations.

\begin{remark}\label{finalremark}
In our subsequent papers \cite{DMP1} and \cite{DMP2} of the series, we continue this analysis focusing on axially symmetric configurations. Letting $\mathbb{S}^1$ act by rotation around the vertical axis on an $\mathbb{S}^1$-invariant domain $\Omega \subset \mathbb{R}^3$, and on $\mathcal{S}_0$ by the induced action $\mathcal{S}_0 \ni A\mapsto R \,A \,R^t  \in \mathcal{S}_0$, $R \in \mathbb{S}^1$, we consider Sobolev maps $Q \in W^{1,2}(\Omega; \mathcal{S}_0)$ satisfying the equivariance property
\begin{equation}
\label{equivariance}
 Q(Rx)=R Q(x) R^t \qquad \forall R\in \mathbb{S}^1 \, .
 \end{equation}
Minimizing the energy functional \eqref{LDGenergytilde} or \eqref{LDGenergy} in the appropriate class of equivariant configurations will provide minimizers which are either smooth and nowhere vanishing, or with singularities/isotropic points, depending on the geometry of the domain and on the chosen boundary data. In case such defects are not present, we will be able to show that the level sets of the signed biaxiality parameter are generically finite union of axially symmetric tori. On the other hand, when singularieties/isotropic points occur, the regularity/absence of isotropic phase results of the present paper will show that axial symmetry of minimizers is not inherited from the boundary condition, and axial symmetry breaking and nonuniqueness phenomena must occur. Such phenomena were already proved in \cite{AlLi} for minimizers of the Frank-Oseen energy, and our results are the natural counterpart for the Landau-de Gennes model, in agreement with the numerical simulations in \cite{DLR}. 
\end{remark}
\vskip10pt

%%%%%%%%%%%%%%%%%%%%%%%%%%%%%%%%%%%%%%%%%%%%%%%%%%%%%%%

\noindent \textbf{Acknowledgements.} We would like to warmly thank Eugene Gartland for his comments and suggestions, and for pointing out reference \cite{Gart}. We would also like to thank the anonymous referees for their comments leading to a strong improvement in the presentation of this article. The whole project started some years ago while A.P. was visiting V.M. at the DMA in \'Ecole Normale Sup\'erieure de Paris. He thanks the DMA for the invitation and for the warm hospitality.

%%%%%%%%%%%%%%%%%%%%%%%%%%%%%%%%%%%%%%%%%%%%%%%%%%%%%%%
%%%%%%%%%%%%%%%%%%%%%%%%%%%%%%%%%%%%%%%%%%%%%%%%%%%%%%%

\section{Small energy regularity theory: a tool box}
\label{sec:3}

%%%%%%%%%%%%%%%%%%%%%%%%%%%%%%%%%%%%%%%%%%%%%%%%%%%%%%%
%%%%%%%%%%%%%%%%%%%%%%%%%%%%%%%%%%%%%%%%%%%%%%%%%%%%%%%

The aim of this section is to provide several regularity estimates, both in the interior and at the boundary,  for weak solutions of \eqref{MasterEq} under certain general conditions.  We emphasize that the material developed here is not restricted to  minimizers of the energy functional $\mathcal{E}_\lambda$, but it applies to rather general critical points satisfying suitable {\sl energy monotonicity formulae}.  With this respect, we shall make a crucial use of the results of this section in our companion papers \cite{DMP1,DMP2} where we considered solutions obtained by  minimization of  $\mathcal{E}_\lambda$ in restricted (symmetric) classes. 
\vskip3pt 
Before going further, let us  precise for completeness the (usual) notion of critical point of $\mathcal{E}_\lambda$ over the nonlinear space $W^{1,2}(\Omega;\mathbb{S}^4)$, and show that critical points are exactly the distributional solutions of \eqref{MasterEq} belonging to $W^{1,2}(\Omega;\mathbb{S}^4)$. 

\begin{definition}\label{defcritpt}
A map $Q_\lambda\in W^{1,2}(\Omega;\mathbb{S}^4)$ is said to be a critical point of $\mathcal{E}_\lambda$ if 
$$\left[\frac{d}{dt}\mathcal{E}_\lambda\left(\frac{Q_\lambda+t\Phi}{|Q_\lambda +t\Phi|}\right)\right]_{t=0}=0 $$
for every $\Phi\in C^1_c(\Omega;\mathcal{S}_0)$. 
\end{definition}

The Euler-Lagrange equation for critical points of $\mathcal{E}_\lambda$ reads as follows. 

\begin{proposition}\label{propstartvarEL}
A map $Q_\lambda\in W^{1,2}(\Omega;\mathbb{S}^4)$ is a critical point of $\mathcal{E}_\lambda$ if and only if  
\begin{equation}\label{varELeqtangfields}
\int_{\Omega} \nabla Q_\lambda : \nabla\Phi\,dx=\lambda\int_\Omega Q_\lambda^2:\Phi\,dx
\end{equation}
for every $\Phi\in W^{1,2}(\Omega;{Q_\lambda}^{\!*}T\mathbb{S}^4)$ compactly supported in $\Omega$ (i.e., for every $\Phi\in W^{1,2}(\Omega;\mathcal{S}_0)$ compactly supported in $\Omega$ and satisfying $\Phi(x)\in T_{Q_\lambda(x)}\mathbb{S}^4$ for a.e. $x\in\Omega$), or equivalently, if and only if 
\begin{equation}\label{distribELeq}
-\Delta Q_\lambda=|\nabla Q_\lambda|^2Q_\lambda+\lambda\Big(Q_\lambda^2-\frac{1}{3}I-{\rm tr}(Q^3_\lambda)Q_\lambda\Big) \quad\text{in $\mathscr{D}^\prime(\Omega)$}\,.
\end{equation}
\end{proposition}

\begin{proof}
{\it Step 1.} Given $Q\in W^{1,2}(\Omega;\mathbb{S}^4)$, let us consider $\Phi\in C^1_c(\Omega;\mathcal{S}_0)$, and set for $t$ small enough, 
$$Q_t:=\frac{Q+t\Phi}{|Q +t\Phi|}\in W^{1,2}(\Omega;\mathcal{S}_0)\,.$$
Classically (see e.g. \cite[Section 2.2]{Simon}), we have 
$$\left[\frac{d}{dt}\int_\Omega \frac{1}{2}|\nabla Q_t|^2\,dx\right]_{t=0}= \int_{\Omega}\big(\nabla Q:\nabla\Phi-|\nabla Q|^2Q:\Phi\big)\,dx\,.$$
On the other hand, a straightforward computations yields 
$$\left[\frac{d}{dt}\int_\Omega W(Q_t)\,dx\right]_{t=0}= -\int_{\Omega}\big(Q^2:\Phi-{\rm tr}(Q^3)Q:\Phi\big)\,dx\,,$$
and thus 
\begin{multline}\label{firstoutvar}
\left[\frac{d}{dt}\mathcal{E}_\lambda\left(\frac{Q+t\Phi}{|Q +t\Phi|}\right)\right]_{t=0}= \int_{\Omega}\nabla Q:\nabla\Phi\,dx-\int_{\Omega}|\nabla Q|^2Q:\Phi\,dx\\
-\lambda\int_{\Omega}\big(Q^2:\Phi-{\rm tr}(Q^3)Q:\Phi\big)\,dx\,.
\end{multline}

\noindent{\it Step 2.} Assume that $Q_\lambda\in W^{1,2}(\Omega;\mathbb{S}^4)$ is a critical point of 
$\mathcal{E}_\lambda$. We  consider $\Phi\in W^{1,2}(\Omega;{Q_\lambda}^{\!*}T\mathbb{S}^4)$ compactly supported in $\Omega$, and prove that \eqref{varELeqtangfields} holds. By a standard truncation argument, we can assume that $\Phi\in L^\infty(\Omega)$. By a usual approximation argument, we can find a sequence $\{\Phi_k\}\subset C^1_c(\Omega;\mathcal{S}_0)$ such that $\Phi_k\to \Phi$ a.e. in $\Omega$ and strongly in $W^{1,2}(\Omega)$, and satisfying $\|\Phi_k\|_{L^\infty(\Omega)}\leq \|\Phi\|_{L^\infty(\Omega)}$. Then we deduce from Step 1 and the criticality of $Q_\lambda$ that
\begin{equation}\label{alongphik}
\int_{\Omega}\nabla Q_\lambda:\nabla\Phi_k\,dx=  \int_{\Omega}|\nabla Q_\lambda|^2Q_\lambda:\Phi_k\,dx+\lambda\int_{\Omega}\big(Q^2_\lambda:\Phi_k-{\rm tr}(Q^3_\lambda)Q_\lambda:\Phi_k\big)\,dx \,.
\end{equation}
Since $Q_\lambda:\Phi=0$ a.e. in $\Omega$, we deduce by dominated convergence that $|\nabla Q_\lambda|^2Q_\lambda:\Phi_k\to 0$ and $(Q^2_\lambda:\Phi_k-{\rm tr}(Q^3_\lambda)Q_\lambda:\Phi_k)\to Q^2_\lambda:\Phi$ in $L^1(\Omega)$. Hence, letting $k\to\infty$ in \eqref{alongphik} leads to \eqref{varELeqtangfields}. 
\vskip5pt

\noindent{\it Step 3.} Assume  that  $Q_\lambda\in W^{1,2}(\Omega;\mathbb{S}^4)$ satisfies  \eqref{varELeqtangfields}, and fix an arbitrary $\Phi\in C^1_c(\Omega;\mathscr{M}^{\rm sym}_{3\times 3}(\R))$. Define $\Phi_0:=\Phi-\frac{1}{3}(\Phi:I)I\in C^1_c(\Omega;\mathcal{S}_0)$. Noticing that 
$$\Phi_0-(Q_\lambda:\Phi_0)Q_\lambda\in W^{1,2}_0(\Omega;{Q_\lambda}^{\!*}T\mathbb{S}^4)\,,$$
we infer from  \eqref{varELeqtangfields} that 
\begin{equation}\label{cpadutou}
\int_{\Omega}\nabla Q_\lambda:\nabla\Phi_0\,dx=  \int_{\Omega}|\nabla Q_\lambda|^2Q_\lambda:\Phi_0\,dx+\lambda\int_{\Omega}\big(Q^2_\lambda:\Phi_0-{\rm tr}(Q^3_\lambda)Q_\lambda:\Phi_0\big)\,dx\,.
\end{equation}
Since $Q_\lambda:I={\rm tr}(Q_\lambda)=0$ and $|Q_\lambda|^2={\rm tr}(Q_\lambda^2)=1$, this last identity leads to 
$$\int_{\Omega}\nabla Q_\lambda:\nabla\Phi\,dx=  \int_{\Omega}|\nabla Q_\lambda|^2Q_\lambda:\Phi\,dx+\lambda\int_{\Omega}\Big(Q^2_\lambda-\frac{1}{3}I -{\rm tr}(Q^3_\lambda)Q_\lambda\Big):\Phi\,dx\,,$$
and \eqref{distribELeq} follows. 
\vskip5pt

\noindent{\it Step 4.} Finally, if $Q_\lambda\in W^{1,2}(\Omega;\mathbb{S}^4)$ satisfies  \eqref{distribELeq}, then \eqref{cpadutou} holds for every $\Phi_0\in C^1_c(\Omega;\mathcal{S}_0)$. 
 In  view of \eqref{firstoutvar}, it implies that $Q_\lambda$ is indeed a critical point of $\mathcal{E}_\lambda$. 
\end{proof}

\begin{remark}
If a map $Q_\lambda\in W^{1,2}(\Omega;\mathbb{S}^4)$ is a minimizer of $\mathcal{E}_\lambda$ among all $Q\in W^{1,2}(\Omega;\mathbb{S}^4)$ such that $Q-Q_\lambda$ is compactly supported in $\Omega$,  then $Q_\lambda$ is a critical point of $\mathcal{E}_\lambda$ by the first order condition for minimality. In particular, if $Q_\lambda$ is minimizing $\mathcal{E}_\lambda$ over $\mathcal{A}_{Q_{\rm b}}(\Omega)$, then $Q_\lambda$ satisfies \eqref{distribELeq} (or equivalently \eqref{varELeqtangfields}).
\end{remark}

%%%%%%%%%%%%%%%%%%%%%%%%%%%%%%%%%%%%%%%%%%%%%%%%%%%%%%%

\subsection{Monotonicity formulae for approximable critical points}\label{sec:eps-reg}
In this  subsection, our goal is (essentially) to derive the afore mentioned monotonicity formulae for certain critical points of $\mathcal{E}_\lambda$. 
Concerning  minimizers, such formulae can be classically obtained by inner variations of the energy. However this argument can not be used when considering energy minimizers over symmetric classes as we do in \cite{DMP1,DMP2}. To circumvent this difficulty, we consider critical points of $\mathcal{E}_\lambda$ which can be (strongly) approximated by critical points of a suitable Ginzburg-Landau functional in which the constraint to be $\mathbb{S}^4$-valued is relaxed. In this way, the approximate solution is smooth enough to derive the monotonicity formulae from the Euler-Lagrange equation, and we conclude by taking the limit in the approximation parameter. This procedure applies of course to minimizers (as we shall see in Section~\ref{sec:regularity}), but also to the symmetric solutions of  \eqref{MasterEq} considered in \cite{DMP1,DMP2}. Let us now describe it in details.
\vskip3pt

Given a bounded open set $\Omega\subset\R^3$, a reference map $Q_{\rm ref}\in \mathcal{A}_{Q_{\rm b}}(\Omega)$ and a small parameter $\varepsilon\in(0,\lambda^{-1/2})$, we consider the energy functional $\mathcal{GL}_{\varepsilon}(Q_{\rm ref}; \cdot)$ defined over $W^{1,2}(\Omega;\mathcal{S}_0)$ by 
\begin{equation}\label{defGLepsQref}
\mathcal{GL}_{\varepsilon}(Q_{\rm ref}; Q):= \mathcal{E}_\lambda(Q)+\frac{1}{4\varepsilon^2}\int_\Omega(1-|Q|^2)^2\,dx+\frac{1}{2}\int_\Omega|Q-Q_{\rm ref}|^2\,dx\,.
\end{equation}
If $Q_{\rm ref}$ can be achieved as a (strong) limit of critical points of  $\mathcal{GL}_{\varepsilon}(Q_{\rm ref};\cdot)$ when $\eps\to0$, then $Q_{\rm ref}$ satisfies the monotonicity formulae  stated in Proposition \ref{monotonprop} below. Its proof involves of course some very classical computations, see e.g. \cite{ChenStruwe},  as implemented in \cite{MaZa}  for the minimizers of the energy functional \eqref{LDGenergytildeori} without norm constraint. Here, the computations follow closely \cite{MaZa}, but they also provide some explicit dependence of the constants with respect to the data, a property which will be used in the subsequent papers \cite{DMP1,DMP2}.

\begin{proposition}\label{monotonprop}
Assume that $\partial\Omega$ is of class $C^{3}$ and $Q_{\rm b}\in C^{1,1}(\partial\Omega;\mathbb{S}^4)$. Let $Q_{\rm ref}\in \mathcal{A}_{Q_{\rm b}}(\Omega)$. For each $\varepsilon>0$, let $Q_\varepsilon\in W^{1,2}_{\rm Q_{\rm b}}(\Omega;\mathcal{S}_0)$
be a critical point of the functional $\mathcal{GL}_{\varepsilon}(Q_{\rm ref}; \cdot)$. If 
\begin{equation}\label{condconvmonotform}
Q_\varepsilon\mathop{\longrightarrow}\limits_{\varepsilon\to 0} Q_{\rm ref} \text{ in $L^2(\Omega)\,$, and }\;\mathcal{GL}_{\varepsilon}(Q_{\rm ref}; Q_\varepsilon)\mathop{\longrightarrow} \limits_{\varepsilon\to 0}\mathcal{E}_\lambda(Q_{\rm ref})\,,
\end{equation}
then $Q_{\rm ref}$ satisfies
\begin{enumerate}
\item[\rm 1)]  the {\sl Interior Monotonicity Formula}:
\begin{multline}\label{IntMonForm}
\frac{1}{r}\mathcal{E}_\lambda(Q_{\rm ref},B_r(x_0)) -\frac{1}{\rho}\mathcal{E}_\lambda(Q_{\rm ref},B_\rho(x_0))=\\
\int_{B_r(x_0)\setminus B_\rho(x_0)} \frac{1}{|x-x_0|}\bigg|\frac{\partial Q_{\rm ref}}{\partial |x-x_0|}\bigg|^2\,dx
+ 2\lambda\int_\rho^r\bigg(\frac{1}{t^2}\int_{B_t(x_0)}W(Q_{\rm ref})\,dx\bigg)\,dt
\end{multline}
for every $x_0\in\Omega$ and every $0<\rho<r\leq{\rm dist}(x_0,\partial\Omega)$;
\vskip5pt

\item[\rm 2)]  the {\sl Boundary Monotonicity Inequality}: there exist two constants $C_\Omega>0$ and ${\bf r}_\Omega>0$ (depending only on $\Omega$) such that 
\begin{multline}\label{BdMonForm}
\frac{1}{r}\mathcal{E}_\lambda(Q_{\rm ref},B_r(x_0)\cap\Omega) -\frac{1}{\rho}\mathcal{E}_\lambda(Q_{\rm ref},B_\rho(x_0)\cap\Omega)\geq -(r-\rho)K_\lambda(Q_{\rm b},Q_{\rm ref})\\
+\int_{\big(B_r(x_0)\setminus B_\rho(x_0)\big)\cap\Omega} \frac{1}{|x-x_0|}\bigg|\frac{\partial Q_{\rm ref}}{\partial |x-x_0|}\bigg|^2\,dx
+ 2\lambda\int_\rho^r\bigg(\frac{1}{t^2}\int_{B_t(x_0)\cap\Omega}W(Q_{\rm ref})\,dx\bigg)\,dt
\end{multline}
for every $x_0\in\partial\Omega$ and every $0<\rho<r<{\bf r}_\Omega$, where 
$$K_\lambda(Q_{\rm b},Q_{\rm ref}):= C_\Omega\bigg(\|\nabla_{\rm tan} Q_{\rm b}\|^2_{L^\infty(\partial\Omega)}+\lambda\|W(Q_{\rm b})\|_{L^1(\partial\Omega)}+\|\nabla Q_{\rm ref}\|^2_{L^2(\Omega)}\bigg)  \,.$$
\end{enumerate}
\end{proposition}

\begin{proof}
{\it Step 1:   Euler-Lagrange equation, regularity, and convergence.} Since $Q_\eps$ is a critical point of $\mathcal{GL}_{\varepsilon}(Q_{\rm ref}; \cdot)$ over $W^{1,2}_{\rm Q_{\rm b}}(\Omega;\mathcal{S}_0)$, it satisfies the Euler-Lagrange equation 
\begin{equation}\label{ELeqGL}
\begin{cases}
\displaystyle-\Delta Q_\varepsilon=\lambda\bigg(Q_\varepsilon^2-\frac{1}{3}|Q_\varepsilon|^2I-\frac{1}{\sqrt{6}}|Q_\eps|^2Q_\eps\bigg)  +\frac{1}{\varepsilon^2}(1-|Q_\varepsilon|^2)Q_\varepsilon -(Q_\varepsilon-Q_{\rm ref}) & \text{in $\Omega$}\,,\\
Q_\varepsilon=Q_{\rm b} & \text{on $\partial\Omega$}\,.
\end{cases}
\end{equation}
This equation can be easily derived from outer variations noticing that the term $\frac{1}{3}|Q_\varepsilon|^2I$ corresponds to the Lagrange multiplier associated with the traceless constraint and using the expression \eqref{Newpotential2} for the potential $W$. 
By the Sobolev embedding $W^{1,2}(\Omega)\hookrightarrow L^6(\Omega)$, we have $Q_\eps\in L^6(\Omega)$, which implies that $\Delta Q_\eps\in L^2(\Omega)$. Note that the regularity assumption on $Q_{\rm b}$ and $\partial\Omega$  ensures that $Q_{\rm b}$ admits a $C^{1,1}$ extension (with values in $\mathcal{S}_0$) to the whole domain $\Omega$ (see the material in Subsection \ref{subsecextension}).  
By elliptic regularity, we thus have $Q_\eps\in W^{2,2}(\Omega)$, see e.g. \cite[Theorem 8.12]{GilbTrud}. In particular, $Q_\eps\in W^{1,6}(\Omega)$ and thus $Q_\eps\in L^\infty(\Omega)$ by the Sobolev embedding $W^{1,6}(\Omega)\hookrightarrow L^\infty(\Omega)$. Hence, $\Delta Q_\eps\in L^\infty(\Omega)$, and by elliptic regularity again, we have $Q_\eps\in C^{1,\alpha}(\overline\Omega)$ for every $\alpha\in(0,1)$, see e.g. \cite[Theorem 8.34]{GilbTrud}. 

We now claim that assumption \eqref{condconvmonotform} implies that 
$$Q_\varepsilon\mathop{\longrightarrow}\limits_{\varepsilon\to 0} Q_{\rm ref} \text{ strongly in $W^{1,2}(\Omega)\,$, and }\; \frac{1}{\eps^2}\int_{\Omega}(1-|Q_\eps|^2)^2\,dx\mathop{\longrightarrow}\limits_{\varepsilon\to 0} 0\,.$$
Indeed, we first infer from \eqref{condconvmonotform} that $\{Q_\eps\}_{\eps>0 }$ remains bounded in $W^{1,2}(\Omega)$ as $\eps\to0$.  Therefore, given an arbitrary sequence $\eps_n\to 0$, we have $Q_{\eps_n}\rightharpoonup Q_{\rm ref}$ weakly in $W^{1,2}(\Omega)$.  In particular, $Q_{\eps_n}\to Q_{\rm ref}$ in $L^4(\Omega)$ by the compact Sobolev embedding $W^{1,2}(\Omega)\hookrightarrow L^4(\Omega)$. As a consequence, $\int_\Omega W(Q_{\eps_n})\,dx\to \int_\Omega W(Q_{\rm ref})\,dx$. On the other hand, by \eqref{condconvmonotform} and lower semi-continuity of the Dirichlet integral, we have 
\begin{multline*}
\mathcal{E}_{\lambda}(Q_{\rm ref})\leq \liminf_{n\to\infty} \mathcal{E}_{\lambda}(Q_{\eps_n})\leq \limsup_{n\to\infty} \mathcal{E}_{\lambda}(Q_{\eps_n}) \\
\leq \lim_{n\to\infty}\Big(\mathcal{E}_{\lambda}(Q_{\eps_n})+\frac{1}{4\eps_n^2}\int_{\Omega}(1-|Q_{\eps_n}|^2)^2\,dx\Big)=\mathcal{E}_{\lambda}(Q_{\rm ref})\,.
\end{multline*}
Hence $\frac{1}{\eps_n^2}\int_\Omega (1-|Q_{\eps_n}|^2)^2\,dx\to 0$, and $\|\nabla Q_{\eps_n}\|_{L^2(\Omega)}\to \|\nabla Q_{\rm ref}\|_{L^2(\Omega)}$. This latter fact, combined with the $W^{1,2}$-weak convergence, implies the  $W^{1,2}$-strong convergence of $Q_{\eps_n}$ toward  $Q_{\rm ref}$. 
\vskip5pt

\noindent{\it Step 2: Interior Monotonicity Formula.} Without loss of generality, we may assume that $x_0=0$. Let us take the inner product of  \eqref{ELeqGL} with $(x\cdot\nabla)Q_\eps$, and integrate by parts over the ball $B_t$ of radius $t\in(\rho,r)$. It yields
\begin{multline*}
-\frac{1}{2}\int_{B_t}|\nabla Q_\eps|^2\,dx+\frac{t}{2}\int_{\partial B_t}|\nabla Q_\eps|^2\,d\mathcal{H}^2-\lambda\int_{B_t}W(Q_\eps)\,dx+\lambda t\int_{\partial B_t}W(Q_\eps)\,d\mathcal{H}^2\\
-\frac{1}{4\eps^2}\int_{B_t}(1-|Q_\eps|^2)^2\,dx+\frac{t}{4\eps^2}\int_{\partial B_t}(1-|Q_\eps|^2)^2\,d\mathcal{H}^2-\frac{1}{2}\int_{B_t}|Q_\eps-Q_{\rm ref}|^2\,dx+\frac{t}{2}\int_{\partial B_t}|Q_\eps-Q_{\rm ref}|^2\,d\mathcal{H}^2\\
=t\int_{\partial B_t}\Big|\frac{\partial Q_\eps}{\partial |x|}\Big|^2\,d\mathcal{H}^2+2\lambda\int_{B_t}W(Q_\eps)\,dx+\frac{1}{2\eps^2}\int_{B_t}(1-|Q_\eps|^2)^2\,dx\\
+\int_{B_t}|Q_\eps-Q_{\rm ref}|^2\,dx-\int_{B_t}(Q_\eps-Q_{\rm ref}):\big((x\cdot\nabla)Q_{\rm ref}\big)\,dx\,.  
\end{multline*}
Dividing both sides by $t^2$, we obtain
\begin{multline*}
\frac{d}{dt}\bigg(\frac{1}{t}\mathcal{GL}_{\varepsilon}(Q_{\rm ref}; Q_\eps,B_t)\bigg)=\frac{1}{t}\int_{\partial B_t}\Big|\frac{\partial Q_\eps}{\partial |x|}\Big|^2\,d\mathcal{H}^2+
\frac{2\lambda}{t^2}\int_{B_t}W(Q_\eps)\,dx\\
+\frac{1}{2\eps^2t^2}\int_{B_t}(1-|Q_\eps|^2)^2\,dx+\frac{1}{t^2}\int_{B_t}|Q_\eps-Q_{\rm ref}|^2\,dx-\frac{1}{t^2}\int_{B_t}(Q_\eps-Q_{\rm ref}):\big((x\cdot\nabla)Q_{\rm ref}\big)\,dx\,.
\end{multline*}
Integrating this identity between $\rho$ and $r$ yields
\begin{multline*}
\frac{1}{r}\mathcal{GL}_{\varepsilon}(Q_{\rm ref}; Q_\eps,B_r)-\frac{1}{\rho}\mathcal{GL}_{\varepsilon}(Q_{\rm ref}; Q_\eps,B_\rho)=\int_{B_r\setminus B_\rho}\frac{1}{|x|}\Big|\frac{\partial Q_\eps}{\partial |x|}\Big|^2\,dx\\
+2\lambda\int_\rho^r\bigg(\frac{1}{t^2}\int_{B_t}W(Q_{\eps})\,dx\bigg)\,dt+\frac{1}{2\eps^2}\int_\rho^r\bigg(\frac{1}{t^2}\int_{B_t}(1-|Q_\eps|^2)^2\,dx\bigg)\,dt\\
+\int_\rho^r\bigg(\frac{1}{t^2}\int_{B_t}|Q_\eps-Q_{\rm ref}|^2\,dx\bigg)\,dt-\int_\rho^r\bigg(\frac{1}{t^2}\int_{B_t}(Q_\eps-Q_{\rm ref}):\big((x\cdot\nabla)Q_{\rm ref}\big)\,dx\bigg)\,dt\,.
\end{multline*}
In view of the convergences established in Step 1, letting $\eps\to0$ in this last identity leads to  \eqref{IntMonForm}. 
\vskip5pt

\noindent{\it Step 3: Boundary Monotonicity Inequality.} We first claim that there exists a constant $C_\Omega>0$ depending only on $\Omega$ such that
\begin{multline}\label{estnormderGL}
\int_{\partial\Omega}\Big|\frac{\partial Q_\eps}{\partial\nu}\Big|^2\,d\mathcal{H}^2\leq C_\Omega\bigg(\|\nabla_{\rm tan}Q_{\rm b}\|^2_{L^2(\partial\Omega)} +\lambda\|W(Q_{\rm b})\|_{L^1(\partial\Omega)} \\
+\|\nabla Q_\eps\|_{L^2(\Omega)}^2+\|\nabla Q_{\rm ref}\|_{L^2(\Omega)}^2+\|Q_\eps-Q_{\rm ref}\|_{L^2(\Omega)}^2\bigg)\,.
\end{multline}
To prove this estimate, let us introduce $\Phi_\Omega\in C^{2,\alpha}(\overline\Omega)$  the unique solution of
$$\begin{cases}
-\Delta\Phi_\Omega=1 & \text{in $\Omega$}\,,\\
\Phi_\Omega=0 & \text{on $\partial\Omega$}\,,
\end{cases}
$$
see e.g. \cite[Theorem 6.14]{GilbTrud}.  
We consider $V:\overline\Omega\to\R^3$ the $C^{1,\alpha}$-vector field given by $V:=-\nabla\Phi_\Omega$. Note that $V=(V\cdot\nu)\nu$ on~$\partial\Omega$ (since $\Phi_\Omega$ is constant on $\partial\Omega$), where $\nu$ is outer unit normal on $\partial \Omega$. Taking the inner product 
of \eqref{ELeqGL} with $(V\cdot\nabla)Q_\eps$, and integrating by parts over $\Omega$ leads to 
\begin{multline*}
\frac{1}{2}\int_{\partial\Omega}\Big|\frac{\partial Q_\eps}{\partial\nu}\Big|^2(V\cdot\nu)\,d\mathcal{H}^2\\
+\int_\Omega\Big( \frac{1}{2}|\nabla Q_\eps|^2+\lambda W(Q_\eps)+\frac{1}{4\eps^2}(1-|Q_\eps|^2)^2+\frac{1}{2}|Q_\eps-Q_{\rm ref}|^2\Big){\rm div}(V)\,dx\\
=\frac{1}{2}\int_{\partial\Omega}|\nabla_{\rm tan} Q_{\rm b}|^2(V\cdot\nu)\,d\mathcal{H}^2+\lambda\int_{\partial\Omega}W(Q_{\rm b})(V\cdot\nu)\,d\mathcal{H}^2\\+\int_\Omega \sum_{i,j=1}^3(\partial_iQ_\eps:\partial_jQ_\eps)\partial_jV_i\,dx
+\int_\Omega(Q_\eps-Q_{\rm ref}):(V\cdot\nabla)Q_{\rm ref}\,dx\,,
\end{multline*}
since $Q_\eps=Q_{\rm ref}=Q_{\rm b}$ on $\partial\Omega$ and $|Q_{\rm b}|=1$. Using ${\rm div}(V)=1$ in $\Omega$, we deduce that  
\begin{multline*}
\int_{\partial\Omega}\Big|\frac{\partial Q_\eps}{\partial\nu}\Big|^2(V\cdot\nu)\,d\mathcal{H}^2\leq C \|V\|_{C^1(\Omega)}\bigg(\|\nabla_{\rm tan}Q_{\rm b}\|^2_{L^2(\partial\Omega)} +\lambda\|W(Q_{\rm b})\|_{L^1(\partial\Omega)}+\|\nabla Q_\eps\|_{L^2(\Omega)}^2\\
+\|\nabla Q_{\rm ref}\|_{L^2(\Omega)}^2+\|Q_\eps-Q_{\rm ref}\|_{L^2(\Omega)}^2\bigg)\,,
\end{multline*}
for some universal constant $C>0$. On the other hand, by the Hopf lemma, there is a constant ${\bf c}^0_\Omega>0$ depending only on $\Omega$ such that $V\cdot\nu\geq {\bf c}^0_\Omega$ on~$\partial\Omega$, and \eqref{estnormderGL} follows. 
\vskip3pt

We now fix $x_0\in \partial\Omega$. By the smoothness assumption on $\partial\Omega$, there are two constants ${\bf r}_\Omega>0$ and ${\bf c}^1_\Omega>0$ (depending only $\Omega$) such that for every $t\in(0,{\bf r}_\Omega)$, 
\begin{equation}\label{propregbdomega}
\mathcal{H}^2\big(B_t(x_0)\cap\partial \Omega\big)\leq {\bf c}^1_\Omega t^2\,, \text{ and }  \big|(x-x_0)\cdot\nu(x)\big|\leq {\bf c}^1_\Omega t^2\text{ on $B_t(x_0)\cap\partial\Omega$}\,.  
\end{equation}
In what follows, we assume without loss of generality that $x_0=0$. Let us fix $0<\rho<r<{\bf r}_\Omega$. Taking once again the inner product of \eqref{ELeqGL} with $(x\cdot\nabla)Q_\eps$, we integrate the result by parts in $B_t\cap\Omega$ with $t\in(\rho,r)$. Similarly to Step 2, it yields (after dividing by $t^2$)
\begin{multline}\label{pohozbdident}
\frac{d}{dt}\bigg(\frac{1}{t}\mathcal{GL}_{\varepsilon}(Q_{\rm ref}; Q_\eps,B_t\cap\Omega)\bigg)=\frac{1}{t}\int_{\Omega\cap\partial B_t}\Big|\frac{\partial Q_\eps}{\partial |x|}\Big|^2\,d\mathcal{H}^2+
\frac{2\lambda}{t^2}\int_{\Omega\cap B_t}W(Q_\eps)\,dx\\
+\frac{1}{2\eps^2t^2}\int_{B_t}(1-|Q_\eps|^2)^2\,dx+\frac{1}{t^2}\int_{\Omega\cap B_t}|Q_\eps-Q_{\rm ref}|^2\,dx
-\frac{1}{t^2}\int_{\Omega\cap B_t}(Q_\eps-Q_{\rm ref}):\big((x\cdot\nabla)Q_{\rm ref}\big)\,dx\\
-\frac{1}{2t^2}\int_{B_t\cap\partial\Omega}|\nabla Q_\eps|^2(x\cdot\nu)\,d\mathcal{H}^2 +\frac{1}{t^2}\int_{B_t\cap\partial\Omega}\frac{\partial Q_\eps}{\partial\nu}:\big((x\cdot\nabla)Q_\eps\big)\,d\mathcal{H}^2\\
-\frac{\lambda}{t^2}\int_{B_t\cap\partial\Omega}W(Q_{\rm b})(x\cdot\nu)\,d\mathcal{H}^2\,. 
\end{multline}
Note that we used once again $Q_\eps=Q_{\rm ref}=Q_{\rm b}$ on $\partial\Omega$,  and $|Q_{\rm b}|=1$.  Next, if we denote by $(\tau_1,\tau_2)$ an orthonormal basis of the tangent space of $\partial\Omega$ at $x$, we have 
\begin{multline*}
-\frac{1}{2}\int_{B_t\cap\partial\Omega}|\nabla Q_\eps|^2(x\cdot\nu)\,d\mathcal{H}^2+\int_{B_t\cap\partial\Omega}\frac{\partial Q_\eps}{\partial\nu}:\big((x\cdot\nabla)Q_\eps\big)\,d\mathcal{H}^2= 
\frac{1}{2}\int_{B_t\cap\partial\Omega}\Big|\frac{\partial Q_\eps}{\partial\nu}\Big|^2(x\cdot\nu)\,d\mathcal{H}^2\\
 -\frac{1}{2}\int_{B_t\cap\partial\Omega}\Big|\frac{\partial Q_{\rm b}}{\partial\tau_1}\Big|^2(x\cdot\nu)\,d\mathcal{H}^2 -\frac{1}{2}\int_{B_t\cap\partial\Omega}\Big|\frac{\partial Q_{\rm b}}{\partial\tau_2}\Big|^2(x\cdot\nu)\,d\mathcal{H}^2\\
 +\int_{B_t\cap\partial\Omega}\frac{\partial Q_\eps}{\partial\nu}:\frac{\partial Q_{\rm b}}{\partial\tau_1}(x\cdot\tau_1)\,d\mathcal{H}^2
 +\int_{B_t\cap\partial\Omega}\frac{\partial Q_\eps}{\partial\nu}:\frac{\partial Q_{\rm b}}{\partial\tau_2}(x\cdot\tau_2)\,d\mathcal{H}^2\,.
\end{multline*}
Then we infer from \eqref{propregbdomega} that 
\begin{multline}\label{estibdpohoz}
-\frac{1}{2}\int_{B_t\cap\partial\Omega}|\nabla Q_\eps|^2(x\cdot\nu)\,d\mathcal{H}^2+\int_{B_t\cap\partial\Omega}\frac{\partial Q_\eps}{\partial\nu}:\big((x\cdot\nabla)Q_\eps\big)\,d\mathcal{H}^2
\geq\\ 
- C_\Omega t^2\big(\|\partial_\nu Q_\eps\|^2_{L^2(\partial\Omega)}+\|\nabla_{\rm tan}Q_{\rm b}\|^2_{L^\infty(\partial\Omega)}\big)\,,
\end{multline}
for a constant $C_\Omega>0$ depending only on the constants ${\bf r}_\Omega$ and ${\bf c}^1_\Omega$. Still by \eqref{propregbdomega}, we have 
\begin{equation}\label{addpotenttermonbdry}
\int_{B_t\cap\partial\Omega}W(Q_{\rm b})(x\cdot\nu)\,d\mathcal{H}^2\leq C_\Omega t^2 \int_{\partial\Omega}W(Q_{\rm b})\,d\mathcal{H}^2\,.
\end{equation}
Inserting \eqref{estibdpohoz}, \eqref{addpotenttermonbdry}, and \eqref{estnormderGL} in \eqref{pohozbdident}, and integrating the resulting inequality between $\rho$ and $r$ yields  
\begin{multline*}
\frac{1}{r}\mathcal{GL}_{\varepsilon}(Q_{\rm ref}; Q_\eps,B_r\cap\Omega) -\frac{1}{\rho}\mathcal{GL}_{\varepsilon}(Q_{\rm ref}; Q_\eps,B_\rho\cap\Omega)
\geq -(r-\rho)\widetilde K_\lambda(Q_{\rm b},Q_{\rm ref},Q_\eps)\\
+\int_{(B_r\setminus B_\rho)\cap\Omega}\frac{1}{|x|}\Big|\frac{\partial Q_\eps}{\partial |x|}\Big|^2\,dx+2\lambda\int_\rho^r\bigg(\frac{1}{t^2}\int_{B_t\cap\Omega}W(Q_{\eps})\,dx\bigg)\,dt\\
+\frac{1}{2\eps^2}\int_\rho^r\bigg(\frac{1}{t^2}\int_{B_t\cap\Omega}(1-|Q_\eps|^2)^2\,dx\bigg)\,dt+\int_\rho^r\bigg(\frac{1}{t^2}\int_{B_t\cap\Omega}|Q_\eps-Q_{\rm ref}|^2\,dx\bigg)\,dt\\
-\int_\rho^r\bigg(\frac{1}{t^2}\int_{B_t\cap\Omega}(Q_\eps-Q_{\rm ref}):\big((x\cdot\nabla)Q_{\rm ref}\big)\,dx\bigg)\,dt\,,
\end{multline*}
where 
\begin{multline*}
\widetilde K_\lambda(Q_{\rm b},Q_{\rm ref},Q_\eps):=C_\Omega\bigg(\|\nabla_{\rm tan}Q_{\rm b}\|^2_{L^\infty(\partial\Omega)}+\lambda\|W(Q_{\rm b})\|_{L^1(\partial\Omega)}\\
+\|\nabla Q_{\eps}\|^2_{L^2(\Omega)}+\|\nabla Q_{\rm ref}\|^2_{L^2(\Omega)}+\|Q_\eps-Q_{\rm ref}\|^2_{L^2(\Omega)}\bigg)\,,
\end{multline*}
and $C_\Omega>0$ is a constant depending only on ${\bf r}_\Omega$, ${\bf c}^1_\Omega$,  $({{\bf c}^0_\Omega})^{-1}\|\nabla\Phi_\Omega\|_{C^1(\Omega)}$, and the ($2$-dimensional)   measure of~$\partial\Omega$. In view of the convergences established in Step 1, letting $\eps\to0$ in this last inequality leads to~\eqref{BdMonForm}. 
\end{proof}

\begin{remark}[Specific geometry \cite{DMP2}]\label{remmonotbdflatcase}
In our companion paper \cite{DMP2}, we consider a domain $\Omega$ and a boundary condition $Q_{\rm b}$ for which the following situation occurs:  
$0\in\partial\Omega$, $B_1\cap\Omega=B_1\cap\{x_3>0\}$, and $Q_{\rm b}$  is constant on $B_1\cap\partial\Omega=B_1\cap\{x_3=0\}$.
In this situation, the boundary monotonicity inequality \eqref{BdMonForm} for points on  $B_1\cap\partial\Omega$ becomes an equality of the following form: 
for every point $x_0\in B_1\cap\partial\Omega$ and every $0<\rho<r<1-|x_0|$, 
\begin{multline*}
\frac{1}{r}\mathcal{E}_\lambda(Q_{\rm ref},B_r(x_0)\cap\Omega) -\frac{1}{\rho}\mathcal{E}_\lambda(Q_{\rm ref},B_\rho(x_0)\cap\Omega)=\\
\int_{\big(B_r(x_0)\setminus B_\rho(x_0)\big)\cap\Omega} \frac{1}{|x-x_0|}\bigg|\frac{\partial Q_{\rm ref}}{\partial |x-x_0|}\bigg|^2\,dx
+ 2\lambda\int_\rho^r\bigg(\frac{1}{t^2}\int_{B_t(x_0)\cap\Omega}W(Q_{\rm ref})\,dx\bigg)\,dt\,.
\end{multline*}
Indeed, it suffices to notice that $(x-x_0)\cdot\nu=0$ and $\nabla_{\rm tan}Q_{\rm b}=0$ on  $B_1\cap\partial\Omega$, and then use this facts in identity \eqref{pohozbdident}. 
\end{remark}

One of the main consequences of the monotonicity formulae in Proposition \ref{monotonprop} is a uniform control of the energy in small balls. Recalling that $\bar{Q}_{\rm b}\in\mathcal{A}_{Q_{\rm b}}(\Omega)$ is a given $\mathbb{S}^4$-valued extension to the domain $\Omega$ of the boundary condition $Q_{\rm b}$, we have 

\begin{lemma}\label{corolmonotform}
Assume that $\partial\Omega$ is of class $C^{3}$ and $Q_{\rm b}\in C^{1,1}(\partial\Omega;\mathbb{S}^4)$. 
If $Q_{\rm ref}\in\mathcal{A}_{Q_{\rm b}}(\Omega)$ satisfies the monotonicity formulae~\eqref{IntMonForm} and \eqref{BdMonForm} with 
$$K_\lambda(Q_{\rm b},Q_{\rm ref})\leq  C_\Omega\big(\|\nabla_{\rm tan} Q_{\rm b}\|^2_{L^\infty(\partial\Omega)}+\lambda\|W(Q_{\rm b})\|_{L^1(\partial\Omega)}+\mathcal{E}_\lambda(\bar{Q}_{\rm b})\big)$$ 
for some constant $C_\Omega>0$ depending only on $\Omega$, then
\begin{enumerate}
\item for every $x_0\in \Omega$ and $r\in\big(0,{\rm dist}(x_0,\partial\Omega)\big)$, 
$$\mathop{\sup}\limits_{B_\rho(x)\subset B_{r/2}(x_0)} \frac{1}{\rho}\mathcal{E}_\lambda\big(Q_{\rm ref},B_\rho(x)\big)\leq \frac{2}{r}\mathcal{E}_\lambda\big(Q_{\rm ref},B_r(x_0)\big)\,;$$
\item there exist two constants ${\bf r}^{(1)}_{\Omega}>0$ (depending only on $\Omega$) and $C^{\lambda}_{Q_{\rm b}}$ (depending only on $\Omega$, $Q_{\rm b}$, $\lambda\|W(Q_{\rm b})\|_{L^1(\partial\Omega)}$, and $\mathcal{E}_\lambda(\bar{Q}_{\rm b})$) such that for every $x_0\in \partial\Omega$ and $r\in(0,{\bf r}^{(1)}_{\Omega})$, 
\begin{equation}\label{conclcormonotformbdry}
\mathop{\sup}\limits_{B_\rho(x)\subset B_{r/6}(x_0)} \frac{1}{\rho}\mathcal{E}_\lambda\big(Q_{\rm ref},B_\rho(x)\cap\Omega\big)\leq \frac{4}{r}\mathcal{E}_\lambda\big(Q_{\rm ref},B_r(x_0)\cap\Omega\big) +C^{\lambda}_{Q_{\rm b}} r\,.
\end{equation}
\end{enumerate}
\end{lemma}

\begin{proof}
{\it Step 1: proof of (1).} We assume without loss of generality that $x_0=0$, and we consider an arbitrary ball $B_\rho(x)\subset B_{r/2}$. 
By the interior monotonicity formula \eqref{IntMonForm}, we have 
$$ \frac{1}{\rho}\mathcal{E}_\lambda\big(Q_{\rm ref},B_\rho(x)\big)\leq  \frac{1}{\rho+|x|}\mathcal{E}_\lambda\big(Q_{\rm ref},B_{\rho+|x|}(x)\big)\leq  \frac{1}{\rho+|x|}\mathcal{E}_\lambda\big(Q_{\rm ref},B_{2(\rho+|x|)}\big)\leq \frac{2}{r}\mathcal{E}_\lambda\big(Q_{\rm ref},B_r\big)\,,$$
and the claim is proved. 
\vskip5pt

\noindent{\it Step 2: proof of (2).} We choose ${\bf r}^{(1)}_\Omega\in (0,{\bf r}_\Omega)$ (where ${\bf r}_\Omega$ is given by Proposition \ref{monotonprop}) in such a way that the nearest point projection $\boldsymbol{\pi}_\Omega$ on $\partial\Omega$ is well defined in the  ${\bf r}^{(1)}_\Omega$-tubular neighborhood of $\partial\Omega$. Once again, we may assume that $x_0=0$, and we consider  $B_\rho(x)\subset B_{r/6}$.  
We now distinguish different cases. 

Assume first that $x\in\partial \Omega$. Then, we deduce from the boundary monotonicity inequality \eqref{BdMonForm} that 
\begin{multline*}
 \frac{1}{\rho}\mathcal{E}_\lambda\big(Q_{\rm ref},B_\rho(x)\cap\Omega\big)\leq  \frac{1}{\rho+|x|}\mathcal{E}_\lambda\big(Q_{\rm ref},B_{\rho+|x|}(x)\cap\Omega\big)+C^{\lambda}_{Q_{\rm b}} |x| \\
 \leq \frac{1}{\rho+|x|}\mathcal{E}_\lambda\big(Q_{\rm ref},B_{2(\rho+|x|)}\cap\Omega\big)+C^{\lambda}_{Q_{\rm b}} r \leq \frac{2}{r}\mathcal{E}_\lambda\big(Q_{\rm ref},B_{r}\cap\Omega\big)+C^{\lambda}_{Q_{\rm b}} r \,.
\end{multline*} 
Next, for $x\not\in\partial\Omega$ and $|x-\boldsymbol{\pi}_\Omega(x)|\leq\rho$, we have  $2\rho+|\boldsymbol{\pi}_\Omega(x)|\leq r/2$ so that 
$$  \frac{1}{\rho}\mathcal{E}_\lambda\big(Q_{\rm ref},B_\rho(x)\cap\Omega\big)  \leq  \frac{1}{\rho}\mathcal{E}_\lambda\big(Q_{\rm ref},B_{2\rho}(\boldsymbol{\pi}_\Omega(x))\cap\Omega\big) \leq  \frac{4}{r}\mathcal{E}_\lambda\big(Q_{\rm ref},B_{r}\cap\Omega\big)+C^{\lambda}_{Q_{\rm b}} r \,,$$
by the previous inequality. 

Finally, for $x\in\Omega$ and $|x-\boldsymbol{\pi}_\Omega(x)|>\rho$, we have $B_\rho(x)\subset\Omega$ and thus 
$$ \frac{1}{\rho}\mathcal{E}_\lambda\big(Q_{\rm ref},B_\rho(x)\big)\leq  \frac{1}{|x-\boldsymbol{\pi}_\Omega(x)|}\mathcal{E}_\lambda\big(Q_{\rm ref},B_{|x-\boldsymbol{\pi}_\Omega(x)|}(x)\big)\leq \frac{4}{r}\mathcal{E}_\lambda\big(Q_{\rm ref},B_{r}\cap\Omega\big)+C^{\lambda}_{Q_{\rm b}} r \,,$$
where we have used again the previous inequality,  $|x-\boldsymbol{\pi}_\Omega(x)|\leq r/6$, and $|\boldsymbol{\pi}_\Omega(x)|\leq r/6$.  
\end{proof}

\begin{remark}[Specific geometry \cite{DMP2}]\label{furtherremflatbdrycase}
As already mentioned in Remark \ref{remmonotbdflatcase}, we consider in our companion paper \cite{DMP2} a situation where $0\in\partial\Omega$, $B_1\cap\Omega=B_1\cap\{x_3>0\}$, $Q_{\rm b}$  is constant on $B_1\cap\partial\Omega=B_1\cap\{x_3=0\}$. In this case, if $Q_{\rm ref}\in\mathcal{A}_{Q_{\rm b}}(\Omega)$ satisfies the boundary monotonicity formula 
in Remark \ref{remmonotbdflatcase}, then  we can repeat the argument 
in Lemma \ref{corolmonotform} above to obtain 
\begin{equation}\label{conclfurtherremflatbdrycase}
\mathop{\sup}\limits_{B_\rho(x)\subset B_{1/6}} \frac{1}{\rho}\mathcal{E}_\lambda\big(Q_{\rm ref},B_\rho(x)\cap\Omega\big)\leq 4\mathcal{E}_\lambda\big(Q_{\rm ref},B_1\cap\Omega\big)
\end{equation}
instead of \eqref{conclcormonotformbdry} (with $x_0=0$ and $r=1$). 
\end{remark}

\subsection{Reflection across the boundary}\label{subsecextension}

To obtain regularity estimates at the boundary for critical points of $\mathcal{E}_\lambda$ in the class $\mathcal{A}_{Q_{\rm b}}(\Omega)$, we rely on  general arguments and results developed by C.~Scheven in \cite{Scheven}. Here we make them fully explicit in our case where the target manifold is a sphere  and the boundary is not flat. We even obtain a slight improvement compared to \cite{Scheven} as we only require $C^{1,1}$-regularity for the boundary condition (compared to $C^{2,\alpha}$ in \cite{Scheven}). The main idea is to construct a suitable reflection across the boundary taking into account the prescribed boundary condition $Q_{\rm b}$ in such a way that the reflected critical point satisfies an equation similar in nature to \eqref{distribELeq} in a larger domain. Boundary regularity can then be treated as an interior regularity problem. The aim of this subsection is to construct such reflection and  to derive the resulting equation in the extended domain. We proceed as follows. 

 We still assume that the boundary of the bounded open set $\Omega\subset\R^3$ is of class $C^{3}$. In this way,  we can find a small number $\delta_\Omega>0$ such that the nearest point projection ${\boldsymbol \pi}_\Omega$ on $\partial\Omega$ is well defined and of class $C^{2}$ in the $(2\delta_\Omega)$-tubular neighborhood of $\partial\Omega$ (see e.g. \cite[Chapter 2, Section~2.12.3]{Simon}). We set for $\delta\in(0,2\delta_{\Omega} ]$, 
\begin{align*}
 U_\delta & :=\big\{x\in\R^3:{\rm dist}(x,\partial\Omega)<\delta\big\}\,,\\[3pt] 
U_\delta^{\rm ex} & :=\big\{x\in U_\delta : (x-{\boldsymbol \pi}_\Omega(x))\cdot\nu({\boldsymbol \pi}_\Omega(x))>0\big\}\,,\\[3pt]
 U_\delta^{\rm in} & :=U_\delta\setminus\overline{U_\delta^{\rm ex}}\,,
\end{align*}
where $\nu$ denotes the outer unit  normal vector field on $\partial\Omega$. 
Choosing $\delta_\Omega$ smaller if necessary, we can assume that 
$$\Omega\cap B_{2\delta_\Omega}(x)= U^{\rm in}_{2\delta_\Omega} \cap B_{2\delta_\Omega}(x)\quad\forall x\in\partial\Omega\,.$$
The {\sl geodesic reflection} across $\partial\Omega$ is the involutive $C^2$-diffeomorphism $\boldsymbol{\sigma}_\Omega:U_{2\delta_\Omega}\to U_{2\delta_\Omega}$ given by 
$$\boldsymbol{\sigma}_\Omega(x):=2{\boldsymbol \pi}_\Omega(x)-x\,. $$
It satisfies
$$\boldsymbol{\sigma}_\Omega(U_\delta^{\rm in})= U_\delta^{\rm ex}\quad\forall\delta\in(0,2\delta_\Omega)\,,\text{ and}\quad \boldsymbol{\sigma}_\Omega(x)=x\quad \forall x\in\partial\Omega\,. $$
Being involutive, its (matrix) differential satisfies
\begin{equation}\label{invdiffsigma}
D\boldsymbol{\sigma}_\Omega(\boldsymbol{\sigma}_\Omega(x))D\boldsymbol{\sigma}_\Omega(x)=I \quad\forall x\in U_{2\delta_\Omega}\,. 
\end{equation}
Moreover, for every $x\in\partial\Omega$ we have  
$$D\boldsymbol{\sigma}_\Omega(x) v = 2{\bf p}_x(v)-v\quad \forall v\in \R^3\,,$$
where ${\bf p}_x$ denotes the orthogonal projection of $\R^3$ onto the tangent plane $T_x(\partial\Omega)$, i.e., in this case $D\boldsymbol{\sigma}_\Omega(x)$ is the (linear) reflection across the tangent plane $T_x(\partial\Omega)$. In particular, 
\begin{equation}\label{diffsigmbdident}
D\boldsymbol{\sigma}_\Omega(x) \big(D\boldsymbol{\sigma}_\Omega(x)\big)^\trans =D\boldsymbol{\sigma}_\Omega(x) D\boldsymbol{\sigma}_\Omega(x) =I\quad \forall x\in\partial\Omega\,,
\end{equation}
where $I$ is the identity matrix. We now extend the domain $\Omega$ to the domain 
\begin{equation}\label{defenlargeddom}
\widehat\Omega:=\Omega\cup U_{\delta_\Omega}=\overline{\Omega}\cup U_{\delta_\Omega}^{\rm ex} \,,
\end{equation}
and we simplify the notation by setting 
$$U:=U_{\delta_\Omega}\,,\quad U^{\rm ex}:=U_{\delta_\Omega}^{\rm ex}\,, \quad U^{\rm in}:=U_{\delta_\Omega}^{\rm in}\,.$$ 
On the extended domain $\widehat\Omega$, we consider the Lipschitz continuous field of symmetric $3\times3$-matrices
\begin{equation}\label{matrixfieldextform}
A(x)=\Big(a_{kl}(x)\Big)_{k,l=1}^3:= 
\begin{cases}
\big|J(\boldsymbol{\sigma}_\Omega(x))\big|\,D\boldsymbol{\sigma}_\Omega(\boldsymbol{\sigma}_\Omega(x))\big(D\boldsymbol{\sigma}_\Omega(\boldsymbol{\sigma}_\Omega(x))\big)^\trans & \text{if $x\in \widehat\Omega\setminus\Omega$}\,,\\
I & \text{otherwise}\,,
\end{cases}
\end{equation}
where $J(\boldsymbol{\sigma}_\Omega)$ denotes the Jacobian determinant of $\boldsymbol{\sigma}_\Omega$. Note that the continuity of $A$ across $\partial\Omega$ follows from \eqref{diffsigmbdident}. In addition, \eqref{invdiffsigma} implies that $A$ is uniformly elliptic, i.e., 
$$m_\Omega I\leq A(x) \leq M_\Omega I \quad\forall x\in\widehat \Omega$$
in the sense of quadratic forms for some constants $m_\Omega>0$ and $M_\Omega>0$ depending only on $\Omega$. 
\vskip3pt

Let us now consider for any given $(Q_1,Q_2) \in \mathcal{S}_0\times\mathcal{S}_0$ their tensor product $Q_1\otimes Q_2$ as the linear mapping  $Q_1 \otimes Q_2 \colon \mathcal{S}_0\to \mathcal{S}_0$ defined by
$$(Q_1\otimes Q_2)P:=(P:Q_2)Q_1$$ 
for any $P\in\mathcal{S}_0$. The {\sl geodesic reflection} on $\mathbb{S}^4 \subset \mathcal{S}_0$ with respect to a point $N\in\mathbb{S}^4$ is given by the linear mapping $(2N\otimes N-{\rm id})$, where ${\rm id}$ denotes the identity map on $\mathcal{S}_0$. Note that $(2N\otimes N-{\rm id})$ is simply the orthogonal symmetry with respect to $\langle N\rangle$ which is the identity along $\langle N\rangle$ and minus the identity along any orthogonal direction to $N$. In particular, it is involutive, isometric, and symmetric. 
Given a boundary data $Q_{\rm b}\in C^{1,1}(\partial\Omega;\mathbb{S}^4)$, we consider the mapping $\boldsymbol{\Sigma}:U\to{\rm  GL}(\mathcal{S}_0)$ of class $C^{1,1}$ given by 
$$\boldsymbol{\Sigma}(x):=2Q_{\rm b}\big({\boldsymbol \pi}_\Omega(x)\big)\otimes Q_{\rm b}\big({\boldsymbol \pi}_\Omega(x)\big) -{\rm id} \,.$$

\noindent Notice that by construction $\partial_\nu \Sigma\equiv 0$ on $\partial \Omega$, as $\partial_\nu {\boldsymbol \pi}_{\Omega}(x)=0$ for any $x\in \partial \Omega$.

 With the help of $\boldsymbol{\Sigma}$, we define the extension procedure of maps in $\mathcal{A}_{Q_{\rm b}}(\Omega)$ to the domain $\widehat\Omega$ as follows: to a map $Q\in \mathcal{A}_{Q_{\rm b}}(\Omega)$ we associate $\widehat Q\in W^{1,2}(\widehat\Omega;\mathbb{S}^4)$ given by 
\begin{equation}\label{extprocedure}
\widehat Q(x):=\begin{cases}
Q(x) & \text{if $x\in \Omega$}\,,\\
\boldsymbol{\Sigma}(x)Q(\boldsymbol{\sigma}_\Omega(x)) & \text{if $x\in\widehat\Omega\setminus\Omega$}\,.
\end{cases}
\end{equation}
Note that $\widehat Q$ indeed belongs to $W^{1,2}(\widehat\Omega)$ since $\boldsymbol{\Sigma}Q\circ \boldsymbol{\sigma}_\Omega=\boldsymbol{\Sigma}Q=\boldsymbol{\Sigma}Q_{\rm b}=Q_{\rm b}$ on $\partial\Omega$. 
\vskip3pt

{\fontencoding{U}\fontfamily{futs}\selectfont\char 66\relax} If no confusion arises, {\sl we shall simply write $Q$ instead of $\widehat Q$} the extension of a map $Q$. 
\vskip3pt

\noindent In what follows, we also denote for $P,Q\in W^{1,2}(\widehat\Omega;\mathcal{S}_0)$, 
$$\langle \nabla P,\nabla Q\rangle_A:=\sum_{i,j=1}^3\big(A\nabla P_{ij}\big)\cdot\nabla Q_{ij}=\sum_{k,l=1}^3a_{kl} \partial_kP:\partial_l Q\quad\text{and}\quad|\nabla Q|_A^2:=\langle \nabla Q,\nabla Q\rangle_A\,,$$
where $A$ is the matrix field defined in \eqref{matrixfieldextform}.
\vskip5pt

We are now in position to present the equation satisfied by the extension to $\widehat\Omega$ of a critical point of $\mathcal{E}_\lambda$ in the class $\mathcal{A}_{Q_{\rm b}}(\Omega)$.

\begin{proposition}\label{ELeqExt}
Assume that $\partial\Omega$ is of class $C^{3}$ and $Q_{\rm b}\in C^{1,1}(\partial\Omega;\mathbb{S}^4)$. If $Q_\lambda\in \mathcal{A}_{Q_{\rm b}}(\Omega)$ is a critical point of $\mathcal{E}_\lambda$, then
\begin{equation}\label{ELeqAfterReflect}
-{\rm div}(A\nabla Q_\lambda)=|\nabla Q_\lambda|_A^2Q_\lambda +G_\lambda(x,Q_\lambda,\nabla Q_\lambda)\quad\text{in $\mathscr{D}^\prime(\widehat\Omega)$}\,,
\end{equation}
where $G_\lambda:\widehat\Omega\times \mathbb{S}^4\times(\mathcal{S}_0)^3\to \mathcal{S}_0$  is a Carath\'eodory\footnote{$G(\cdot,Q,\xi)$ is measurable for every $(Q,\xi)\in \mathbb{S}^4\times (\mathcal{S}_0)^3$, and $G(x,\cdot,\cdot)$ is continuous for a.e. $x\in \widehat\Omega$.} map, and
\begin{equation}\label{growthGlambdaext}
|G_\lambda(x,Q,\xi)|\leq C_{Q_{\rm b}}\big(1+\lambda+|\xi|\big) \quad \forall (x,Q,\xi)\in \widehat\Omega\times\mathbb{S}^4\times (\mathcal{S}_0)^3\,,
\end{equation}
for a constant $C_{Q_{\rm b}}>0$ depending only on $\Omega$ and $Q_{\rm b}$. 
\end{proposition}

The proof of Proposition \ref{ELeqExt} essentially rests on the following lemma. 

\begin{lemma}\label{eqextalongtangfield}
Assume that $\partial\Omega$ is of class $C^3$ and $Q_{\rm b}\in C^{1,1}(\partial\Omega;\mathbb{S}^4)$. 
If $Q_\lambda\in \mathcal{A}_{Q_{\rm b}}(\Omega)$ is a critical point of $\mathcal{E}_\lambda$, then 
\begin{multline}\label{eqextenstangfield}
\int_{\widehat \Omega} \langle \nabla  Q_\lambda,\nabla \Phi\rangle_A\,dx=\lambda \int_{\Omega} Q^2_\lambda :\Phi\,dx\\
+\lambda\int_{U^{\rm ex}} \big((Q_\lambda \boldsymbol{\Sigma} Q_\lambda):\Phi\big) f(x)\,dx+\int_{U^{\rm ex}}F(x, Q_\lambda,\nabla  Q_\lambda):\Phi\,dx
\end{multline}
for every $\Phi \in W^{1,2}(\widehat\Omega;{Q_\lambda}^{\!*}T\mathbb{S}^4)$ compactly supported in $\widehat\Omega$, 
where the function $f:U^{\rm ex}\to \R$ is continuous, the map $F:U^{\rm ex}\times\mathbb{S}^4\times (\mathcal{S}_0)^3\to \mathcal{S}_0$ is Carath\'eodory, and 
$$0\leq f(x)\leq C_\Omega\quad\text{and}\quad \big|F(x,Q,\xi)\big|\leq C_{Q_{\rm b}}(1+|\xi|) \quad\forall (x,Q,\xi)\in U^{\rm ex}\times\mathbb{S}^4\times (\mathcal{S}_0)^3\,, $$
for some constants $C_\Omega>0$ (depending only on $\Omega$) and $C_{Q_{\rm b}}>0$ (depending only on $\Omega$ and $Q_{\rm b}$). 
\end{lemma}

\begin{proof}
If $\Phi \in W^{1,2}(\widehat\Omega;{Q_\lambda}^{\!*}T\mathbb{S}^4)$ is compactly supported in $\Omega$, then \eqref{eqextenstangfield} reduces to \eqref{varELeqtangfields}. Therefore, it suffices to consider the case where $\Phi$ is compactly supported in $U$. Following the argument in \cite{Scheven}, we decompose $\Phi$ into its equivariant and anti-equivariant parts with respect to the involution $\Phi(x) \to \Sigma(x)\Phi(\sigma_\Omega(x))$, defined for $x\in U$ by 
$$\Phi^{\rm e}(x):=\frac{1}{2}\Big(\Phi(x)+\boldsymbol{\Sigma}(x)\Phi(\boldsymbol{\sigma}_\Omega(x))\Big)\quad \text{and}\quad \Phi^{\rm a}(x):=\frac{1}{2}\Big(\Phi(x)-\boldsymbol{\Sigma}(x)\Phi(\boldsymbol{\sigma}_\Omega(x))\Big) \,.$$
Here equivariance is understood in terms of the joint reflections across the boundary and on $\mathbb{S}^4$. Thus, one simply obtains 
$$\Phi^{\rm e}(\boldsymbol{\sigma}_\Omega(x))=\boldsymbol{\Sigma}(x)\Phi^{\rm e}(x) \quad\text{and} \quad \Phi^{\rm a}(\boldsymbol{\sigma}_\Omega(x))=-\boldsymbol{\Sigma}(x)\Phi^{\rm a}(x) \quad \forall x\in U\,.$$ 
We shall prove  \eqref{eqextenstangfield} for $\Phi^{\rm e}$ and $\Phi^{\rm a}$ separately, starting with $\Phi^{\rm a}$. To  this purpose, we consider $Q_\lambda$ as extended to the whole $U$ as in  \eqref{extprocedure} and we also introduce for $x\in U$, 
$$Q_\lambda^*(x):=Q_\lambda(\boldsymbol{\sigma}_\Omega(x))=\boldsymbol{\Sigma}(x)Q_\lambda(x)\,.$$ 
We start from the identity 
\begin{align}
\nonumber\int_{U^{\rm ex}} \langle \nabla  Q_\lambda,\nabla \Phi^{\rm a}\rangle_A\,dx 
&=\sum_{k,l=1}^3 \int_{U^{\rm ex}} a_{kl}\partial_k (\boldsymbol{\Sigma} Q^*_\lambda) : \partial_l\Phi^{\rm a}\,dx\\
\nonumber &=\sum_{k,l=1}^3 \int_{U^{\rm ex}} a_{kl}( \boldsymbol{\Sigma}\partial_k Q^*_\lambda) : \partial_l\Phi^{\rm a}\,dx
+\sum_{k,l=1}^3 \int_{U^{\rm ex}} a_{kl} \big((\partial_k\boldsymbol{\Sigma}) Q^*_\lambda \big): \partial_l\Phi^{\rm a}\,dx\\
\label{computreflect1}&=: I + II\,.
\end{align}
To compute the $II$-term, we integrate by parts. Since $A$ is the identity matrix on $\partial\Omega$ and $\partial_\nu \boldsymbol{\Sigma}=0$ on $\partial\Omega$, the boundary term vanishes, and we are left with 
\begin{multline}\label{computreflect2}
II=- \sum_{k,l=1}^3 \int_{U^{\rm ex}} \partial_l\big[a_{kl} (\partial_k\boldsymbol{\Sigma}) Q^*_\lambda \big] : \Phi^{\rm a}\,dx=- \sum_{k,l=1}^3 \int_{U^{\rm ex}} \partial_l\big[a_{kl} (\partial_k\boldsymbol{\Sigma}) \boldsymbol{\Sigma}  Q_\lambda \big] : \Phi^{\rm a}\,dx\\
=- \sum_{k,l=1}^3 \int_{U^{\rm ex}} (\partial_l a_{kl}) \big((\partial_k\boldsymbol{\Sigma})\boldsymbol{\Sigma} Q_\lambda \big) : \Phi^{\rm a}\,dx-\sum_{k,l=1}^3 \int_{U^{\rm ex}} a_{kl}\big((\partial^2_{kl}\boldsymbol{\Sigma})\boldsymbol{\Sigma} Q_\lambda \big) : \Phi^{\rm a}\,dx\\
-\sum_{k,l=1}^3 \int_{U^{\rm ex}} a_{kl}\big((\partial_{k}\boldsymbol{\Sigma})(\partial_{l}\boldsymbol{\Sigma}) Q_\lambda \big) : \Phi^{\rm a}\,dx-\sum_{k,l=1}^3 \int_{U^{\rm ex}} a_{kl}\big((\partial_{k}\boldsymbol{\Sigma})\boldsymbol{\Sigma}\partial_{l} Q_\lambda \big) : \Phi^{\rm a}\,dx\,.
\end{multline} 
Concerning the $I$-term, we use the anti-equivariance of $\Phi^{\rm a}$ to derive 
\begin{align}
\nonumber I = &\sum_{k,l=1}^3 \int_{U^{\rm ex}} a_{kl}\partial_k Q^*_\lambda : ( \boldsymbol{\Sigma}\partial_l\Phi^{\rm a})\,dx\\
\nonumber =&\sum_{k,l=1}^3 \int_{U^{\rm ex}} a_{kl}\partial_k Q^*_\lambda : \partial_l( \boldsymbol{\Sigma}\Phi^{\rm a})\,dx
-\sum_{k,l=1}^3 \int_{U^{\rm ex}} a_{kl}\partial_k Q^*_\lambda : \big( (\partial_l\boldsymbol{\Sigma})\Phi^{\rm a}\big)\,dx\\
\nonumber =&-\sum_{k,l=1}^3 \int_{U^{\rm ex}} a_{kl}\partial_k (Q_\lambda\circ\boldsymbol{\sigma}_\Omega) : \partial_l( \Phi^{\rm a}\circ\boldsymbol{\sigma}_\Omega)\,dx
-\sum_{k,l=1}^3 \int_{U^{\rm ex}} a_{kl}\big((\partial_l\boldsymbol{\Sigma})\partial_k( \boldsymbol{\Sigma} Q_\lambda)\big) : \Phi^{\rm a}\,dx\\
\nonumber =&-\sum_{k,l=1}^3 \int_{U^{\rm ex}} a_{kl}\partial_k (Q_\lambda\circ\boldsymbol{\sigma}_\Omega) : \partial_l( \Phi^{\rm a}\circ\boldsymbol{\sigma}_\Omega)\,dx
-\sum_{k,l=1}^3 \int_{U^{\rm ex}} a_{kl}\big((\partial_l\boldsymbol{\Sigma})(\partial_k\boldsymbol{\Sigma}) Q_\lambda\big) : \Phi^{\rm a}\,dx \\
\label{computreflect3} &- \sum_{k,l=1}^3 \int_{U^{\rm ex}} a_{kl}\big((\partial_l\boldsymbol{\Sigma})\boldsymbol{\Sigma}\partial_k Q_\lambda\big) : \Phi^{\rm a}\,dx\,.
\end{align} 
Next we change variables  in the first term of the last identity, and by \eqref{invdiffsigma} we obtain 
\begin{align}
\nonumber -\sum_{k,l=1}^3 \int_{U^{\rm ex}} a_{kl}\partial_k &(Q_\lambda\circ\boldsymbol{\sigma}_\Omega) : \partial_l( \Phi^{\rm a}\circ\boldsymbol{\sigma}_\Omega)\,dx =-\sum_{i,j=1}^3  \int_{U^{\rm ex}} A\nabla (Q_{\lambda,ij}\circ\boldsymbol{\sigma}_\Omega) \cdot \nabla( \Phi^{\rm a}_{ij}\circ\boldsymbol{\sigma}_\Omega)\,dx\\
\nonumber &= -\sum_{i,j=1}^3\int_{U^{\rm ex}} \big[D\boldsymbol{\sigma}_\Omega(x) A(x)(D\boldsymbol{\sigma}_\Omega(x))^\trans\big] \nabla Q_{\lambda,ij}(\boldsymbol{\sigma}_\Omega(x)) \cdot \nabla\Phi^{\rm a}_{ij}(\boldsymbol{\sigma}_\Omega(x))\,dx\\
\nonumber &= -\sum_{i,j=1}^3\int_{U^{\rm ex}}  \nabla Q_{\lambda,ij}(\boldsymbol{\sigma}_\Omega(x)) \cdot \nabla\Phi^{\rm a}_{ij}(\boldsymbol{\sigma}_\Omega(x)) \big|J(\boldsymbol{\sigma}_\Omega(x))\big|\,dx\\
\nonumber &= - \sum_{i,j=1}^3\int_{U^{\rm in}}  \nabla Q_{\lambda,ij} \cdot \nabla\Phi^{\rm a}_{ij}\,dx\\
\label{computreflect4} &=- \int_{U^{\rm in}}\langle \nabla  Q_\lambda,\nabla \Phi^{\rm a}\rangle_A\,dx\,.
\end{align}
Since $\boldsymbol{\Sigma}^2={\rm id}$, we have the identities everywhere (resp. a.e.) in $U$, 
$$(\partial_k\boldsymbol{\Sigma})\boldsymbol{\Sigma}+\boldsymbol{\Sigma}(\partial_k\boldsymbol{\Sigma})=0 \text{ and }(\partial^2_{kl}\boldsymbol{\Sigma})\boldsymbol{\Sigma}+(\partial_k\boldsymbol{\Sigma})(\partial_l\boldsymbol{\Sigma})+(\partial_l\boldsymbol{\Sigma})(\partial_k\boldsymbol{\Sigma})+\boldsymbol{\Sigma}(\partial^2_{kl}\boldsymbol{\Sigma})=0\,,  $$
so that gathering \eqref{computreflect1}, \eqref{computreflect2}, \eqref{computreflect3}, and \eqref{computreflect4} yields 
\begin{multline*}
\nonumber\int_{U^{\rm ex}} \langle \nabla  Q_\lambda,\nabla \Phi^{\rm a}\rangle_A\,dx =- \int_{U^{\rm in}}\langle \nabla  Q_\lambda,\nabla \Phi^{\rm a}\rangle_A\,dx\\
+ \sum_{k,l=1}^3\int_{U^{\rm ex}} \boldsymbol{\Sigma}\Big(\big(a_{kl}\partial^2_{kl}\boldsymbol{\Sigma}+\partial_la_{kl}\partial_k\boldsymbol{\Sigma}\big)Q_\lambda+2a_{kl}(\partial_k\boldsymbol{\Sigma})\partial_l Q_\lambda \Big):\Phi^{\rm a}\,dx\,.
\end{multline*}
Consequently, 
\begin{equation}\label{eqreflecanteq}
\int_{U} \langle \nabla  Q_\lambda,\nabla \Phi^{\rm a}\rangle_A\,dx= \int_{U^{\rm ex}} F(x, Q_\lambda,\nabla  Q_\lambda):\Phi^{\rm a}\,dx
\end{equation}
with 
$$F(x, Q_\lambda,\nabla  Q_\lambda):=\sum_{k,l=1}^3\boldsymbol{\Sigma}(x)\Big(\big(a_{kl}(x)\partial^2_{kl}\boldsymbol{\Sigma}(x)+\partial_la_{kl}(x)\partial_k\boldsymbol{\Sigma}(x)\big) Q_\lambda +2a_{kl}(x)\partial_k\boldsymbol{\Sigma}(x)\partial_l Q_\lambda\Big)\,. $$
Clearly, $F:U^{\rm ex}\times\mathbb{S}^4\times (\mathcal{S}_0)^3\to \mathcal{S}_0$ is Carath\'eodory and it is sublinear in its third argument because $\Sigma \in C^{1,1}$ and $|Q_\lambda|\leq 1$ in $U$. 
\vskip3pt

It now remains to perform the computations with the equivariant part $\Phi^{\rm e}$. First, we observe that $\Phi^{\rm e}=0$ on $\partial\Omega$. Indeed, since the function $( Q_\lambda:\Phi)$ belongs to $W^{1,1}(U)$, it has a trace on $\partial\Omega$, and this trace is  equal to the inner product of the traces on $\partial\Omega$.  Since $(Q_\lambda:\Phi)=0$ in $U$, and $Q_\lambda=Q_{\rm b}$ on~$\partial\Omega$, we infer that  $(Q_{\rm b}:\Phi)=0$ on $\partial\Omega$. Hence $\boldsymbol{\Sigma}\Phi=-\Phi$ on $\partial\Omega$, which yields $\Phi^{\rm e}=0$ on~$\partial\Omega$. As a consequence, $\Phi^{\rm e}\in W^{1,2}_0(U^{\rm in};\mathcal{S}_0)$.  Moreover, for a.e. $x\in U^{\rm in}$,  
$$\Phi^{\rm e}(x):Q_\lambda(x)=\frac{1}{2}\Phi\big(\boldsymbol{\sigma}_\Omega(x)\big):\big(\boldsymbol{\Sigma}(x)Q_\lambda(x)\Big)= \frac{1}{2}\Phi\big(\boldsymbol{\sigma}_\Omega(x)\big):Q_\lambda\big(\boldsymbol{\sigma}_\Omega(x)\big)=0\,,$$
and thus $\Phi^{\rm e}\in W^{1,2}_0(U^{\rm in};{Q_\lambda}^{\!*}T\mathbb{S}^4)$.
Thanks to the regularity of $\partial\Omega$, \eqref{varELeqtangfields} holds for every test function in $W^{1,2}_0(\Omega;{Q_\lambda}^{\!*}T\mathbb{S}^4)$ by approximation. Therefore,  
\begin{equation}\label{eqtangeq}
\int_{U^{\rm in}}\langle \nabla  Q_\lambda,\nabla \Phi^{\rm e}\rangle_A\,dx=\int_{U^{\rm in}}\nabla Q_\lambda :\nabla \Phi^{\rm e} \,dx
=\lambda \int_{U^{\rm in}} Q^2_\lambda: \Phi^{\rm e}\,dx\,.
\end{equation}
Next, from the definition of $Q^*_\lambda$ we have an identity analogous to \eqref{computreflect1}, namely
\begin{align}
\nonumber \int_{U^{\rm ex}} \langle \nabla  Q_\lambda,\nabla \Phi^{\rm e}\rangle_A\,dx 
&=\sum_{k,l=1}^3 \int_{U^{\rm ex}} a_{kl}( \boldsymbol{\Sigma}\partial_k Q^*_\lambda) : \partial_l\Phi^{\rm e}\,dx
+\sum_{k,l=1}^3 \int_{U^{\rm ex}} a_{kl} \big((\partial_k\boldsymbol{\Sigma}) Q^*_\lambda \big): \partial_l\Phi^{\rm e}\,dx\\
\label{partIII+IV}&=: III + IV\,.
\end{align} 
The computations of $IV$ are identical to the ones of $II$ in \eqref{computreflect2}, with $\Phi^{\rm e}$ instead of $\Phi^{\rm a}$. Similarly, we can compute $III$ in a way similar to \eqref{computreflect3}, thus using the equivariance of $\Phi^{\rm e}$ and the change of variable as in \eqref{computreflect4} we obtain
\begin{align}
\nonumber III =&\sum_{k,l=1}^3 \int_{U^{\rm ex}} a_{kl}\partial_k (Q_\lambda\circ\boldsymbol{\sigma}_\Omega) : \partial_l( \Phi^{\rm e}\circ\boldsymbol{\sigma}_\Omega)\,dx
-\sum_{k,l=1}^3 \int_{U^{\rm ex}} a_{kl}\big((\partial_l\boldsymbol{\Sigma})(\partial_k\boldsymbol{\Sigma}) Q_\lambda\big) : \Phi^{\rm e}\,dx \\
\nonumber &- \sum_{k,l=1}^3 \int_{U^{\rm ex}} a_{kl}\big((\partial_l\boldsymbol{\Sigma})\boldsymbol{\Sigma}\partial_k Q_\lambda\big) : \Phi^{\rm e}\,dx\, \\
 =& \int_{U^{\rm in}}\langle \nabla  Q_\lambda,\nabla \Phi^{\rm e}\rangle_A\,dx\, 
-\sum_{k,l=1}^3 \int_{U^{\rm ex}} a_{kl}\big((\partial_l\boldsymbol{\Sigma})(\partial_k\boldsymbol{\Sigma}) Q_\lambda\big) : \Phi^{\rm e}\,dx \\
\label{computreflect5} &- \sum_{k,l=1}^3 \int_{U^{\rm ex}} a_{kl}\big((\partial_l\boldsymbol{\Sigma})\boldsymbol{\Sigma}\partial_k Q_\lambda\big) : \Phi^{\rm e}\,dx\,
\end{align} 

Summing up the contributions for $III$ and $IV$, in view of the identities for $\Sigma$ and its derivatives we infer
\begin{equation}
\label{computeqfin}
\int_{U^{\rm ex}} \langle \nabla  Q_\lambda,\nabla \Phi^{\rm e}\rangle_A\,dx =\int_{U^{\rm in}}\langle \nabla  Q_\lambda,\nabla \Phi^{\rm e}\rangle_A\,dx+ \int_{U^{\rm ex}} F(x, Q_\lambda,\nabla  Q_\lambda):\Phi^{\rm e}\,dx\,,
\end{equation}
with the same $F$ as in \eqref{eqreflecanteq}.

Combining \eqref{eqtangeq} and \eqref{computeqfin} leads to 
\begin{equation}\label{eqrefleceq}
\int_{U}\langle \nabla  Q_\lambda,\nabla \Phi^{\rm e}\rangle_A\,dx=2\lambda \int_{U^{\rm in}} Q^2_\lambda: \Phi^{\rm e}\,dx+ \int_{U^{\rm ex}} F(x, Q_\lambda,\nabla  Q_\lambda):\Phi^{\rm e}\,dx\,.
\end{equation}
Finally, summing \eqref{eqreflecanteq} with \eqref{eqrefleceq}, we are led to 
\begin{multline}\label{reflectangfield1}
\int_{U}\langle \nabla  Q_\lambda,\nabla \Phi\rangle_A\,dx=2\lambda \int_{U^{\rm in}} Q^2_\lambda: \Phi^{\rm e}\,dx+ \int_{U^{\rm ex}} F(x, Q_\lambda,\nabla  Q_\lambda):\Phi\,dx\\
=\lambda \int_{U^{\rm in}} Q^2_\lambda: \Phi\,dx+\lambda  \int_{U^{\rm in}} Q^2_\lambda: (\boldsymbol{\Sigma}\Phi\circ\boldsymbol{\sigma}_\Omega)\,dx\\
+\int_{U^{\rm ex}} F(x, Q_\lambda,\nabla  Q_\lambda):\Phi\,dx\,.
\end{multline}
Changing variables once again, we derive 
\begin{align}
\nonumber\int_{U^{\rm in}} Q^2_\lambda: (\boldsymbol{\Sigma}\Phi\circ\boldsymbol{\sigma}_\Omega)\,dx=&\int_{U^{\rm in}}\big[  \boldsymbol{\Sigma} Q^2_\lambda(\boldsymbol{\sigma}_\Omega(x))\big]: \Phi(\boldsymbol{\sigma}_\Omega(x))\,dx\\
\label{reflectangfield2}=& \int_{U^{\rm ex}}\big( ( \boldsymbol{\Sigma} Q^2_\lambda): \Phi\big)f(x)\,dx\,,
\end{align}
with $f:=|J(\boldsymbol{\sigma}_\Omega)|$. Combining \eqref{reflectangfield1} and \eqref{reflectangfield2}, the conclusion follows. 
\end{proof}

\begin{proof}[Proof of Proposition \ref{ELeqExt}]
Starting from Lemma \ref{eqextalongtangfield}, we proceed as in the proof of Proposition~\ref{propstartvarEL}. Given $\Phi\in C^\infty_c\big(\widehat \Omega;\mathscr{M}^{\rm sym}_{3\times 3}(\R)\big)$, we consider $\Phi_0:=\Phi-\frac{1}{3}(\Phi:I) I\in C^\infty_c\big(\widehat \Omega;\mathcal{S}_0\big)$ and 
$$\Phi_*:=\Phi_0-( Q_\lambda:\Phi_0) Q_\lambda \in W^{1,2}(\widehat \Omega;{Q_\lambda}^{\!*}T\mathbb{S}^4)\,. $$
Since $\Phi_*$ is compactly supported in $\widehat \Omega$,  \eqref{eqextenstangfield} applies. On the other hand, direct computations yield
\begin{align}
\nonumber\int_{\widehat \Omega}\langle \nabla Q_\lambda,\nabla \Phi_*\rangle_A\,dx&= \int_{\widehat \Omega}\langle \nabla  Q_\lambda,\nabla \Phi_0\rangle_A\,dx- \int_{\widehat \Omega}|\nabla Q_\lambda|^2_AQ_\lambda:\Phi_0\,dx\\
\label{computtensfield}&= \int_{\widehat \Omega}\langle \nabla  Q_\lambda,\nabla \Phi\rangle_A\,dx- \int_{\widehat \Omega}|\nabla Q_\lambda|^2_AQ_\lambda:\Phi\,dx\,,
\end{align}
and 
\begin{multline}\label{computsourceext}
\lambda \int_{\Omega} Q^2_\lambda :\Phi_*\,dx
+\lambda\int_{U^{\rm ex}} \big((\boldsymbol{\Sigma} Q^2_\lambda):\Phi_*\big) f(x)\,dx+\int_{U^{\rm ex}}F(x, Q_\lambda,\nabla  Q_\lambda):\Phi_*\,dx\\
= \int_{\hat \Omega} G_\lambda(x,Q_\lambda,\nabla Q_\lambda):\Phi\,dx\,,
\end{multline}
with 
\begin{multline*}
G_\lambda(x, Q_\lambda,\nabla Q_\lambda) :=\lambda\chi_\Omega(x)\Big[  Q_\lambda^2 -\frac{1}{3} I - {\rm tr}( Q_\lambda^3)  Q_\lambda\Big]\\
+\lambda\chi_{U^{\rm ex}}(x)f(x)\Big[  \boldsymbol{\Sigma} Q^2_\lambda-\frac{1}{3}{\rm tr}( \boldsymbol{\Sigma}  Q^2_\lambda) I - {\rm tr}\big(\boldsymbol\Sigma Q^3_\lambda \big) Q_\lambda\Big]\\
+\chi_{U^{\rm ex}}(x)\Big[F(x, Q_\lambda,\nabla  Q_\lambda)-\frac{1}{3}{\rm tr}\big(F(x, Q_\lambda,\nabla  Q_\lambda)\big)I -{\rm tr}\big(F(x, Q_\lambda,\nabla  Q_\lambda) Q_\lambda\big)  Q_\lambda\Big]\,.
\end{multline*}
Combining \eqref{eqextenstangfield}, \eqref{computtensfield}, and \eqref{computsourceext} leads to the conclusion. 
\end{proof}

Before closing the subsection, we provide a counterpart to Lemma \ref{corolmonotform} for reflected maps. 

\begin{lemma}\label{lemmacontrolscalenergext}
Assume that $\partial\Omega$ is of class $C^3$ and $Q_{\rm b}\in C^{1,1}(\partial\Omega;\mathbb{S}^4)$. Let $Q_{\rm ref}\in\mathcal{A}_{Q_{\rm b}}(\Omega)$ satisfying conclusion \eqref{conclcormonotformbdry} in Lemma~\ref{corolmonotform}. There exist two constants ${\bf r}^{(2)}_\Omega>0$ and $\boldsymbol{\kappa}=\boldsymbol{\kappa}_\Omega\in(0,1)$ depending only on  $\Omega$ such that 
 for every $x_0\in\partial\Omega$ and $r\in(0, {\bf r}^{(2)}_\Omega)$, 
\begin{equation}\label{contrenergboulext}
\mathop{\sup}\limits_{B_\rho(x)\subset B_{\boldsymbol{\kappa} r}(x_0)} \frac{1}{\rho}\int_{B_\rho(x)}|\nabla \widehat Q_{\rm ref}|^2\,dx\leq \frac{C_\Omega}{r}\mathcal{E}_\lambda\big(Q_{\rm ref},B_r(x_0)\cap\Omega\big) +C^\lambda_{Q_{\rm b}} r\,,
\end{equation}
where $C_\Omega>0$ only depends on $\Omega$, and $C^\lambda_{Q_{\rm b}}>0$ only depends on $\Omega$, $Q_{\rm b}$,  $\lambda\|W(Q_{\rm b})\|_{L^1(\partial\Omega)}$, and $\mathcal{E}_\lambda(\bar{Q}_{\rm b})$. 
\end{lemma}

\begin{proof}
Set $\boldsymbol{\kappa}:=\frac{1}{6}\min(\|D\boldsymbol{\sigma}_\Omega\|^{-1}_{L^\infty(U)},1)$, and  ${\bf r}^{(2)}_\Omega:=\min({\bf r}^{(1)}_\Omega, \delta_\Omega)$, where ${\bf r}^{(1)}_\Omega>0$ is given by Lemma \ref{corolmonotform}. Given a point $x_0\in \partial\Omega$ and a radius $r\in(0,{\bf r}^{(2)}_\Omega)$, we apply  \eqref{conclcormonotformbdry} to estimate in a ball  $B_\rho(x)\subset B_{\boldsymbol{\kappa} r}(x_0)$,  
\begin{align}
\nonumber\frac{1}{\rho}\int_{B_\rho(x)}|\nabla \widehat Q_{\rm ref}|^2\,dx&=\frac{1}{\rho}\int_{B_\rho(x)\cap\Omega}|\nabla Q_{\rm ref}|^2\,dx+\frac{1}{\rho}\int_{B_\rho(x)\cap U^{\rm ex}}|\nabla \widehat Q_{\rm ref}|^2\,dx\\
\label{varenergext1}&\leq \frac{4}{r}\mathcal{E}_\lambda(Q_{\rm ref},B_r(x_0)\cap\Omega)+ \frac{1}{\rho}\int_{B_\rho(x)\cap U^{\rm ex}}|\nabla \widehat Q_{\rm ref}|^2\,dx+C^\lambda_{Q_{\rm b}} r\,.
\end{align}
Using the facts that $\boldsymbol{\Sigma}(x)$ is isometric for every $x\in U$ and $|Q_{\rm ref}|=1$, we estimate 
\begin{multline*}
\int_{B_\rho(x)\cap U^{\rm ex}}|\nabla \widehat Q_{\rm ref}|^2\,dx=\int_{B_\rho(x)\cap U^{\rm ex}} \big|\nabla ({\boldsymbol\Sigma}Q_{\rm ref}\circ \boldsymbol{\sigma}_\Omega)\big|^2\,dx\\
 \leq 2 \int_{B_\rho(x)\cap U^{\rm ex}} \big|\nabla (Q_{\rm ref}\circ \boldsymbol{\sigma}_\Omega)\big|^2\,dx+C_{Q_{\rm b}}\rho^3\leq C_\Omega \int_{\boldsymbol{\sigma}_\Omega(B_\rho(x))\cap U^{\rm in}} |\nabla Q_{\rm ref}|^2\,dx+C_{Q_{\rm b}}\rho^3\,,
\end{multline*}
where the last inequality follows from a change of variables. Setting $y:=\boldsymbol{\sigma}_\Omega(x)$, we observe that  $\boldsymbol{\sigma}_\Omega(B_\rho(x))\cap U^{\rm in}\subset B_{\rho/(6\boldsymbol{\kappa})}(y)\cap U^{\rm in}$ and $B_{\rho/(6\boldsymbol{\kappa})}(y)\subset B_{r/6}(x_0)$, and consequently 
\begin{multline}\label{varenergext2}
\frac{1}{\rho} \int_{B_\rho(x)\cap U^{\rm ex}}|\nabla \widehat Q_{\rm ref}|^2\,dx\leq \frac{C_\Omega}{\rho} \int_{B_{\rho/(6\boldsymbol{\kappa})}(y)\cap \Omega} |\nabla Q_{\rm ref}|^2\,dx+C_{Q_{\rm b}}\rho^2\\
\leq \frac{C_\Omega}{r}\mathcal{E}_\lambda(Q_{\rm ref},B_r(x_0)\cap\Omega)+C^\lambda_{Q_{\rm b}}r\,,
\end{multline}
thanks again to \eqref{conclcormonotformbdry}. The result now follows from \eqref{varenergext1} and \eqref{varenergext2}. 
\end{proof}

\begin{remark}[Specific geometry \cite{DMP2}]\label{specifgeomreflec}
Recall from Remark \ref{furtherremflatbdrycase} that we shall consider in \cite{DMP2} the  following situation:  $0\in\partial\Omega$, $B_1\cap\Omega=B_1\cap\{x_3>0\}$, and $Q_{\rm b}$  is constant on $B_1\cap\partial\Omega=B_1\cap\{x_3=0\}$. In this case,  $\boldsymbol{\Sigma}$ is constant in $B_1$, and $\boldsymbol{\sigma}_\Omega(x)=(x_1,x_2,-x_3)=:\bar{x}$ for every $x=(x_1,x_1,x_3)\in B_1$. Hence $|\nabla \widehat Q_{\rm ref}(x)|^2=|\nabla Q_{\rm ref}(\bar{x})|^2$ for every $x\in B_1\cap\{x_3<0\}$.  As a consequence, if $Q_{\rm ref}$ satisfies conclusion \eqref{conclfurtherremflatbdrycase} in  Remark \ref{furtherremflatbdrycase}, then 
$$\mathop{\sup}\limits_{B_\rho(x)\subset B_{1/6}} \frac{1}{\rho}\mathcal{E}_\lambda\big(\widehat Q_{\rm ref},B_\rho(x)\big)\leq 8\mathcal{E}_\lambda\big(Q_{\rm ref},B_1\cap\Omega\big)\,,$$
instead of \eqref{contrenergboulext} (with $x_0=0$ and $r=1$). 
\end{remark}

\subsection{The $\varepsilon$-regularity theorem} In this subsection, we present the main regularity estimate which provides local H\"older regularity for weak solutions of \eqref{distribELeq} under a smallness assumption on the energy. To treat interior and boundary estimates in a unified way, we consider the case of a general system with diagonal principal part, corresponding to the scalar operator  $Lv=- {\rm div}(A\nabla v)$, as it appears in Proposition \ref{ELeqExt}.

\begin{theorem}\label{epsregthm}
Let $r_0\in(0,1]$ and $A:B_{r_0}\to\mathscr{M}^{\rm sym}_{3\times3}(\R)$ be a Lipschitz field of  symmetric matrices, and assume that $A$ is uniformly elliptic (i.e., $m I\leq A\leq M I$ for some constants $m>0$ and $M>0$). Let $Q\in W^{1,2}(B_{r_0};\mathbb{S}^4)$ and $G\in L^2(B_{r_0};\mathcal{S}_0)$ be such that 
\begin{equation}\label{eqepsregthm}
-{\rm div}(A\nabla Q)=|\nabla Q|^2_AQ+G \quad\text{in $\mathscr{D}^\prime(B_{r_0})$} \,.
\end{equation}
There exist two constants $\boldsymbol{\eps}_A>0$ and $C_A>0$, and an exponent $\alpha=\alpha(A)\in(0,1)$ depending only on the Lipschitz norm of $A$ in $B_{r_0}$ and the ellipticity bounds $m$ and $M$ such that the condition 
\begin{equation}\label{epscondthm}
 \mathop{\sup}\limits_{B_r(\bar{x})\subset B_{r_0}} \left(\frac{1}{r}\int_{B_r(\bar{x})} |\nabla Q|^2\,dx +r\int_{B_r(\bar{x})}|G|^2\,dx\right) \leq \boldsymbol{\eps}_A
 \end{equation}
implies $Q\in C^{0,\alpha}(\overline B_{r_0/2})$ with  $[Q]_{C^{0,\alpha}(B_{r_0/2})}\leq C_Ar_0^{-\alpha}$. 
\end{theorem}

We postpone the proof of this theorem as we require some preliminary lemmas. To this purpose, let us first recall the notion of {\sl function of bounded mean oscillation}. Given an open ball $B\subset \R^d$, a function $u\in L^1(B)$ belongs to the space ${\rm BMO}(B)$ if  
$$\|u\|_{{\rm BMO}(B)}:=\sup_{ \overline B_\rho(y)\subset B} \dashint_{B_\rho(y)}\Big| u -\dashint_{B_\rho(y)} u\, \Big|\,dx<+\infty\,, $$
where the supremum is taken over closed balls $\overline B_\rho(y)$ as above.  Analogously, for $p>1$ a function $u\in L^p(B)$ belongs to the space ${\rm BMO}^p(B)$ if  
$$\|u\|^p_{{\rm BMO}^p(B)}:=\sup_{\overline B_\rho(y)\subset B} \dashint_{B_\rho(y)}\Big| u -\dashint_{B_\rho(y)} u\, \Big|^p\,dx<+\infty\,, $$
where as above the supremum is taken over closed balls $\overline B_\rho(y)$. It is well known that taking closed cubes inside $B$ or closed balls $\overline B_\rho(y)$ such that $B_{2\rho}(y) \subset B$ gives equivalent definitions where the previous quantities are equivalent norms (see \cite{Sta}).

A first ingredient coming into play is the  classical John-Nirenberg inequality, see e.g. \cite[Chapter~19]{HKM}.  

\begin{lemma}[John-Nirenberg inequality]\label{Johnirlem}
For every $1<p<\infty$, there exists a constant $C_p>1$ depending only on $p$ and the dimension such that 
$$\frac{1}{C_p}\|u\|^p_{{\rm BMO}(B)} \leq \|u\|^p_{{\rm BMO^p}(B)}
\leq C_p\|u\|^p_{{\rm BMO}(B)} $$
for every $u\in {\rm BMO}(B)$. 
\end{lemma}

The second result is a standard scaling-invariant local regularity estimate for solutions of linear elliptic PDE's. Since the result is standard but we were not able to find a reference in the literature we sketch the proof for the reader's convenience.

\begin{lemma}\label{linearestiepsreg}
For $d\geq 3$, let $A: \widetilde{\Omega} \subset\R^d\to\mathscr{M}_{d\times d}(\R)$ be a Lipschitz field of  symmetric matrices, and assume that $A$ is uniformly elliptic (i.e., $m I\leq A\leq M I$ in $\widetilde{\Omega}$ for some constants $m>0$ and $M>0$). Let $f\in L^2(\widetilde{\Omega};\R^d)$, $g\in L^2(\widetilde{\Omega})$ and for each $B_r \subset \widetilde{\Omega}$, $0<r\leq 1$, consider $u\in W_0^{1,2}(B_{r})$   the (unique) weak solution of 
$$\begin{cases}
-{\rm div}(A\nabla u)={\rm div}\,f + g & \text{in $B_{r}$}\,,\\
u=0 & \text{on $\partial B_{r}$}\,.
\end{cases}$$
For every $q\in(\frac{d}{d-1},2)$, there exists a constant $C_A=C_A(q)$ depending only on $q$, $d$ and the Lipschitz norm of $A$ in $\widetilde{\Omega}$ (i.e.,  not on the radius $r$) such that 
$$\|\nabla u\|_{L^q(B_{ r})} \leq C_A\Big(\|f\|_{L^q(B_{ r})} + \|g\|_{L^{\frac{dq}{d+q}}(B_{r})}\Big)\,.$$
\begin{proof}(Sketch) Since all the norms in the inequality have the same scaling properties and the Lipschitz norm of $A$ is decreasing under scaling with factor $r\leq 1$ we may assume $r=1$. Then the estimate for $q=2$ just follows testing with $u$, integrating by parts and using Sobolev inequality. The case $q\in (2,d)$ follows from the case $q=2$ and the combination of \cite[Theorem 9.15]{GilbTrud} for the case $f\equiv0$ with \cite[Theorem 10.17]{Giusti} for the case $g\equiv0$. Finally, standard duality arguments give the desired conclusion in the dual range of exponents $q\in(\frac{d}{d-1},2)$. 
\end{proof}
\end{lemma}

The final ingredient is the following local gradient estimate for $A-$harmonic functions.

\begin{lemma}\label{lemsupgradAharmfct}
For $d\geq 2$,  let $A: \widetilde{\Omega} \subset\R^d\to\mathscr{M}_{d\times d}(\R)$ be a Lipschitz field of  symmetric matrices, and assume that $A$ is uniformly elliptic (i.e., $m I\leq A\leq M I$ in $\widetilde{\Omega}$ for some constants $m>0$ and $M>0$). If $B_r \subset \widetilde{\Omega}$, $0<r\leq 1$, and $u\in  W^{1,2}(B_r)$ satisfies in the weak sense 
\begin{equation}\label{cpasAharmfctesti}
-{\rm div}(A\nabla u)=0 \quad\text{in $B_r$}\,,
\end{equation}
then $u\in C^1(B_r)$ and  
$$\sup_{B_{r/4}}|\nabla u|^2\leq \frac{C_A}{r^2}\,\dashint_{\partial B_r} |u-\xi|^2\,dx \qquad \forall \xi\in\R\,,$$
for some constant $C_A>0$ depending only on $d$ and the Lipschitz norm of $A$ in $\widetilde{\Omega}$ (i.e.,  not on the radius $r$). 
\end{lemma}

\begin{proof}
Since $u-\xi$ also solves \eqref{cpasAharmfctesti}, we may assume that $\xi=0$. By standard elliptic regularity theory,  $u$ is of class $C^{1,\alpha}$ locally inside $B_r$, and the following estimate holds (see e.g. 
\cite[Theorem~5.19]{GiaqMart})
$$ \sup_{B_{r/4}}|\nabla u|^2 \leq C_A\dashint_{B_{r/2}}|\nabla u|^2\,dx\,.$$
On the other hand, Caccioppoli's inequality (see e.g. \cite[Theorem 4.4]{GiaqMart}) yields 
$$\int_{B_{r/2}}|\nabla u|^2\,dx\leq \frac{C_A}{r^2} \int_{B_r}|u|^2\,dx\,,$$
so that 
\begin{equation}\label{supgradl2vol}
 \sup_{B_{r/4}}|\nabla u|^2 \leq \frac{C_A}{r^2} \dashint_{B_r}|u|^2\,dx\,. 
 \end{equation}
Next we observe that $|u|^2\in W^{1,1}(B_r)$ satisfies (in the $W^{-1,1}$-sense)
\begin{equation}\label{subsolmodusq}
-{\rm div}(A\nabla |u|^2)= -2(A\nabla u)\cdot\nabla u\leq 0\quad\text{in $B_r$}\,.
\end{equation}
According to \cite[Theorem 9.15]{GilbTrud}, there exists a unique strong solution $\varphi$ of 
$$\begin{cases} 
-{\rm div}(A\nabla \varphi)=1 & \text{in $B_r$}\,,\\
\varphi=0 & \text{on $\partial B_r$}\,,
\end{cases}$$
which belongs to $W^{2,p}(B_r)$ for every $p<\infty$. In particular, $\varphi \in C^1(\overline B_r)$ by Sobolev embedding whenever $p>d$, and an elementary scaling argument (using $r\leq 1$) leads to  
\begin{equation}\label{esticapaphir}
\|\nabla\varphi\|_{L^\infty(B_r)}\leq C_A r\,,
\end{equation}
for some constant $C_A>0$ depending only on $d$ and the Lipschitz norm of $A$ in $\widetilde{\Omega}$ (and independent of $r$). Moreover, $\varphi\geq 0$ in $B_r$ by the maximum principle.

Next we write $|u|^2=-|u|^2{\rm div}(A\nabla \varphi)$, and we integrate by parts  over $B_r$ to obtain 
\begin{equation}\label{l2voltoL2surf}
\int_{B_r}|u|^2\,dx=\int_{B_r}(A\nabla|u|^2)\cdot\nabla\varphi\,dx- \int_{\partial B_r}|u|^2(A\nabla\varphi)\cdot\nu\,dx\leq C_A r\int_{\partial B_r}|u|^2\,dx\,, 
\end{equation}
thanks to \eqref{subsolmodusq} and \eqref{esticapaphir}. Gathering \eqref{supgradl2vol} and \eqref{l2voltoL2surf} yields the announced conclusion. 
\end{proof}

\begin{proof}[Proof of Theorem \ref{epsregthm}]
 We start with some useful pointwise identities which hold a.e. in the domain and which allow to perform the so-called H\'{e}lein's trick and rewrite the quadratic term in the right hand side of \eqref{eqepsregthm} in divergence form. 
 
 From the identity $|Q|^2=1$, we first infer that $Q:\partial_k Q=0$ for each $k\in\{1,2,3\}$. As a consequence,  
$$\sum_{k,l=1}^3Q_{kl}(A\nabla Q_{ij})\cdot\nabla Q_{kl} =0 \quad\forall i,j\in\{1,2,3\}\,,$$
which in turn implies that 
$$|\nabla Q|_A^2Q_{ij}= \sum_{k,l=1}^3Q_{ij} (A\nabla Q_{kl})\cdot\nabla Q_{kl}= \sum_{k,l=1}^3B^{kl}_{ij}\cdot\nabla Q_{kl}\,,$$
with the vector fields
\begin{equation}
\label{defBijkl}
B^{kl}_{ij}:= Q_{ij} (A\nabla Q_{kl})-Q_{kl}(A\nabla Q_{ij})  \in L^2(B_{r_0};\mathbb{R}^3) \, ,\quad i,j,k,l\in\{1,2,3\}\,.
\end{equation}
We now claim that in view of the previous pointwise identities for every $i,j,k,l\in\{1,2,3\}$, 
\begin{equation}
\label{divBequation}
{\rm div} \,B_{ij}^{kl}=G_{kl}Q_{ij}-G_{ij}Q_{kl}\quad \text{in $\mathscr{D}^\prime(B_{r_0})$}\,.
\end{equation}
 Indeed, given a test function $\varphi\in \mathscr{D}(B_{r_0})$, we integrate by parts using equation \eqref{eqepsregthm}  to obtain   
\begin{align*}
\int_{B_1} B_{ij}^{kl}\cdot\nabla\varphi\,dx=&\int_{B_1} (A\nabla Q_{kl})\cdot\nabla(Q_{ij}\varphi)\,dx -\int_{B_1} (A\nabla Q_{ij})\cdot\nabla(Q_{kl}\varphi)\,dx\\
=&\int_{B_1} G_{kl}Q_{ij}\varphi\,dx -\int_{B_1} G_{ij}Q_{kl}\varphi\,dx\,,
\end{align*}
and the claim follows. 

We may now write in the sense of distributions
$$B^{kl}_{ij}\cdot\nabla Q_{kl}={\rm div}\big(Q_{kl}B^{kl}_{ij} \big)+Q^2_{kl}G_{ij}-G_{kl}Q_{kl}Q_{ij} \, ,$$
in such a way that for each $i,j\in\{1,2,3\}$
$$-{\rm div}(A\nabla Q_{ij})={\rm div}(Q:B_{ij}) +(Q:Q)G_{ij}-(G:Q)Q_{ij} \quad \text{in $W^{-1,2}(B_{r_0})$}\,, $$
where  $B_{ij}\in L^2(B_{r_0};(\mathcal{S}_0)^3)$ are matrix-valued vector fields given by $B_{ij}:=(B_{ij}^{kl})_{k,l=1}^3$ as defined in \eqref{defBijkl}. 
\vskip3pt

Finally, if $T\in \mathcal{S}_0$ is a constant matrix, we have for every $ i,j\in\{1,2,3\}$, 
\begin{equation}\label{reweqepsregthm}
-{\rm div}(A\nabla Q_{ij})={\rm div}\big((Q-T):B_{ij}\big) +F_{ij} \quad \text{in $W^{-1,2}(B_{r_0})$}\,, 
\end{equation}
with $F_{ij}:=(Q:(Q-T))G_{ij}-(G:(Q-T))Q_{ij}\in L^2(B_{r_0})$. 
\vskip5pt

Let $\sigma\in(0,1/8]$ be a constant to be specified later. 
We fix $x_0\in B_{r_0/2}$ and $t\in(0,r_0/2)$ such that $B_t(x_0)\subset B_{r_0}$, and then arbitrary $\bar x\in B_{\sigma t}(x_0)$ and $r\in (0,t)$ such that $B_{\sigma r}(\bar x)\subset B_{\sigma t}(x_0)$. Note that $B_r(\bar x)\subset B_t(x_0)\subset B_{r_0}$, and thus assumption \eqref{epscondthm} yields   
\begin{equation}\label{newcondepsregthm}
\sup_{0<\rho\leq r}   \left(\frac{1}{\rho}\int_{B_\rho(\bar x)} |\nabla Q|^2\,dx +\rho \int_{B_\rho(\bar x)}|G|^2\,dx\right) \leq \boldsymbol{\eps}_A\,.
\end{equation}
Define 
$$T:=\dashint_{B_{r}(\bar x)}Q\,dx\in\mathcal{S}_0\,.$$
By a standard average argument based on Fubini's theorem, we can find a good radius $\bar r\in(r/2,r)$ for which 
\begin{equation}\label{condgoodrad}
\int_{\partial B_{\bar r}(\bar x)} |Q-T|^2\,d \mathcal{H}^2\leq \frac{4}{r}\int_{B_r(\bar x)}|Q-T|^2\,dx\,.
\end{equation}
Since $Q\in W^{1/2,2}\big(\partial B_{\bar r}(\bar x);\mathcal{S}_0\big)$, there exists a unique $H\in W^{1,2}( B_{\bar r}(\bar x);\mathcal{S}_0)$ satisfying
\begin{equation}\label{Aharmscaleq}
\begin{cases}
-{\rm div}(A\nabla H)=0 & \text{in $B_{\bar r }(\bar x)$}\,,\\
H=Q & \text{on $\partial B_{\bar r }(\bar x)$}\,. 
\end{cases}
\end{equation}
In addition, applying Lemma \ref{lemsupgradAharmfct} with $\widetilde{\Omega}=B_{r_0}$we infer that $H$ belongs to $ C^1(B_{\bar r }(\bar x))$ and that  
\begin{equation}\label{estigradAharmfct}
\sup_{B_{\bar r/4 }(\bar x)}|\nabla H|^2\leq \frac{C_A}{\bar r^2} \dashint_{\partial B_{\bar r}(\bar x)}|H-T|^2\,d\mathcal{H}^2= \frac{C_A}{\bar r^2} \dashint_{\partial B_{\bar r}(\bar x)}|Q-T|^2\,d\mathcal{H}^2\leq \frac{C_A}{r^2} \dashint_{B_{r}(\bar x)}|Q-T|^2\,dx\,,
\end{equation}
thanks to our choice of $\bar r$ made in \eqref{condgoodrad}. 
\vskip3pt

By \eqref{reweqepsregthm} and \eqref{Aharmscaleq}, the map $Q-H$ has components which solve 
$$\begin{cases}
-{\rm div}(A\nabla (Q_{ij}-H_{ij}))={\rm div}\big((Q-T):B_{ij}\big) +F_{ij} & \text{in $W^{-1,2}(B_{\bar r }(\bar x))$}\,,\\
Q_{ij}-H_{ij}=0 & \text{on $\partial B_{\bar r }(\bar x)$}\,,
\end{cases}$$  
and our aim now is to apply Lemma \ref{linearestiepsreg}. To this purpose, let us fix the exponents  
$$q\in(3/2,2)\quad\text{and}\quad s:=\frac{3q}{3+q}\in (1,6/5)\,. $$
(One can choose for instance $q=7/4$.)
Using the identity $|Q|=1$ and H\"older's inequality, we estimate with the help of \eqref{newcondepsregthm}, 
\begin{align*}
\|(Q-T):B_{ij}\|_{L^q(B_{\bar r }(\bar x))}& \leq \|B_{ij}\|_{L^2(B_{\bar r }(\bar x))}\|Q-T\|_{L^{\frac{2q}{2-q}}(B_{\bar r }(\bar x))}\\
&\leq C_A\|\nabla Q\|_{L^2(B_{\bar r }(\bar x))}\|Q-T\|_{L^{\frac{2q}{2-q}}(B_{\bar r }(\bar x))}\\
& \leq C_A(\boldsymbol{\eps}_A \bar{r})^{1/2} \|Q-T\|_{L^{\frac{2q}{2-q}}(B_{\bar r }(\bar x))}\,,
\end{align*}
as well as 
\begin{align*}
\|F_{ij}\|_{L^s(B_{\bar r }(\bar x))}
&\leq  C\|G\|_{L^2(B_{\bar r }(\bar x))}\|Q-T\|_{L^{\frac{6q}{6-q}}(B_{\bar r }(\bar x))}\\
&\leq C(\boldsymbol{\eps}_A /{\bar r})^{1/2}\|Q-T\|_{L^{\frac{6q}{6-q}}(B_{\bar r }(\bar x))} \\
&\leq C (\boldsymbol{\eps}_A {\bar r})^{1/2} \|Q-T\|_{L^{\frac{2q}{2-q}}(B_{\bar r }(\bar x))}\,.
\end{align*}
According to Lemma \ref{linearestiepsreg}, we thus have 
$$\|\nabla(Q-H)\|_{L^q(B_{\bar r }(\bar x))}\leq C_A (\boldsymbol{\eps}_A \bar{r})^{1/2}\|Q-T\|_{L^{\frac{2q}{2-q}}(B_{\bar r }(\bar x))}\,.  $$
Since $\bar{r}\in (r/2,r)$, the previous estimate and the Sobolev inequality in $W^{1,p}_0(B_{\bar r}({\bar x}))$ yield 
\begin{multline}\label{estdistQH}
\left(\dashint_{B_{\bar r }(\bar x)} |Q-H|^p\,dx\right)^{1/p}\leq \frac{C}{{\bar r}^{3/p}} \|\nabla(Q-H)\|_{L^q(B_{\bar r }(\bar x))}\\
\leq C_A\boldsymbol{\eps}_A^{1/2}\left(\dashint_{B_{r}(\bar x)}
|Q-T|^{\frac{2q}{2-q}}\,dx\right)^{\frac{2-q}{2q}}\,,
\end{multline}
where $p:=q^*=\frac{3q}{3-q} >2$  
is the Sobolev exponent. Next we set  
$$\overline H:=\dashint_{B_{\sigma r}(\bar x)} H\,dx\quad \text{and}\quad \overline Q:=\dashint_{B_{\sigma r}(\bar x)} Q\,dx\,, $$
and we infer from \eqref{estigradAharmfct} and H\"older's inequality, as $\bar{r}\in (r/2,r)$ and $\frac{2q}{2-q} >2$, that 
\begin{equation}\label{estdistQbarH}
\left(\dashint_{B_{\sigma r}(\bar x)}|H-\overline H|^p\,dx\right)^{1/p}\leq C\sigma r\sup_{B_{\bar r/4}(\bar x)} |\nabla H|\leq C_A\sigma\left(\dashint_{B_r(\bar x)}|Q-T|^{\frac{2q}{2-q}}\,dx\right)^{\frac{2-q}{2q}}\,. 
\end{equation}
In view of \eqref{estdistQH} and \eqref{estdistQbarH}, as $\bar{r}\in (r/2,r)$ we may now deduce from Minkowski's inequality and the John-Nirenberg inequality in Lemma \ref{Johnirlem} that 
\begin{align}
\nonumber \left(\dashint_{B_{\sigma  r }(\bar x)} |Q-\overline H|^p\,dx\right)^{1/p}& \leq  C\sigma^{-3/p}\left(\dashint_{B_{\bar r }(\bar x)} |Q-H|^p\,dx\right)^{1/p}+\left(\dashint_{B_{\sigma r}(\bar x)}|H-\overline H|^p\,dx\right)^{1/p}\\
\nonumber &\leq  C_A\big(\sigma^{-3/p}\boldsymbol{\eps}_A^{1/2}+ \sigma\big)\left(\dashint_{B_ r(\bar x)}|Q-T|^{\frac{2q}{2-q}}\,dx\right)^{\frac{2-q}{2q}}\\
\label{estdistQbarQ} &\leq  C_A\big(\sigma^{-3/p}\boldsymbol{\eps}_A^{1/2}+ \sigma\big) \|Q\|_{{\rm BMO}(B_t(x_0))}\,.
\end{align}
It now follows from \eqref{estdistQH} and \eqref{estdistQbarQ}   together with H\"older's inequality and  the John-Nirenberg inequality again that 
\begin{align*}
\dashint_{B_{\sigma  r }(\bar x)} |Q-\overline Q|\,dx & \leq \dashint_{B_{\sigma  r }(\bar x)} |Q-\overline H|\,dx + |\overline H-\overline Q|  \\
& \leq \dashint_{B_{\sigma  r }(\bar x)} |Q-\overline H|\,dx +\dashint_{B_{\sigma  r }(\bar x)} |Q-H|\,dx \\
& \leq \left(\dashint_{B_{\sigma  r }(\bar x)} |Q-\overline H|^p\,dx\right)^{1/p}+C\sigma^{-3/p}\left(\dashint_{B_{\bar r }(\bar x)} |Q-H|^p\,dx\right)^{1/p}\\
& \leq C_A\big(\sigma^{-3/p}\boldsymbol{\eps}_A^{1/2}+ \sigma\big) \|Q\|_{{\rm BMO}(B_t(x_0))}\,.
\end{align*}
Finally, taking the supremum over $\bar x$ and $r$, we conclude that 
$$ \|Q\|_{{\rm BMO}(B_{\sigma t}(x_0))}\leq C_A\big(\sigma^{-3/p}\boldsymbol{\eps}_A^{1/2}+ \sigma\big) \|Q\|_{{\rm BMO}(B_t(x_0))}\,.$$
We then choose $\sigma\in(0,1/8]$ and $\boldsymbol{\eps}_A>0$ small enough (depending only on $A$) in such a way that 
$$ \|Q\|_{{\rm BMO}(B_{\sigma t}(x_0))}\leq \frac{1}{2} \|Q\|_{{\rm BMO}(B_t(x_0))}\,.$$
In view of the arbitrariness of $t\in(0,r_0/2)$, the inequality above holds for every $t\in(0,r_0/2)$. A classical iteration argument on the function $t\mapsto\|Q\|_{{\rm BMO}(B_t(x_0))}$ then shows that 
\begin{equation}\label{bmodecay}
\|Q\|_{{\rm BMO}(B_t(x_0))}\leq \|Q\|_{{\rm BMO}(B_{r_0/2}(x_0))} 2^\alpha r_0^{-\alpha}t^\alpha\leq 2^{\alpha+1} r_0^{-\alpha}t^{\alpha} \quad \forall t\in(0,r_0/2)\,, 
\end{equation}
where $\alpha\in(0,1/3)$ is determined by $\sigma^\alpha=1/2$ (note that we have used the fact that $|Q|=1$ in the second inequality). In particular, \eqref{bmodecay} leads to   
$$\dashint_{B_t(x_0)} \Big|Q-\dashint_{B_t(x_0)}Q\,dy \Big|\,dx\leq Cr_0^{-\alpha}t^{\alpha} \qquad \forall t\in(0,r_0/2)\,. $$
In view of the arbitrariness of $x_0\in B_{r_0/2}$, it implies that $Q\in C^{0,\alpha}(\overline B_{r_0/2})$ with the announced estimate by Campanato's criterion, see e.g. \cite[Theorem 6.1]{Maggi}.
\end{proof}

Applying Theorem  \ref{epsregthm} to our main equation \eqref{distribELeq} yields the following interior regularity estimate. 

\begin{corollary}\label{corolholderint}
Let $Q_\lambda\in W^{1,2}(B_{r_0};\mathbb{S}^4)$ be such that 
$$-\Delta Q_\lambda=|\nabla Q_\lambda|^2Q_\lambda+\lambda\Big(Q_\lambda^2-\frac{1}{3}I-{\rm tr}(Q^3_\lambda)Q_\lambda\Big) \quad\text{in $\mathscr{D}^\prime(B_{r_0})$}\,.$$
There exist two universal constants  $\boldsymbol{\eps}_{\rm in}>0$ and ${\bf r}_{\rm in}>0$ such that for every ball $B_r(x_0)\subset B_{r_0}$ of  radius 
$0<r<{\bf r}_{\rm in}(1+\lambda)^{-1/2}$, the condition
$$\sup_{B_\rho(x)\subset B_{r}(x_0)} \frac{1}{\rho}\int_{B_\rho(x)}|\nabla Q_\lambda|^2\,dx \leq \boldsymbol{\eps}_{\rm in}$$
implies $Q_\lambda\in C^{0,\alpha}(\overline B_{r/2}(x_0))$ with $[Q_\lambda]_{C^{0,\alpha}(B_{r/2}(x_0))}\leq Cr^{-\alpha}$ for some constants $\alpha\in(0,1)$ and $C>0$ independent of $\lambda$.  
\end{corollary}

\begin{proof}
Since $Q_\lambda$ is a weak solution of \eqref{MasterEq}, it solves \eqref{eqepsregthm} in $B_r(x_0)$ with the matrix $A=I$, and 
$G:=\lambda\big( Q_\lambda^2 -\frac{1}{3} I - {\rm tr}( Q_\lambda^3)  Q_\lambda\big)$. The map $Q_\lambda$ being $\mathbb{S}^4$-valued, we have 
$$\sup_{B_\rho(x)\subset B_{r}(x_0)} \rho \int_{B_\rho(x)}|G|^2\,dx\leq C \frac{{\bf r}^{4}_{\rm in}\lambda^2}{(1+\lambda)^{2}} \leq C{\bf r}^{4}_{\rm in}\,,  $$
for some universal constant $C>0$. Hence, we can choose $\boldsymbol{\eps}_{\rm in}$ and ${\bf r}_{\rm in}$ small enough in such a way that \eqref{epscondthm} holds 
(with $\boldsymbol{\eps}_A=\boldsymbol{\eps}_I$), and the conclusion follows from Theorem~\ref{epsregthm}. 
\end{proof}

Concerning boundary regularity estimates under a Dirichlet boundary condition, we apply the refection procedure of the previous subsection, and then Theorem  \ref{epsregthm} to equation \eqref{ELeqAfterReflect}.

\begin{corollary}\label{corolholderbdry}
Assume that $\partial\Omega$ is of class $C^3$ and $Q_{\rm b}\in C^{1,1}(\partial\Omega;\mathbb{S}^4)$. Let $Q_\lambda\in \mathcal{A}_{Q_{\rm b}}(\Omega)$ be a critical point of $\mathcal{E}_\lambda$, and $\widehat Q_\lambda$ its extension to $\widehat\Omega$ given by \eqref{extprocedure}. There exist two constants $\boldsymbol{\eps}_{\rm bd}>0$ and ${\bf r}_{\rm bd}>0$ depending only on $\Omega$ and $Q_{\rm b}$ such that for every ball $B_r(x_0)\subset\widehat \Omega$ with $x_0\in\partial\Omega$ and $0<r<{\bf r}_{\rm bd}(1+\lambda)^{-1/2}$, the condition
 $$\sup_{B_\rho(x)\subset B_{r}(x_0)} \frac{1}{\rho}\int_{B_\rho(x)}|\nabla \widehat Q_\lambda|^2\,dx\leq \boldsymbol{\eps}_{\rm bd}$$
 implies $\widehat Q_\lambda\in C^{0,\alpha}(\overline B_{r/2}(x_0))$ with $[\widehat Q_\lambda]_{C^{0,\alpha}(B_{r/2}(x_0))}\leq C_{Q_{\rm b}} r^{-\alpha}$ for some constants $\alpha\in(0,1)$ and $C_{Q_{\rm b}} >0$ depending only on $\Omega$ and~$Q_{\rm b}$ (and not on $\lambda$). 
\end{corollary}

\begin{proof}
By Proposition \ref{ELeqExt}, $\widehat Q_\lambda$ solves \eqref{eqepsregthm} in $B_r(x_0)$ with the matrix field $A$ given by \eqref{matrixfieldextform}, and the map $G$ given by 
$G:=G_\lambda(\cdot,\widehat Q_\lambda,\nabla \widehat Q_\lambda)$  where $G_\lambda$ satisfies the growth condition \eqref{growthGlambdaext}. In particular,
\begin{multline*}
\sup_{B_\rho(x)\subset B_{r}(x_0)} \rho \int_{B_\rho(x)}|G|^2\,dx\leq  C_{Q_{\rm b}} \sup_{B_\rho(x)\subset B_{r}(x_0)} \rho \int_{B_\rho(x)}\big((1+\lambda)^2+|\nabla\widehat Q_\lambda|^2\big)\,dx\\
 \leq  C_{Q_{\rm b}} {\bf r}_{\rm bd}^2 \big( {\bf r}_{\rm bd}^2+ \boldsymbol{\eps}_{\rm bd}\big)\,,
 \end{multline*}
 for a constant $C_{Q_{\rm b}} >0$ depending only on $\Omega$ and~$Q_{\rm b}$. Hence, we can choose $\boldsymbol{\eps}_{\rm bd}$ and ${\bf r}_{\rm bd}$ small enough in such a way that \eqref{epscondthm} holds, and the conclusion follows from Theorem~\ref{epsregthm}. 
\end{proof}

\subsection{Higher order regularity}\label{sechigherreg}
In this subsection, we improve H\"{o}lder continuity estimates from the previous one into Lipschitz estimates. Finally, we deduce analytic regularity both in the interior and at the boundary, whenever boundary data permit.
\begin{proposition}\label{propLipcont}
Let $r\in (0,1]$ and let $A:B_{r}\to\mathscr{M}^{\rm sym}_{3\times3}(\R)$ be a Lipschitz field of  symmetric matrices. Assume that $A$ is uniformly elliptic, i.e., $m I\leq A\leq M I$ for some constants $m>0$ and $M>1$. Let $G:B_{r}\times \mathbb{S}^4 \times (\mathcal{S}_0)^3\to\mathcal{S}_0$ be a Carath\'eodory map satisfying
\begin{equation}\label{hypGlipprop}
|G(x,q,\xi)|\leq C_*(\Lambda+|\xi|^2) \qquad \forall (x,q,\xi)\in B_{r}\times \mathbb{S}^4 \times (\mathcal{S}_0)^3\,,
\end{equation}
for some constants $\Lambda>0$ and $C_*>0$. 
Let $Q\in W^{1,2}(B_{r};\mathbb{S}^4)$ be such that 
$$-{\rm div}(A\nabla Q)=G(x,Q,\nabla Q) \quad \text{in $\mathscr{D}^\prime(B_{r})$}\,.$$
If $Q\in C^{0,\alpha}(\overline B_{r})$ for some $\alpha\in(0,1)$ and $[Q]_{C^{0,\alpha}(B_{r})}\leq \boldsymbol{\kappa} r^{-\alpha}$, then $Q\in W^{1,\infty}(B_{r/2})$ and 
$$r^2\|\nabla Q\|^2_{L^\infty(B_{r/2})}\leq C\bigg(\frac{1}{r}\int_{B_{r}}|\nabla Q|^2\,dx +\Lambda r^2\bigg)\,, $$
for some constant $C>0$ depending only on $\|A\|_{Lip(B_r)}$, $m$, $M$, $C_*$, $\alpha$, and ${\boldsymbol \kappa}$. 
\end{proposition}

\begin{proof}
 Let us fix an arbitrary point $x_0\in B_{r/2}$, and set $A_0:=A(x_0)$, $r_1:=r/(2\sqrt{M})<1$. We change variables by setting for $x\in B_{r_1}$ (so that $A_0^{1/2}x+x_0\in B_{r/2}(x_0)$), 
$$\bar Q(x):=Q\bigg(A_0^{1/2}x+x_0\bigg)\,. $$
Then $\bar Q\in W^{1,2}(B_{r_1};\mathbb{S}^4)\cap C^{0,\alpha}(\overline B_{r_1})$ satisfies $[\bar Q]_{C^{0,\alpha}(B_{r_1})}\leq M^{\alpha/2}\boldsymbol{\kappa} r_1^{-\alpha}$, and it solves 
\begin{equation}\label{eqgenchangvar}
-{\rm div}\big(\bar A\nabla \bar Q\big)= \bar G(x,\bar Q,\nabla \bar Q)  \quad \text{in $\mathscr{D}^\prime(B_{r_1})$}\,,
\end{equation}
with 
$$\bar A(x):=A_0^{-1/2}A\big(A_0^{1/2}x+x_0\big)A_0^{-1/2}$$
and 
$$\bar G(x,q,\xi) := G\big(A_0^{1/2}x+x_0,q,A_0^{-1/2}\xi\big)\,. $$
We observe that $\bar A$ is Lipschitz continuous in $B_{r_1}$,  and 
$$\frac{m}{M}\, I\leq \bar A \leq \frac{M}{m}I\quad\text{and}\quad \bar A(0)=I\,. $$
Concerning $\bar G$, it satisfies
\begin{equation}\label{controlbarG}
 |\bar G(x,q,\xi)|\leq \widetilde C_*(\Lambda+|\xi|^2) \qquad \forall (x,q,\xi)\in B_{r_1}\times \mathbb{S}^4 \times (\mathcal{S}_0)^3\,,
 \end{equation}
for some constant $\widetilde C_*>0$ depending only on $C_*$ and $A$. 
\vskip3pt

We now fix an arbitrary radius $\rho\in(0,r_1]$, and we consider $H\in W^{1,2}(B_\rho;\mathcal{S}_0)\cap C^0(\overline B_\rho)$ the (unique) solution of 
$$\begin{cases}
-\Delta H=0 & \text{in $B_\rho$}\,,\\
H=\bar Q & \text{in $\partial B_\rho$}\,.
\end{cases}$$
Representing $H$ through the Poisson integral formula, one easily obtains  
$$\mathop{{\rm osc}}\limits_{B_\rho} H=  \mathop{{\rm osc}}\limits_{\partial B_\rho}\bar Q\leq Cr_1^{-\alpha}\rho^\alpha\,,$$
for some constant $C>0$ depending only $A$ and ${\boldsymbol \kappa}$ (and ${\rm osc}$ is meant for oscillation).  Since $H-\bar Q=0$ on $\partial B_\rho$, we deduce that  
\begin{equation}\label{estunifharmreplac}
\sup_{B_\rho}|\bar Q-H|\leq  \mathop{{\rm osc}}\limits_{B_\rho}  \bar Q +  \mathop{{\rm osc}}\limits_{B_\rho} H \leq Cr_1^{-\alpha}\rho^\alpha\,,
\end{equation}
with  $C>0$ depending only $A$ and ${\boldsymbol \kappa}$. 

On the other hand, concerning the harmonic function $H$, we have $H\in C^\infty(B_\rho)$ and also $\Delta |\nabla H|^2=2 |D^2 H|^2 \geq 0$. Hence the function $\rho \to \rho^{-2} \int_{|x|=\rho} |\nabla H|^2 d\mathcal{H}^2$ is nondecreasing, and in turn $\rho \to \rho^{-3} \int_{B_\rho} |\nabla H|^2 dx$ is nondecreasing as well. As a consequence,  since $H$ is equal to $\bar Q$ on $\partial B_\rho$, it satisfies 
\begin{equation}\label{monotharmfunct}
\dashint_{B_{\rho^\prime}} |\nabla H|^2\,dx\leq \dashint_{B_{\rho}} |\nabla H|^2\,dx\leq \dashint_{B_{\rho}} |\nabla \bar Q|^2\,dx \qquad\forall \rho^\prime\in(0,\rho)\,.
\end{equation}

We are now ready to estimate 
\begin{equation}\label{minklipest}
\left(\dashint_{B_{\rho/2}}|\nabla \bar Q|^2_{\bar A}\,dx\right)^{1/2}\leq \left(\dashint_{B_{\rho/2}}|\nabla H|^2_{\bar A}\,dx\right)^{1/2}+C\left(\dashint_{B_{\rho}}|\nabla(\bar Q-H)|^2_{\bar A}\,dx\right)^{1/2}=:I^{1/2}+CII^{1/2}\,, 
\end{equation}
and we shall treat separately the two terms $I$ and $II$. Since $A$ is Lipschitz and $\bar A(0)=I$, we have $|\bar A - I|\leq C_A\rho$ in $B_\rho$, and we infer from \eqref{monotharmfunct} that
\[
I\leq (1+C_A\rho)\dashint_{B_{\rho/2}}|\nabla H|^2\,dx\leq (1+C_A\rho)\dashint_{B_{\rho}}|\nabla H|^2\,dx\leq (1+C_Ar_1^{-\alpha}\rho^{\alpha})\dashint_{B_{\rho}}|\nabla \bar Q|^2\,dx\,, 
\]
where we have used that $0<\rho \leq r_1\leq 1$. Using again this property together with the ellipticity bounds on $A$ and $|\bar A - I|\leq C_A\rho$ in $B_\rho$ we conclude, 
\begin{equation}\label{estiIIlip1}
\sqrt{I}\leq (1+C_Ar_1^{-\alpha/2}\rho^{\alpha/2}) \left(\dashint_{B_{\rho}}|\nabla \bar Q|_{\bar A}^2\,dx\, \right)^{1/2} \, . 
\end{equation} 
Next we write
\begin{equation}\label{decoupleIIlip}
II= \dashint_{B_{\rho}}\langle \nabla \bar Q, \nabla(\bar Q-H)\rangle_{\bar A}\,dx + \dashint_{B_{\rho}}\langle \nabla  H,\nabla (H-\bar Q)\rangle_{\bar A}\,dx\,.
\end{equation}
Since $\bar Q-H\in W^{1,2}_0(B_\rho)\cap L^\infty$, we can apply \eqref{eqgenchangvar} and then deduce from \eqref{controlbarG} and \eqref{estunifharmreplac} that 
\begin{equation}\label{useeqlipest}
\dashint_{B_{\rho}}\langle \nabla \bar Q,
\nabla (\bar Q-H)\rangle_{\bar A}\,dx =  \dashint_{B_{\rho}} \bar G(x,\bar Q,\nabla\bar Q):(\bar Q-H)\,dx\leq Cr_1^{-\alpha}\rho^{\alpha}\left(\dashint_{B_\rho} |\nabla \bar Q|^2\,dx+\Lambda\right)\,.
\end{equation}
Since $H$ is harmonic and $\bar Q-H=0$ on $\partial B_\rho$, we have $\int_{B_\rho}\nabla H : \nabla(\bar Q-H)\,dx=0$, and consequently  
\begin{multline}\label{useeqlipest2}
\dashint_{B_{\rho}}\langle \nabla  H,\nabla (H-\bar Q)\rangle_{\bar A}\,dx\leq \dashint_{B_{\rho}} |\bar A-I |\,|\nabla  H| |\nabla(H-\bar Q)|\,dx\\
\leq C_A \rho\left(  \dashint_{B_{\rho}} |\nabla  H|^2\,dx+ \dashint_{B_{\rho}} |\nabla\bar Q|^2\,dx\right)\leq Cr_1^{-\alpha}\rho^{\alpha} \dashint_{B_{\rho}} |\nabla\bar Q|^2\,dx\,,
\end{multline}
where we have used again $|\bar A -I|\leq C_A\rho$ in $B_\rho$, \eqref{monotharmfunct}, and $0<\rho\leq r_1\leq 1$. Combining now \eqref{decoupleIIlip}, \eqref{useeqlipest}, and \eqref{useeqlipest2}  
leads to 
\[
II\leq C_Ar_1^{-\alpha}\rho^\alpha \left(\dashint_{B_{\rho}} |\nabla\bar Q|^2\,dx+\Lambda\right)\,.
\]
As $0<\rho\leq r_1\leq 1$, in view of the ellipticity bounds of $A$ and $|\bar A -I|\leq C_A\rho$ in $B_\rho$ we conclude
\begin{equation}\label{estiIIlip2}
\sqrt{II}\leq C_Ar_1^{-\alpha/2}\rho^{\alpha/2} \left(\dashint_{B_{\rho}} |\nabla\bar Q|_{\bar A}^2\,dx+\Lambda\right)^{1/2}\,.
\end{equation}

Combining \eqref{minklipest} with \eqref{estiIIlip1} and \eqref{estiIIlip2}, we obtain
\begin{equation}\label{decrlipestim}
\left(\dashint_{B_{\rho/2}}|\nabla \bar Q|^2_{\bar A}\,dx \right)^{1/2}\leq \big(1+C_Ar_1^{-\alpha/2}\rho^{\alpha/2}\big) \left(\dashint_{B_{\rho}}|\nabla \bar Q|^2_{\bar A}\,dx \right)^{1/2}+ C_A\sqrt{\Lambda} r_1^{-\alpha/2}\rho^{\alpha/2}\,,
\end{equation}
for a constant $C_A>0$ depending only on $A$, $C_*$, and ${\boldsymbol \kappa}$ and for all $0<\rho\leq r_1\leq 1$.
\vskip3pt

In view of the arbitrariness of $\rho$, we can apply \eqref{decrlipestim} with $\rho_k:=2^{-k}r_1$ and $k\in\mathbb{N}$. It leads to 
$$\left(\dashint_{B_{\rho_{k+1}}}|\nabla \bar Q|^2_{\bar A}\,dx \right)^{1/2}\leq \big(1+C_A2^{-\alpha k/2}\big)\left(\dashint_{B_{\rho_k}}|\nabla \bar Q|^2_{\bar A}\,dx\right)^{1/2}+ C_A\sqrt{\Lambda} 2^{-\alpha k/2}\quad\forall k\in\mathbb{N}\,.$$
Now if $\{\theta_k\} \subset (1,\infty)$, $\theta=\Pi_{k=0}^\infty \theta_k<\infty$, $\{\sigma_k\} \subset (0,\infty)$, $\sigma=\Sigma_{k=0}^\infty \sigma_k<\infty$, and $\{y_k\} \subset [0,\infty)$ satisfy $y_{k+1}\leq \theta_k y_k+\sigma_k$ for each $k\geq 0$, then a simple induction argument gives $y_{k+1}\leq \theta (y_0+\sigma)$ for each $k\geq 0$. As a consequence, if we let
\[ y_k= \left(\dashint_{B_{\rho_k}}|\nabla \bar Q|^2_{\bar A}\,dx\right)^{1/2} \, , \quad \theta_k=\big(1+C_A2^{-\alpha k/2}\big) \, , \quad \sigma_k= C_A\sqrt{\Lambda} 2^{-\alpha k/2} \, , \]
then we obtain
\begin{equation}\label{leblipest}
\left(\dashint_{B_{\rho_{k}}}|\nabla \bar Q|^2_{\bar A}\,dx\right)^{1/2}\leq C\left[ \left(\dashint_{B_{r_1}}|\nabla \bar Q|^2_{\bar A}\,dx\right)^{1/2}+ \sqrt{\Lambda} \right] \quad\forall k\in\mathbb{N}\,,
\end{equation} 
for some constant $C>0$ depending only on $A$, $C_*$,  ${\boldsymbol \kappa}$, and $\alpha$.

Finally, if $x_0$ was chosen to be a Lebesgue point of $|\nabla Q|^2$ (which holds for a.e. $x_0\in B_{r_0/2}$ by the Lebesgue differentiation theorem), then $0$ is a Lebesgue point for $|\nabla \bar Q|^2_{\bar A}$, and letting $k\to\infty$ in \eqref{leblipest} yields (recall that $\bar A(0)=I$)
$$|\nabla \bar Q(0)|^2\leq C\left(\dashint_{B_{r_1}}|\nabla \bar Q|^2_{\bar A}\,dx+\Lambda\right)\,. $$
Changing variables again and using the uniform ellipticity of $A$, we deduce from the definition of $r_1$ that
$$|\nabla Q(x_0)|^2\leq  C'\left(\frac{1}{r_1^3}\int_{B_{r/2}(x_0)}|\nabla Q|^2\,dx+\Lambda\right)\leq C\left(\frac{1}{r^3}\int_{B_{r}}|\nabla Q|^2\,dx+\Lambda\right)\,,$$
 for some constants $C>0$ and $\Lambda>0$ depending only on $A$, $C_*$,  ${\boldsymbol \kappa}$, and $\alpha$ and the conclusion follows. 
\end{proof}
Once Lipschitz continuity is obtained, one can derive higher regularity from linear elliptic theory.
\begin{corollary}
\label{corolhigherreg}
Let $Q_\lambda\in W^{1,2}(B_{r}(x_0);\mathbb{S}^4)$ be such that 
$$-\Delta Q_\lambda=|\nabla Q_\lambda|^2Q_\lambda+\lambda\Big(Q_\lambda^2-\frac{1}{3}I-{\rm tr}(Q^3_\lambda)Q_\lambda\Big) \quad\text{in $\mathscr{D}^\prime(B_{r}(x_0))$}\,.$$
If $0<r<{\bf r}_{\rm in}(1+\lambda)^{-{1/2}}$ and 
 $$\sup_{B_\rho(x)\subset B_{r}(x_0)} \frac{1}{\rho}\int_{B_\rho(x)}|\nabla  Q_\lambda|^2\,dx\leq \boldsymbol{\eps}_{\rm in}\,,$$
 where  ${\bf r}_{\rm in}$ and $\boldsymbol{\eps}_{\rm in}$ are given by Corollary \ref{corolholderint}, then $Q_\lambda\in C^\omega(B_{r/4}(x_0))$. 
In  addition, $Q_\lambda$ satisfies  for each $k\in\mathbb{N}$, 
\begin{equation}\label{controlatcenterallderiv}
\|\nabla^k Q_\lambda\|_{L^\infty(B_{r/8}(x_0))}\leq C_k r^{-k}\,,
 \end{equation}
for a  constant $C_k>0$ depending only on $k$.
\end{corollary}

\begin{proof} {\it Step 1.} 
By Corollary \ref{corolholderint}, $Q_\lambda\in C^{0,\alpha}(B_{r/2}(x_0))$ with $[Q_\lambda]_{C^{0,\alpha}(B_{r/2}(x_0))}\leq Cr^{-\alpha}$ for some $\alpha\in(0,1)$ and $C>0$ independent of $\lambda$. Applying Proposition \ref{propLipcont} with $A=I$ and 
$$G(x,Q,\nabla Q):=|\nabla Q|^2Q+\lambda\big( Q^2 -\frac{1}{3} I - {\rm tr}( Q^3)  Q\big)$$ 
(so that $G$ satisfies \eqref{hypGlipprop} with $\Lambda:=\lambda+1$) yields $Q_\lambda\in W^{1,\infty}(B_{r/4}(x_0))$ and 
\begin{multline*}
 r^2\|\nabla Q_\lambda\|^2_{L^\infty(B_{r/4}(x_0))} \leq C\bigg(\frac{1}{r}\int_{B_{r}(x_0)}|\nabla Q_\lambda|^2\,dx +(1+\lambda)r^2\bigg)\\
 \leq C\bigg(\frac{1}{r}\int_{B_{r}(x_0)}|\nabla Q_\lambda|^2\,dx +1\bigg)\leq C\,,
 \end{multline*}
for some universal constant $C>0$. As a consequence, we have $\Delta Q_\lambda \in L^\infty(B_{r/4}(x_0))$. By linear elliptic regularity theory (see e.g. \cite[Theorem 3.13]{HanLin}), it follows that $Q_\lambda\in C^{1,\alpha}_{\rm loc}(B_{r/4}(x_0))$ for every $\alpha\in(0,1)$. A classical bootstrap argument based on Schauder estimates then shows that  $Q_\lambda\in C^\infty(B_{r/4}(x_0))$ (see e.g. \cite[Chapters 6 \& 8]{GilbTrud}), and standard results in \cite[Chapter 6]{Morrey}  give analytic regularity. 
\vskip5pt

\noindent{\it Step 2.} In this second step, our aim is to prove the remaining estimate \eqref{controlatcenterallderiv} for $k\geq 2$. Let us fix a point $y\in B_{r/8}(x_0)$, and 
rescale variables setting $\widetilde Q(x):=Q_\lambda(y+r x)$. Then, 
\begin{equation}\label{rescaleqestimderiv}
-\Delta \widetilde Q=|\nabla \widetilde Q|^2\widetilde Q+\widetilde \lambda\Big(\widetilde Q^2-\frac{1}{3}I-{\rm tr}(\widetilde Q^3)\widetilde Q\Big) \quad\text{in $B_{1/8}$}\,,
\end{equation}
with $\widetilde\lambda:=r^2\lambda\in(0,{\bf r}_{\rm in})$. Let us fix $j\in\{1,2,3\}$, and set $v:=\partial_j\widetilde Q$. Differentiating \eqref{rescaleqestimderiv} with respect to the $j$-th variable, we obtain that $v$ satisfies a linear system of the form  
$$-\Delta v+b\cdot\nabla v + c\cdot v=d\quad\text{in $B_{1/8}$}\,,$$
where the coefficients $b$, $c$, and $d$ satisfy 
$$\|b\|_{L^\infty(B_{1/8})}+\|c\|_{L^\infty(B_{1/8})}+\|d\|_{L^\infty(B_{1/8})}\leq C$$  
since $|\widetilde Q|=1$ and $\|\nabla \widetilde Q\|_{L^\infty(B_{1/8})}\leq C$. By elliptic regularity (see e.g. \cite[Chapter 8, Section~8.11]{GilbTrud}), $v$ satisfies the estimate 
$$\sup_{B_{1/16}}|\nabla v|\leq C\big(\|v\|_{L^\infty(B_{1/8})}+ \|d\|_{L^\infty(B_{1/8})}\big)\leq C\,. $$
From the arbitrariness of $j$, we conclude that $\|\nabla^2 \widetilde Q\|_{L^\infty(B_{1/16})}\leq C$. Now we can proceed by induction on $k$ following the same strategy (differentiating $(k-1)$-times equation \eqref{rescaleqestimderiv}) to prove that $\|\nabla^k \widetilde Q\|_{L^\infty(B_{2^{-(k+2)}})}\leq C_k$ for a constant $C_k$ depending only on $k$. Scaling variables back, we obtain that $|\nabla^kQ_\lambda(y)|\leq C_kr^{-k}$, 
and \eqref{controlatcenterallderiv}  follows from the arbitrariness of $y$.
\end{proof}

A similar argument then yields higher regularity near the boundary when the boundary data are sufficiently regular.
\begin{corollary}\label{corolhigherregdbry}
Assume that $\partial\Omega$ is of class $C^3$ and $Q_{\rm b}\in C^{1,1}(\partial\Omega;\mathbb{S}^4)$. 
Let $Q_\lambda\in \mathcal{A}_{Q_{\rm b}}(\Omega)$ be a critical point of $\mathcal{E}_\lambda$,  $\widehat Q_\lambda$ its extension to $\widehat\Omega$ given by \eqref{extprocedure}, and $B_r(x_0)\subset\widehat \Omega$ with $x_0\in\partial\Omega$. If $0<r<{\bf r}_{\rm bd}(1+\lambda)^{-1/2}$ and 
 $$\sup_{B_\rho(x)\subset B_{r}(x_0)} \frac{1}{\rho}\int_{B_\rho(x)}|\nabla \widehat Q_\lambda|^2\,dx\leq \boldsymbol{\eps}_{\rm bd}\,,$$
 where  ${\bf r}_{\rm bd}$ and $\boldsymbol{\eps}_{\rm bd}$ are given by Corollary \ref{corolholderbdry},
 then $\|\nabla \widehat Q_\lambda\|_{L^\infty(B_{r/4}(x_0))}\leq C_{Q_{\rm b}}r^{-1}$ for some constant $C_{Q_{\rm b}}>0$ depending only on $\Omega$ and $Q_{\rm b}$. As a consequence $Q_\lambda\in C^\omega(B_{r/4}(x_0)\cap\Omega)\cap C_{\rm loc}^{1,\alpha}(B_{r/4}(x_0)\cap\overline\Omega)$ for every $\alpha\in(0,1)$.
 
In addition, 
\begin{enumerate}
\item[(i)] if $\partial \Omega$ is of class $C^{k,\beta}$ and $Q_{\rm b}\in C^{k,\beta}(\partial\Omega;\mathbb{S}^4)$ with $k\geq 2$, then $Q_\lambda\in C_{\rm loc}^{k,\beta}(B_{r/4}(x_0)\cap\overline\Omega)$; 
\item[(ii)] if $\partial \Omega$ is real-analytic and $Q_{\rm b}\in C^{\omega}(\partial\Omega;\mathbb{S}^4)$, then $Q_\lambda\in C^{\omega}(B_{r/4}(x_0)\cap\overline\Omega)$. 
\end{enumerate}
\end{corollary}

\begin{proof}
By Corollary \ref{corolholderbdry}, $\widehat Q_\lambda\in C^{0,\alpha}(B_{r/2}(x_0))$  with $[\widehat Q_\lambda]_{C^{0,\alpha}(B_{r/2}(x_0))}\leq C_{Q_{\rm b}} r^{-\alpha}$ for some exponent $\alpha\in(0,1)$ and a constant $C_{Q_{\rm b}}>0$ independent of $\lambda$. By Proposition \ref{ELeqExt}, we can apply Proposition \ref{propLipcont}  with the matrix field $A$ given by \eqref{matrixfieldextform}, and $G(x,Q,\nabla Q)$ given by the right-hand side of \eqref{ELeqAfterReflect}  (once again, $G$ satisfies \eqref{hypGlipprop} with $\Lambda:=\lambda+1$). It yields $\widehat Q_\lambda\in W^{1,\infty}(B_{r/4}(x_0))$ and $ r^2\|\nabla \widehat Q_\lambda\|^2_{L^\infty(B_{r/4}(x_0))} \leq C_{Q_{\rm b}}$ (as in the proof of Corollary \ref{corolhigherreg}, Step 1). From the equation \eqref{ELeqAfterReflect} satisfied by $\widehat Q_\lambda$, we deduce that ${\rm div}(A\nabla\widehat Q_\lambda)\in L^\infty(B_{r/4}(x_0))$. By elliptic regularity (see e.g. \cite[Theorem 3.13]{HanLin}), it implies that 
$\widehat Q_\lambda\in C^{1,\alpha}_{\rm loc}(B_{r/4}(x_0))$ for every $\alpha\in(0,1)$, and consequently $Q_\lambda\in C_{\rm loc}^{1,\alpha}(B_{r/4}(x_0)\cap\overline\Omega)$ for every $\alpha\in(0,1)$.  
Since $|\nabla Q_\lambda|\in L^\infty(B_{r/4}(x_0)\cap\Omega)$, we can argue as in the proof of Corollary \ref{corolhigherreg}, Step 1, to show that $Q_\lambda\in C^{\omega}(B_{r/4}(x_0)\cap\Omega)$. 

Finally, under the assumption that $\partial\Omega$ is of class $C^{k,\beta}$ and $Q_{\rm b}\in C^{k,\beta}(\partial\Omega;\mathbb{S}^4)$ with $k\geq 2$, the fact that  $Q_\lambda\in C^{k,\beta}_{\rm loc}(B_{r/4}(x_0)\cap\overline\Omega)$ now follows from equation \eqref{distribELeq} and standard elliptic regularity at the boundary, see e.g. \cite[Chapter 6]{GilbTrud}. The corresponding conclusion within the analytic class follows again from the results in e.g. \cite[Chapter 6]{Morrey}. 
\end{proof}

\subsection{Bochner inequality and uniform regularity estimates}
In this subsection, we refine the previous analysis and clarify the dependence of the regularity estimates for the smooth solutions $Q_\lambda$ of \eqref{MasterEq} on the parameter $\lambda$. The results of this subsection are not used in the present paper but they will be a fundamental tool in the subsequent paper \cite{DMP2} of our series where we will study (axially symmetric) minimizers in the asymptotic limit $\lambda \to \infty$.

\begin{proposition}\label{epsregunifesti}
Let $Q_\lambda\in W^{1,2}(B_{r};\mathbb{S}^4)$ be a smooth solution of \eqref{MasterEq}  in $B_{r}$. There exists a universal constant $\boldsymbol{\eps}_{\rm reg}>0$ such that the condition 
$$\frac{1}{r}\mathcal{E}_\lambda(Q_\lambda,B_{r})\leq  \boldsymbol{\eps}_{\rm reg}$$
implies 
$$\sup_{B_{r/4}}\bigg(\frac{1}{2}|\nabla Q_\lambda|^2+\lambda W(Q_\lambda)\bigg)\leq  C r^{-2} \,,$$
for a further universal constant $C>0$. 
\end{proposition}

The result presented in Proposition \ref{epsregunifesti} is reminiscent from Ginzburg-Landau theories, where the main ingredient is a Bochner type inequality on the energy density in the spirit of the classical inequality for harmonic maps (see e.g. \cite{ChenStruwe,Sch}). In general, Bochner inequality is available as soon as the potential  is not degenerate and the solution of the GL-equation under consideration takes values in a sufficiently small neighborhood of the well. In the Landau-de Gennes context, this is precisely the path adopted in \cite[Lemma 6]{MaZa}. In our context, half way between Landau-de Gennes and harmonic maps, we are able to prove a new {\sl global Bochner inequality} with no restrictions, which is the main ingredient in the proof of Proposition \ref{epsregunifesti}.

\begin{lemma}[Bochner inequality]\label{Bochineq}
Let $Q_\lambda$ be a smooth solution of \eqref{MasterEq}  in $B_{r}$. Setting ${\rm e}_\lambda:=\frac{1}{2}|\nabla Q_\lambda|^2+\lambda W(Q_\lambda)$, we have 
$$-\Delta {\rm e}_\lambda\leq Ce^2_\lambda \quad\text{in $B_{r}$}$$
for some universal constant $C>0$.
\end{lemma}

To prove  Lemma \ref{Bochineq}, we first need to establish the following elementary but quite tricky (new) estimate. 

\begin{lemma}\label{lemmaspectrboch}
There exists a universal constant ${\bf c}_\star>0$ such that for every $Q\in\mathbb{S}^4$ and $T\in \mathcal{S}_0$ satisfying $T:Q=0$,  
$$ 2\,{\rm tr}(TQT) \leq  \bigg(\frac{1}{\sqrt{6}}+ {\bf c}_\star \sqrt{W(Q)}\bigg)|T|^2\,.$$
\end{lemma}

\begin{proof}
Let $\mu_3\leq \mu_2\leq \mu_1$ be the eigenvalues of $Q$. Using that $\mu_1+\mu_2+\mu_3={\rm tr}(Q)=0$ and $\mu^2_1+\mu^2_2+\mu^2_3=|Q|^2=1$, we deduce that 
$0<\mu_1\leq \frac{2}{\sqrt{6}}$ and $-\frac{2}{\sqrt{6}}\leq \mu_3<0$. We now consider  a matrix $P\in SO(3)$ such that $Q=PDP^\trans$ with $D={\rm diag}(\mu_1,\mu_2,\mu_3)\in\mathbb{S}^4$. Setting $\widetilde T:=P^\trans TP$, we observe that $\widetilde T:D=T:Q=0$, $|\widetilde T|=|T|$, ${\rm tr}(\widetilde TD\widetilde T)={\rm tr}(TQT)$, and $W(Q)=W(D)$.  Hence, it suffices to show that 
\begin{equation}\label{spectrineqmatr}
2\,{\rm tr}(\widetilde TD\widetilde T) \leq  \bigg(\frac{1}{\sqrt{6}}+ {\bf c}_\star \sqrt{W(D)}\bigg)|\widetilde T|^2\,,
\end{equation}
for some universal constant ${\bf c}_\star>0$, i.e., that the claim holds when $Q=D$ is a diagonal matrix. To this purpose, let us first recall that 
$$W(D)=0 \quad\Longleftrightarrow\quad \mu_2=\mu_3\quad \Longleftrightarrow\quad \mu_1=\frac{2}{\sqrt{6}}\;\text{ and }\;\mu_2=\mu_3=\frac{-1}{\sqrt{6}}\,.$$
Let us fix a small constant $0<{\bf t}_0<1$ to be choosen later, and set 
$$\boldsymbol{\ell}_0:= \min\bigg\{ W\big({\rm diag}(\nu_1,\nu_2,\nu_3)\big) : \nu_1\geq \nu_2\geq \nu_3+{\bf t}_0\,,\;\nu_1+\nu_2+\nu_3=0\,,\nu_1^2+\nu_2^2+\nu_3^2=1 \bigg\}>0\,.$$
If $\mu_2-\mu_3\geq {\bf t}_0$, then \eqref{spectrineqmatr} clearly holds for ${\bf c}_\star\geq 2\boldsymbol{\ell}_0^{-1/2}$ since $|D|=1$. Hence it remains to prove the inequality in the case $\mu_2-\mu_3< {\bf t}_0$. To this purpose let us set 
$t:=\mu_2-\mu_3\in[0, {\bf t}_0)$. Choosing ${\bf t}_0$ small enough ensures that $\mu_2<0$, and direct computations yield 
$$\mu_1=\frac{2}{\sqrt{6}}(1-t^2/2)^{1/2}\,,\quad\mu_2=\frac{t}{2}-\frac{1}{\sqrt{6}}(1-t^2/2)^{1/2} \,,\quad \mu_3=-\frac{t}{2}-\frac{1}{\sqrt{6}}(1-t^2/2)^{1/2}\,,$$
and, as $t\to 0$,  
$$W(D)=\frac{1-(1-t^2/2)^{3/2}}{3\sqrt{6}} +\frac{t^2}{2\sqrt{6}} (1-t^2/2)^{1/2}=\frac{3}{4\sqrt{6}}t^2+o(t^2)\,.$$
In particular, if ${\bf t}_0$ is sufficiently small, then $t\in[0, {\bf t}_0)$ yields
$$\sqrt{W(D)}\geq \frac{t}{4}\,. $$
Let us now write 
$$\widetilde T= \begin{pmatrix} x_1 & x_4 & x_6 \\ x_4 & x_2 & x_5 \\ x_6 & x_5 & x_3 \end{pmatrix}\,,$$
so that $|\widetilde T|^2= x_1^2+x_2^2+x_3^3+2x_4^2+2x_5^2+2x_6^2$, 
and 
\[2\,{\rm tr}(\widetilde TD\widetilde T) =2\mu_1x_1^2+2\mu_2x_2^2+2\mu_3x_3^2+ 2(\mu_1+\mu_2)x_4^2+2(\mu_2+\mu_3)x_5^2+ 2 (\mu_1+\mu_3) x_6^2 \, .\]
Since  $\mu_1+\mu_2+\mu_3=0$, $\mu_3\leq \mu_2<0$ and $\mu_1 \leq \frac{2}{\sqrt{6}}$, from the previous formulas for the eigenvalues we easily get $-2\mu_2 \leq \frac{2}{\sqrt{6}}$,  $-2\mu_3 \leq t+ \frac{2}{\sqrt{6}}$ and
\begin{align}
\nonumber 2\,{\rm tr}(\widetilde TD\widetilde T) &\leq  2\mu_1x_1^2+2(\mu_1+\mu_2)x_4^2+ 2(\mu_1+\mu_3) x_6^2 =  2\mu_1 x_1^2-2\mu_3 x_4^2 
-2\mu_2 \,x_6^2\\
\label{esttrTDT}&\leq \frac{4}{\sqrt{6}}\,x_1^2+2\big(\frac{1}{\sqrt{6}}+2\sqrt{W(D)}\big)x_4^2+\frac{2}{\sqrt{6}}\,x_6^2\,.
\end{align}
On the other hand $x_1+x_2+x_3=0$ since ${\rm tr}(\widetilde T)=0$, and $\mu_1x_1+\mu_2x_2+\mu_3x_3=0$ since $\widetilde T:D=0$. It implies that 
$$\bigg(\frac{3}{\sqrt{6}}(1-t^2/2)^{1/2}-\frac{t}{2}\bigg) x_1= tx_3\,,$$
and consequently, $x_1^2\leq t^2x_3^2\leq \frac{1}{4}x_3^2$ for ${\bf t}_0$ small enough. Back to \eqref{esttrTDT}, we conclude that 
$$2\,{\rm tr}(\widetilde TD\widetilde T)\leq \frac{1}{\sqrt{6}}\,x_3^2+2\big(\frac{1}{\sqrt{6}}+2\sqrt{W(D)}\big)x_4^2+\frac{2}{\sqrt{6}}\,x_6^2\leq  \bigg(\frac{1}{\sqrt{6}}+ 2\sqrt{W(D)}\bigg)|\widetilde T|^2\,, $$
which completes the proof for a (small) universal constant ${\bf t}_0>0$ and ${\bf c}_\star=\max \{2,2\boldsymbol{\ell}_0^{-1/2}\}$. 
\end{proof}

\begin{proof}[Proof of Lemma \ref{Bochineq}]
First compute
$$-\Delta \left(\frac{1}{2}|\nabla Q_\lambda|^2\right)= -|\nabla^2Q_\lambda|^2 + \nabla Q_\lambda :\nabla(-\Delta Q_\lambda) \,.$$
From \eqref{MasterEq}, we derive that  
\begin{multline*}
\partial_k(-\Delta Q_\lambda)= 2(\nabla Q_\lambda:\nabla (\partial_k Q_\lambda))Q_\lambda+|\nabla Q_\lambda|^2\partial_k Q_\lambda\\
+\lambda\bigg ((\partial_k Q_\lambda) Q_\lambda+Q_\lambda\partial_kQ_\lambda-3{\rm tr}(Q_\lambda^2\partial_kQ_\lambda)Q_\lambda -{\rm tr}(Q_\lambda^3)\partial_kQ_\lambda\bigg)\,. 
\end{multline*}
Since $Q_\lambda:\partial_kQ_\lambda=0$ and ${\rm tr}(Q_\lambda^3)= -3W(Q_\lambda)+1/\sqrt{6}$, we obtain 
$$-\Delta \left(\frac{1}{2}|\nabla Q_\lambda|^2\right)\leq |\nabla Q_\lambda|^4 +3\lambda W(Q_\lambda)|\nabla Q_\lambda|^2+\lambda \sum_{k=1}^3\bigg( 2{\rm tr}\big((\partial_kQ_\lambda)Q_\lambda\partial_kQ_\lambda\big) -\frac{1}{\sqrt{6}}|\partial_k Q_\lambda|^2\bigg)\,.$$
It then follows from Lemma \ref{lemmaspectrboch} (applied to $Q=Q_\lambda$ and $T=\partial_k Q_\lambda$) that 
\begin{equation}\label{prebochineq}
-\Delta \left(\frac{1}{2}|\nabla Q_\lambda|^2\right)\leq |\nabla Q_\lambda|^4+3\lambda W(Q_\lambda)|\nabla Q_\lambda|^2+ {\bf c}_\star\lambda\sqrt{W(Q_\lambda)} \,|\nabla Q_\lambda|^2\,. 
\end{equation}
Next, we compute
$$-\Delta\big(W(Q_\lambda)\big) =-{\rm tr}\big(Q_\lambda^2(-\Delta Q_\lambda)\big)+\sum_{k=1}^32{\rm tr}\big((\partial_kQ_\lambda )Q_\lambda\partial_k Q_\lambda\big)\,, $$
and it  follows from  \eqref{MasterEq}  that 
$$-\Delta\big(W(Q_\lambda)\big) =-|\nabla Q_\lambda|^2{\rm tr} (Q_\lambda^3)-\lambda\bigg({\rm tr} \, Q^4-\frac{1}{3}-\big({\rm tr} (Q_\lambda^3)\big)^2\bigg)+\sum_{k=1}^32{\rm tr}\big((\partial_kQ_\lambda )Q_\lambda\partial_k Q_\lambda\big) \,.$$
Noticing that ${\rm tr}\, Q^4 =1/2$, we obtain from Lemma \ref{lemmaspectrboch}, 
\begin{align}
\nonumber -\Delta\big(W(Q_\lambda)\big)& =3W(Q_\lambda)|\nabla Q_\lambda|^2+9\lambda W^2(Q_\lambda)-\lambda\sqrt{6}\,W(Q_\lambda)+\sum_{k=1}^3\bigg( 2{\rm tr}\big((\partial_kQ_\lambda)Q_\lambda\partial_kQ_\lambda\big) -\frac{1}{\sqrt{6}}|\partial_k Q_\lambda|^2\bigg) \\
\label{lastineqproofboch}&\leq  3W(Q_\lambda)|\nabla Q_\lambda|^2+9\lambda W^2(Q_\lambda)-\lambda\sqrt{6}\,W(Q_\lambda)+ {\bf c}_\star\sqrt{W(Q_\lambda)} \,|\nabla Q_\lambda|^2\,.
\end{align}
Combining \eqref{prebochineq} and \eqref{lastineqproofboch}, we are led to 
\begin{align*}
-\Delta {\rm e}_\lambda & \leq  |\nabla Q_\lambda|^4+6\lambda W(Q_\lambda)|\nabla Q_\lambda|^2+9\lambda^2 W^2(Q_\lambda)-\lambda^2\sqrt{6}\,W(Q_\lambda)+2{\bf c}_\star\lambda\sqrt{W(Q_\lambda)} \,|\nabla Q_\lambda|^2\\
&\leq (1+{\bf c}^2_\star/\sqrt{6}) |\nabla Q_\lambda|^4+6\lambda W(Q_\lambda)|\nabla Q_\lambda|^2+9\lambda^2 W^2(Q_\lambda)\\
&\leq C {\rm e}^2_\lambda
\end{align*}
for a universal constant $C>0$. 
\end{proof}

\begin{remark}\label{remmonotsmoothsol}
If $Q_\lambda$ is a smooth solution of \eqref{MasterEq}  in $B_{r}$, then $Q_\lambda$ satisfies the interior monotonicity formula \eqref{IntMonForm} in the ball $B_r$ (see the proof of Proposition \ref{monotonprop}, Step 2, or  \cite[Chapter 2, Sections 2.2 and 2.4]{Simon}). As a consequence, $Q_\lambda$ satisfies 
$$\mathop{\sup}\limits_{B_\rho(x)\subset B_{r/2}} \frac{1}{\rho}\mathcal{E}_\lambda\big(Q_\lambda,B_\rho(x)\big)\leq \frac{2}{r}\mathcal{E}_\lambda\big(Q_\lambda,B_r\big)\,,$$
exactly as in Lemma \ref{corolmonotform}. 
\end{remark}

With Lemma \ref{Bochineq} at hand, Proposition \ref{epsregunifesti} follows from the original argument in \cite{ChenStruwe,Sch}  that we provide for completeness. 

\begin{proof}[Proof of Proposition \ref{epsregunifesti}]
We argue as in \cite{ChenStruwe}, where the scaling argument first presented in \cite{Sch} for harmonic maps is adapted to the harmonic heat flow. Since $Q_\lambda$ is smooth, we can find  $\sigma_\lambda\in(0,r/2)$ such that 
$$\Big(\frac{r}{2}-\sigma_\lambda\Big)^2\sup_{B_{\sigma_\lambda}} {\rm e}_\lambda\geq \frac{1}{2}\sup_{0< \sigma < r/2} \Big(\frac{r}{2}-\sigma\Big)^2\sup_{B_{\sigma}} {\rm e}_\lambda\,. $$
In addition, by continuity we can find $x_\lambda\in \overline B_{\sigma_\lambda}$ such that 
$$\sup_{B_{\sigma_\lambda}} {\rm e}_\lambda={\rm e}_\lambda(x_\lambda):=\overline{\bf e}_\lambda\,. $$
Set $\rho_\lambda:=(\frac{r}{2}-\sigma_\lambda)/2>0$, and notice that $B_{\rho_\lambda}(x_\lambda)\subset B_{\sigma_\lambda+\rho_\lambda}\subset B_{r/2}$. Since $\sigma=\rho_\lambda+\sigma_\lambda<r/2$ and $r/2-\sigma=\frac12(r/2-\sigma_\lambda)$, by definition of $\sigma_\lambda$ we have 
$$\sup_{B_{\rho_\lambda}(x_\lambda)}  {\rm e}_\lambda\leq \sup_{B_{\sigma_{\lambda}+\rho_\lambda}}  {\rm e}_\lambda\leq 8\,\overline{\bf e}_\lambda\,.$$
We define $r_\lambda:=\rho_\lambda\sqrt{\overline{\bf e}_\lambda} $, and, as $B_{\rho_\lambda}(x_\lambda) \subset B_{r/2}$, we also define  
$$\widetilde Q(x):=Q_\lambda \Big(x_\lambda+\frac{x}{\sqrt{\overline{\bf e}_\lambda} }\Big) \quad\text{for $x\in B_{r_\lambda}$}\,.$$
Then $\widetilde Q$ is smooth in $B_{r_\lambda}$, and it solves \eqref{MasterEq}  in $B_{r_\lambda}$ with $\widetilde\lambda:=\lambda/\overline{\bf e}_\lambda$ in place of $\lambda$. Setting
$$\widetilde {\rm e}_{\widetilde\lambda}:= \frac{1}{2}|\nabla \widetilde Q|^2+\widetilde\lambda W(\widetilde Q)\,, $$
we infer from our choice of $\sigma_\lambda$ and $x_\lambda$ that $\widetilde {\rm e}_{\widetilde\lambda}(0)=e_\lambda(x_\lambda)/\overline{\bf e}_\lambda= 1$,  and $\widetilde {\rm e}_{\widetilde\lambda}\leq 8$  in $B_{r_\lambda}$. 
We now claim that $r_\lambda\leq 1$. Indeed, assume by contradiction that $r_\lambda>1$. Then we infer from Lemma \ref{Bochineq}  that 
$$-\Delta\widetilde {\rm e}_{\widetilde\lambda}\leq C\,\widetilde {\rm e}_{\widetilde\lambda}^2\leq 8C\,\widetilde {\rm e}_{\widetilde\lambda}\quad\text{in $B_1$}\,,$$
for a universal constant $C>0$. By Moser's Harnack inequality (see e.g. \cite[Theorem 4.1]{HanLin}) and Remark \ref{remmonotsmoothsol}, we have 
$$1=\widetilde {\rm e}_{\widetilde\lambda}(0)\leq C\int_{B_1}\widetilde {\rm e}_{\widetilde\lambda}\,dx = C\sqrt{\overline{\bf e}_\lambda}\int_{B_{1/\sqrt{\overline{\bf e}_\lambda}}(x_\lambda)} {\rm e}_\lambda\,dx\leq 2 C \boldsymbol{\eps}_{\rm reg}\,,$$
for a universal constant $C>0$. Here we have used that $B_{1/\sqrt{\overline{\bf e}_\lambda}}(x_\lambda)\subset B_{r/2}$ since $1/\sqrt{\overline{\bf e}_\lambda}<\rho_\lambda$.  Therefore, $1\leq 2C \boldsymbol{\eps}_{\rm reg}$ which is clearly a contradiction if $ \boldsymbol{\eps}_{\rm reg}$ is small enough. 

Knowing that $r_\lambda\leq 1$, we may now deduce from our choice of $\sigma_\lambda$ and the definition of $\rho_\lambda$ that 
$$\sup_{0< \sigma < r/2} \Big(\frac{r}{2}-\sigma\Big)^2\sup_{B_{\sigma}} {\rm e}_\lambda\leq 8\rho_\lambda^2\overline{\bf e}_\lambda =8 r_\lambda^2 \leq 8\,.$$
Choosing $\sigma=r/4$ now yields ${\rm e}_\lambda\leq 128 r^{-2}$ in $B_{r/4}$, and the proof is complete. 
\end{proof}

%%%%%%%%%%%%%%%%%%%%%%%%%%%%%%%%%%%%%%%%%%%%%%%%%%%%%%%
%%%%%%%%%%%%%%%%%%%%%%%%%%%%%%%%%%%%%%%%%%%%%%%%%%%%%%%

\section{Regularity of minimizers under norm constraint}
\label{sec:regularity}

%%%%%%%%%%%%%%%%%%%%%%%%%%%%%%%%%%%%%%%%%%%%%%%%%%%%%%%
%%%%%%%%%%%%%%%%%%%%%%%%%%%%%%%%%%%%%%%%%%%%%%%%%%%%%%%

The aim of this section is to prove Theorem \ref{thm:full-regularity}, and the proof is divided according to the following subsections. Recall that in the statement of Theorem \ref{thm:full-regularity}, we assume that the boundary~$\partial\Omega$ is of class $C^3$ and $Q_{\rm b}\in C^{1,1}(\partial\Omega;\mathbb{S}^4)$.

\subsection{Monotonicity formulae}
We start establishing the monotonicity formulae for minimizers of $\mathcal{E}_\lambda$ over $\mathcal{A}_{Q_{\rm b}}(\Omega)$ applying the general principle in Proposition \ref{monotonprop}. First, let us recall that $\bar{Q}_{\rm b}\in\mathcal{A}_{Q_{\rm b}}(\Omega)$ is a given ($\mathbb{S}^4$-valued) reference extension to $\Omega$ of the boundary condition~$Q_{\rm b}$. 

\begin{proposition}\label{corintmonotform}
If $Q_\lambda$ is a minimizer of $\mathcal{E}_\lambda$ over $\mathcal{A}_{Q_{\rm b}}(\Omega)$, then $Q_\lambda$ satisfies the Interior Monotonicity Formula \eqref{IntMonForm} and the Boundary Monotonicity Inequality \eqref{BdMonForm}. Moreover the quantity $K_\lambda(Q_{\rm b},Q_{\lambda})$ in \eqref{BdMonForm} satisfies 
\begin{equation}\label{contrKOm}
K_\lambda(Q_{\rm b},Q_{\lambda})\leq C_\Omega\bigg(\|\nabla_{\rm tan} Q_{\rm b}\|^2_{L^\infty(\partial\Omega)}+\lambda\|W(Q_{\rm b})\|_{L^1(\partial\Omega)}+\mathcal{E}_\lambda(\bar{Q}_{\rm b})\bigg)\,.
\end{equation}
\end{proposition}

\begin{proof}
We first notice that, due to \eqref{Newpotential1} and \eqref{signedbiaxiality}, the potential $W$ is nonnegative for every $Q\in \mathcal{S}_0$. 
 Hence, for each $\varepsilon >0$ the functional $\mathcal{GL}_{\varepsilon}(Q_{\lambda}; \cdot)$ defined in \eqref{defGLepsQref}  is well defined and coercive on $W^{1,2}(\Omega;\mathcal{S}_0)$. Moreover, using the compact Sobolev embedding $W^{1,2}(\Omega;\mathcal{S}_0)\hookrightarrow L^4(\Omega)$, we easily obtain that $\mathcal{GL}_{\varepsilon}(Q_{\lambda}; \cdot)$ is lower semi-continuous with respect to the weak $W^{1,2}$-convergence since all the terms not containig derivatives of $Q$ are weakly continuous. It then follows from the direct method  of calculus of variations that $\mathcal{GL}_{\varepsilon}(Q_{\lambda}; \cdot)$ admits at least one minimizer $Q_\eps$ over $W^{1,2}_{Q_{\rm b}}(\Omega;\mathcal{S}_0)$. 

By Proposition \ref{monotonprop}, it now suffices to show that $Q_\eps$ satisfies \eqref{condconvmonotform} (with $Q_\lambda$ in place of $Q_{\rm ref}$). In addition, observe that \eqref{contrKOm} follows from the minimality of $Q_\lambda$. Indeed, since $\bar{Q}_{\rm b}\in \mathcal{A}_{Q_{\rm b}}(\Omega)$ is an admissible competitor, we have $\|\nabla Q_\lambda\|^2_{L^2(\Omega)}\leq 2\mathcal{E}_\lambda(Q_\lambda)\leq 2\mathcal{E}_\lambda(\bar{Q}_{\rm b})$. 

Now, let us consider an arbitrary sequence $\eps_n\to 0$ satisfying $\eps_n\in(0,\lambda^{-1/2})$. First, we infer from the minimality of $Q_{\eps_n}$ that 
\begin{equation}\label{firstestglapprox}
\frac{1}{2}\int_\Omega|\nabla Q_{\eps_n}|^2+|Q_{\eps_n}-Q_\lambda|^2\,dx\leq \mathcal{GL}_{\varepsilon_n}(Q_{\lambda}; Q_{\eps_n}) \leq \mathcal{GL}_{\varepsilon_n}(Q_{\lambda}; Q_\lambda)=\mathcal{E}_\lambda(Q_\lambda)\,.
\end{equation}
Hence, the sequence $\{Q_{\eps_n}\}$ is bounded in $W^{1,2}_{Q_{\rm b}}(\Omega;\mathcal{S}_0)$, and we can extract a (not relabelled) subsequence such that $Q_{\eps_n}\rightharpoonup Q_*$ weakly in  $W^{1,2}(\Omega)$ for some $Q_*\in W^{1,2}_{Q_{\rm b}}(\Omega;\mathcal{S}_0)$. Up to a further subsequence, we can assume that $Q_{\eps_n}\to Q_*$ strongly in $L^4(\Omega)$ (and therefore in $L^2(\Omega)$) since the embedding $W^{1,2}(\Omega)\hookrightarrow L^4(\Omega)$ is compact. As a consequence, $\int_\Omega W(Q_{\eps_n})\,dx\to \int_\Omega W(Q_*)\,dx$
which, combined with \eqref{firstestglapprox}, implies that  $\int_\Omega(1-|Q_{\eps_n}|^2)^2\,dx\to 0$. 
Therefore, $|Q_*|=1$ a.e. in $\Omega$, and thus $Q_*\in \mathcal{A}_{Q_{\rm b}}(\Omega)$. Now we infer from the minimality of $Q_\lambda$, the weak lower semicontinuity of $\mathcal{E}_\lambda$, the $L^2$-convergence and \eqref{firstestglapprox} that 
\begin{multline*}
\mathcal{E}_\lambda(Q_{\lambda})\leq\mathcal{E}_\lambda(Q_{*})+\frac{1}{2}\int_{\Omega}|Q_*-Q_\lambda|^2\,dx\leq  \liminf_{n\to\infty}\Big(\mathcal{E}_\lambda(Q_{\eps_n})+\frac{1}{2}\int_{\Omega}|Q_{\eps_n}-Q_\lambda|^2\,dx\Big) \\
\leq \liminf_{n\to\infty}\mathcal{GL}_{\varepsilon_n}(Q_{\lambda}; Q_{\eps_n}) \leq \limsup_{n\to\infty} \mathcal{GL}_{\varepsilon_n}(Q_{\lambda}; Q_{\eps_n}) \leq \mathcal{E}_\lambda(Q_\lambda)\,.
\end{multline*}
Consequently, $Q_*=Q_\lambda$ and $\lim_{n} \mathcal{GL}_{\varepsilon_n}(Q_{\lambda}; Q_{\eps_n})=\mathcal{E}_\lambda(Q_{\lambda})$, which completes the proof. 
\end{proof}

\subsection{Compactness of blow-ups and smallness of the scaled energy}
When proving regularity the main issue is to analyse the asymptotic behavior of minimizers at small scales, and the key property is the compactness of rescaled maps. When rescaling around an interior point, we have the following statement.  

\begin{proposition}\label{interiorblowup}
Let $Q_\lambda$ be a minimizer of $\mathcal{E}_\lambda$ over  $\mathcal{A}_{Q_{\rm b}}(\Omega)$. Given $x_0\in \Omega$ and $0<r\leq r_0$ such that $\overline{B_{r_0}(x_0)} \subset \Omega$, consider the rescaled map $Q_{\lambda,r} \in W^{1,2}(B_{r_0/r};\mathbb{S}^4)$ defined by 
$$Q_{\lambda,r}(x):=Q_{\lambda}(x_0+r x)\,.$$
 For every sequence $r_n\to0$, there exist a (not relabeled subsequence) and $Q_* \in W^{1,2}_{\rm{loc}}(\mathbb{R}^3;\mathbb{S}^4)$ such that $Q_{\lambda,r_n} \to Q_*$ strongly in $W^{1,2}_{\rm{loc}}(\R^3)$. 
In addition, $Q_*$ is a degree-zero homogeneous energy minimizing harmonic map into $\mathbb{S}^4$. 
\end{proposition}

To prove Proposition \ref{interiorblowup}, we need two auxiliary lemmata.

\begin{lemma}\label{H1energycompetitor}
Let $Q_{\lambda,r_n}$ be as in Proposition \ref{interiorblowup} and $\rho>0$. For each $n\in\mathbb{N}$ such that $\rho r_n<r_0$, let $v_n\in W^{1,2}(B_\rho; \mathbb{S}^4)$ be such that $v_n=Q_{\lambda, r_n}$ on $\partial B_\rho$ in the sense of traces. Then, 
\[ \limsup_{n \to \infty} \int_{B_\rho} |\nabla Q_{\lambda,r_n}|^2 \, dx \leq  \limsup_{n\to \infty} \int_{B_\rho} |\nabla v_n|^2 \, dx \, .\]
\end{lemma}
\begin{proof}
By minimality of $Q_\lambda$ and a change of variables, $Q_{\lambda,r_n}$ is minimizing $\mathcal{E}_{\lambda r_n^2}(\cdot, B_\rho)$ among all maps in $W^{1,2}(B_\rho;\mathbb{S}^4)$ having the same trace $Q_{\lambda,r_n}$ on $\partial B_\rho$.  Since $v_n$ is an admissible competitor and the potential $W$ is bounded on $\mathbb{S}^4$, we have
\[ \frac{1}{2}\int_{B_\rho} |\nabla Q_{\lambda,r_n}|^2 \, dx \leq \mathcal{E}_{\lambda r_n^2}(Q_{\lambda,r_n}, B_\rho) \leq \frac{1}{2}\int_{B_\rho} |\nabla v_n|^2 \, dx +C\lambda\rho^3  r_n^2 \, , \]
for a constant $C$ depending only on $W$. Then the claim follows letting $n\to\infty$. 
\end{proof}

The following interpolation lemma is due to S. Luckhaus \cite{Lu}.

\begin{lemma}\label{luckhaus}
Let $u, v \in W^{1,2}(\mathbb{S}^2;\mathbb{S}^4)$. For each $\sigma \in (0,1)$, there exists $w \in W^{1,2}(\mathbb{S}^2 \times (1-\sigma, 1);\mathcal{S}_0)$ such that ${w_|}_{\mathbb{S}^2\times \{ 1-\sigma\}}=v$,  ${w_|}_{\mathbb{S}^2\times \{ 1\}}=u$,
\begin{equation}
\label{lukbound1}
\int_{\mathbb{S}^2\times (1-\sigma,1)} |\nabla w|^2 \, dx \leq C \sigma \int_{\mathbb{S}^2}  \left( |\nabla_{\rm tan} u|^2+ |\nabla_{\rm tan} v|^2 \right) \, d\mathcal{H}^2 +C \sigma^{-1} \int_{\mathbb{S}^2}|u-v|^2 \, d\mathcal{H}^2\,,   
\end{equation}
and
\begin{multline}
\label{lukbound2}
{\rm dist}^2(w(x),\mathbb{S}^4) \leq C \sigma^{-2} \left( \int_{\mathbb{S}^2}  \left( |\nabla_{\rm tan} u|^2+ |\nabla_{\rm tan} v|^2 \right) \, d\mathcal{H}^2\right)^{\frac12} \left( \int_{\mathbb{S}^2}|u-v|^2 \, d\mathcal{H}^2 \right)^{\frac12} \\
+  C \sigma^{-3} \int_{\mathbb{S}^2}|u-v|^2 \, d\mathcal{H}^2 
\end{multline}
for a.e. $x\in\mathbb{S}^2\times(1-\sigma,1)$, and a universal constant $C>0$. 
\end{lemma}

\begin{proof}[Proof of Proposition \ref{interiorblowup}]
We essentially follow the proof of \cite[Lemma 2.2.13]{LiWa2} with minor modifications. By Proposition \ref{corintmonotform}, $Q_\lambda$ satisfies the interior monotonicity formula \eqref{IntMonForm}. Rescaling this formula yields
\begin{equation}\label{diffofvacmardi}
\frac{1}{R_2}\mathcal{E}_{\lambda r_n^2}(Q_{\lambda,r_n},B_{R_2})-\frac{1}{R_1}\mathcal{E}_{\lambda r_n^2}(Q_{\lambda,r_n},B_{R_1})\geq \int_{B_{R_2}\setminus B_{R_1}} \frac{1}{|x|}\bigg|\frac{\partial Q_{\lambda,r_n}}{\partial |x|}\bigg|^2\,dx 
\end{equation}
for every $0<R_1<R_2\leq r_0/r_n$. As a consequence, for every  $0<R<r_0/r_n$, we have
$$\frac{1}{R}\mathcal{E}_{\lambda r_n^2}(Q_{\lambda,r_n},B_R)\leq \frac{r_n}{r_0}\mathcal{E}_{\lambda r_n^2}(Q_{\lambda,r_n},B_{r_0/r_n})= \frac{1}{r_0}\mathcal{E}_{\lambda}(Q_{\lambda},B_{r_0}(x_0))\,.$$
Consequently, we can find a (not relabeled) subsequence such that  $Q_{\lambda,r_n}$ converges to a map $Q_*$ weakly in $W^{1,2}_{\rm{loc}}(\mathbb{R}^3)$ and strongly in $L^2_{\rm loc}(\R^3)$. Up to a further subsequence, $Q_{\lambda,r_n}\to Q_*$ a.e. in~$\R^3$, and thus $Q_* \in W^{1,2}_{\rm{loc}}(\mathbb{R}^3;\mathbb{S}^4)$. By the monotonicity formula  \eqref{IntMonForm} satisfied by $Q_\lambda$, we have 
$$\lim_{n\to\infty}\frac{1}{R}\mathcal{E}_{\lambda r_n^2}(Q_{\lambda,r_n},B_{R})
= \lim_{n\to\infty}\frac{1}{Rr_n}\mathcal{E}_{\lambda}(Q_{\lambda},B_{Rr_n}(x_0))
= \lim_{r\to 0}\frac{1}{r}\mathcal{E}_{\lambda}(Q_{\lambda},B_{r}(x_0))$$
 for every $R>0$. Consequently, letting $n\to\infty$ in \eqref{diffofvacmardi} yields by  $W^{1,2}$-weak convergence and lower semicontinuity, 
$$\int_{B_{R_2}\setminus B_{R_1}} \frac{1}{|x|}\bigg|\frac{\partial Q_{*}}{\partial |x|}\bigg|^2\,dx=0 $$
for every $0<R_1<R_2$, which shows that $Q_*$ is $0$-homogeneous. 

Now we aim to prove that, for every radius $R>0$,  $Q_{\lambda,r_n}\to Q_*$ strongly in $W^{1,2}(B_R)$, and that 
$$\int_{B_R}|\nabla Q_*|^2\,dx\leq  \int_{B_R}|\nabla \bar{Q}|^2\,dx$$
for every competitor  $\bar{Q}\in W^{1,2}(B_R;\mathbb{S}^4)$ such that $\bar{Q}-Q_*$ is compactly supported in $B_R$ (i.e., $Q_*$ is a minimizing harmonic map into $\mathbb{S}^4$ on the whole space $\R^3$ w.r.to compactly supported perturbations). By homogeneity of $Q_*$, the value of the radius $R$ does not play a role, and it is enough to show strong $W^{1,2}$-convergence and energy minimality in a ball $B_\rho$ for some radius $\rho\in(0,1)$.  

We fix a competitor $\bar{Q} \in W^{1,2}(B_1; \mathbb{S}^4)$ and $\delta \in (0,1)$ such that $\bar{Q}\equiv Q_*$ a.e. in $B_1\setminus B_{1-\delta}$. Extracting a further subsequence if necessary, by Fatou's lemma and Fubini's theorem, we can select a radius $\rho \in (1-\delta,1)$ and a constant $C>0$ such that 
\begin{equation}
\label{choicerho} \lim_{n\to \infty} \int_{\partial B_\rho} |Q_{\lambda,r_n} -Q_*|^2 \, d\mathcal{H}^2 =0\quad\text{and}\quad \int_{\partial B_\rho} \left( |\nabla Q_{\lambda,r_n}|^2+ |\nabla Q_*|^2 \right) \, d\mathcal{H}^2\leq C \, .
\end{equation} 
We apply Lemma \ref{luckhaus} with a choice $\sigma=\sigma_n\in(0,\delta)$, $u(x)=Q_{\lambda,r_n}(\rho x)$ and $v(x)=Q_*(\rho x)$, $x\in \mathbb{S}^2$, for a sequence of numbers $\sigma_n\to 0$ to interpolate between $Q_{\lambda,r_n}$ and $Q_*$. For $n$ sufficiently large, we choose $\sigma_n:=\|Q_{\lambda,r_n} -Q_*\|^{1/3}_{L^2(\partial B_\rho)}<\delta$, and in this way, 
we obtain $w_n \in W^{1,2}(B_\rho; \mathcal{S}_0)$ satisfying
\[ w_n(x)=
\begin{cases}
\displaystyle\bar{Q}\left( \frac{x}{1-\sigma_n} \right)  & \text{ for $|x| \leq \rho (1-\sigma_n)$} \, , \\[8pt]
Q_{\lambda,r_n}(x)  & \text{for $|x|=\rho$} \,, 
\end{cases}
\]
with the estimate
\begin{multline}
\label{H1competitorbound}
 \int_{B_\rho \setminus B_{\rho(1-\sigma_n)}} |\nabla w_n|^2 \, dx \leq C \Big( \sigma_n \int_{\partial B_\rho}  \big( |\nabla_{\rm tan} Q_{\lambda,r_n}|^2+ |\nabla_{\rm tan} Q_*|^2 \big) \, d\mathcal{H}^2\\ 
 + \frac{1}{\sigma_n} \int_{\partial B_\rho}|Q_{\lambda,r_n}-Q_*|^2 \, d\mathcal{H}^2   \Big) \mathop{\longrightarrow}\limits_{n\to \infty} 0 \,, 
 \end{multline}
and ${\rm dist}(w_n, \mathbb{S}^4) =\mathcal{O}(\sigma_n) \to 0$ uniformly on $B_\rho \setminus B_{\rho(1-\sigma_n)}$ as $n \to \infty$ because of \eqref{choicerho}, \eqref{lukbound2} and our choice of $\sigma_n$.

For $n$ large enough we have $|w_n|\geq 1/2$ on $B_\rho$, hence we can define a sequence of comparison maps $v_n \in W^{1,2}(B_\rho; \mathbb{S}^4)$, so that $v_n=Q_{\lambda,r_n}$ on $\partial B_\rho$, by setting
\begin{equation}
\label{energycompetitor}
v_n(x):=
\begin{cases}
\displaystyle \bar{Q}\left( \frac{x}{1-\sigma_n} \right)  & \text{if $|x| \leq \rho (1-\sigma_n)$} \, , \\[10pt]
\displaystyle \frac{w_n(x)}{|w_n(x)|} &  \text{if $\rho(1-\sigma_n) \leq |x|\leq \rho$} \, .
\end{cases}
\end{equation}
Notice that, since $|w_n|\geq 1/2$, we have $|\nabla v_n|\leq C|\nabla w_n|$ a.e. in the annulus $\{\rho(1-\sigma_n) \leq |x|\leq \rho \}$. 
In view of Lemma \ref{H1energycompetitor}, combining \eqref{H1competitorbound} and \eqref{energycompetitor} together with the weak $W^{1,2}$-convergence of $Q_{\lambda,r_n}$ towards $Q_*$, we obtain
\begin{multline*}
\int_{B_\rho}  |\nabla Q_*|^2 \, dx \leq  \liminf_{n\to \infty} \int_{B_\rho} |\nabla Q_{\lambda,r_n}|^2 \, dx  \leq \limsup_{n\to \infty} \int_{B_\rho} |\nabla Q_{\lambda,r_n}|^2 \, dx \\
\leq \limsup_{n\to \infty} \int_{B_\rho} |\nabla v_n|^2 \, dx 
=  \limsup_{n\to \infty} \left[(1-\sigma_n) \int_{B_\rho}  |\nabla \bar{Q}|^2 \, dx+\int_{B_\rho \setminus B_\rho(1-\sigma_n)}  \big|\nabla v_n\big|^2 \, dx \right]  \\
\leq  \lim_{n \to \infty} \left[ (1-\sigma_n) \int_{B_\rho}  |\nabla \bar{Q} |^2 \, dx+ C \int_{B_\rho \setminus B_{\rho(1-\sigma_n)}} |\nabla w_n|^2 \, dx  \right] =\int_{B_\rho}  |\nabla \bar{Q} |^2 \, dx \, .
\end{multline*}
Since $\bar{Q}$ and $\delta$  are arbitrary, this chain of inequalities provides both the strong $W^{1,2}$-convergence $Q_{\lambda,r_n} \to Q_*$ (using $\bar{Q}=Q_*$) and the energy minimality of $Q_*$ in the ball $B_\rho$. 
\end{proof}

We now aim to perform a similar blow-up analysis around a boundary point. To this purpose, let us recall that $\partial\Omega$ is assumed to be of class $C^3$, and $Q_{\rm b}\in C^{1,1}(\partial\Omega;\mathbb{S}^4)$. We consider the enlarged domain $\widehat\Omega$ defined in \eqref{defenlargeddom}, and we extend $Q_{\rm b}$ to $\widehat\Omega\setminus\Omega$ by setting $\widehat Q_{\rm b}(x):=Q_{\rm b}(\boldsymbol{\pi}_\Omega(x))$ for $x\in\widehat\Omega\setminus\Omega$, where $\boldsymbol{\pi}_\Omega$ is the nearest point projection on $\partial\Omega$. By the regularity assumption on $\partial\Omega$ and $Q_{\rm b}$, we have $\widehat Q_{\rm b}\in C^{1,1}(\widehat\Omega\setminus\Omega)$.

\begin{proposition}\label{boundaryblowup}
Let $Q_\lambda$ be a minimizer of $\mathcal{E}_\lambda$ over $\mathcal{A}_{Q_{\rm b}}(\Omega)$, and denote by $\widehat{Q}_\lambda$ the extension of $Q_\lambda$ to  $\widehat{\Omega}$ given by $\widehat Q_\lambda=\widehat{Q}_{\rm b}$ in $\widehat\Omega\setminus\Omega$. Given ${x}_0\in \partial \Omega$ and $0<r\leq r_0$ such that $\overline{B_{r_0}(x_0)} \subset \widehat{\Omega}$,   consider the rescaled map $\widehat{Q}_{\lambda,r}\in W^{1,2}(B_{r_0/r};\mathbb{S}^4)$ defined by 
$$\widehat{Q}_{\lambda,r}(x)=\widehat{Q}_{\lambda}(x_0+r x)\,.$$  
For every sequence $r_n\to 0$, there exist a (not relabeled) subsequence and $Q_*\in W^{1,2}_{\rm{loc}}(\mathbb{R}^3; \mathbb{S}^4)$ such that $\widehat Q_{\lambda,r_n} \to Q_*$ strongly in $W^{1,2}_{\rm{loc}}(\R^3)$. In addition, $Q_*$ is homogeneous of degree zero, and up to a rotation of coordinates, $Q_*$ is a minimizing harmonic map in the upper half space $\{ x_3 >0 \}$ and $Q_* \equiv Q_{\rm b}(x_0)$ in $\{ x_3<0 \}$.
\end{proposition}

\begin{proof}
 Up to a  translation and a rotation, we may assume that $\{x_3=0\}$ is the tangent plane to $\partial \Omega$ at $x_0$ and the vector $(0,0,-1)$ is the outward unit normal. By Proposition \ref{corintmonotform}, $Q_\lambda$ satisfies the Boundary Monotonicity Inequality \eqref{BdMonForm}, and by rescaling variables, 
\begin{multline}\label{rescmonotbdryblwupprop}
\frac{1}{R_2}\mathcal{E}_{\lambda r_n^2}(\widehat Q_{\lambda,r_n},B_{R_2}\cap\Omega_n)-\frac{1}{R_1}\mathcal{E}_{\lambda r_n^2}(\widehat Q_{\lambda,r_n},B_{R_1}\cap\Omega_n)\geq \int_{(B_{R_2}\setminus B_{R_1})\cap\Omega_n} \frac{1}{|x|}\bigg|\frac{\partial \widehat Q_{\lambda,r_n}}{\partial |x|}\bigg|^2\,dx\\ 
-r_n(R_2-R_1)K_\lambda(Q_{\rm b},Q_\lambda)
\end{multline}
for every $0<R_1<R_2\leq r_0/r_n$, where we have set $\Omega_n:=r_n^{-1}(\Omega-x_0)$. As a consequence, 
$$ \frac{1}{R}\mathcal{E}_{\lambda r_n^2}(\widehat Q_{\lambda,r_n},B_{R}\cap\Omega_n)\leq \frac{1}{r_0}\mathcal{E}_{\lambda}(Q_{\lambda},B_{r_0}(x_0)\cap\Omega)+r_0 K_\lambda(Q_{\rm b},Q_\lambda)$$
 for every $0<R<r_0/r_n$. Since $\widehat Q_{\rm b}\in C^{1,1}(\widehat\Omega\setminus\Omega)$ and $\widehat Q_{\lambda,r_n}(x)=\widehat Q_{\rm b}(x_0+r_nx)$ for $x\in B_R\setminus\Omega_n$ and $0<R<r_0/r_n$,   
 in view of \eqref{rescmonotbdryblwupprop} the sequence $\{ \widehat{Q}_{\lambda,r_n}\}$ is bounded in $W^{1,2}_{\rm{loc}}(\mathbb{R}^3)$. Consequently, there exists a (not relabeled) subsequence such that $\widehat Q_{\lambda,r_n}$ converges to 
 a map $Q_*$ weakly in $W^{1,2}_{\rm{loc}}(\mathbb{R}^3; \mathbb{S}^4)$ and strongly in $L^2_{\rm{loc}}(\R^3)$. Up to a further subsequence, $\widehat Q_{\lambda,r_n}\to Q_*$ a.e. in $\R^3$, and thus 
 $Q_*\in W^{1,2}_{\rm{loc}}(\mathbb{R}^3; \mathbb{S}^4)$. Now observe that $\Omega_n\to \{x_3>0\}$ locally in the Hausdorff metric. Since  $\widehat{Q}_{\rm b}$ is continuous at $x_0$, $ \widehat{Q}_{\lambda,r_n}\to {Q}_{\rm b}(x_0)$ locally uniformly in the open half space $\{x_3<0\}$. Therefore, 
 $Q_*(x) \equiv Q_{\rm b}(x_0)$ in $\{ x_3<0 \}$, and it has constant trace on the plane $\{x_3=0\}$. Arguing essentially as in the proof of Proposition \ref{interiorblowup}, we can let $n\to\infty$ in  \eqref{rescmonotbdryblwupprop} to infer that 
$$ \int_{(B_{R_2}\setminus B_{R_1})\cap\{x_3>0\}} \frac{1}{|x|}\bigg|\frac{\partial  Q_{*}}{\partial |x|}\bigg|^2\,dx=0$$
for every $0<R_1<R_2$. Since the map $Q_*$ is constant in $\{x_3<0\}$, it follows that $Q_*$ is $0$-homogeneous in the whole $\mathbb{R}^3$.  

Now it remains to show the strong convergence of $Q_{\lambda,r_n}$ in $W^{1,2}_{\rm loc}(\R^3)$, and the local energy minimality of $Q_*$ in $\{x_3>0\}$. As in the proof of Proposition \ref{interiorblowup}, by homogeneity, it is enough to show strong $W^{1,2}$-convergence in a ball $B_\rho\subset B_1$ (perhaps up to a subsequence), and energy minimality of $Q_*$ in $B_\rho\cap\{x_3>0\}$. 
  We first notice that,  $\widehat{Q}_{\rm b}$ being $C^{1,1}$ in $\widehat{\Omega}\setminus\Omega$, we have 
$$\int_{B_\rho \setminus \Omega_n } |\nabla \widehat{Q}_{\lambda,r_n}|^2 \, dx = \frac{1}{r_n}\int_{B_{\rho r_n}(x_0) \setminus \Omega } |\nabla \widehat{Q}_{\rm b}|^2 \, dx \mathop{\longrightarrow}\limits_{n\to\infty} 0=\int_{B_\rho \cap \{x_3 \leq 0 \}} |\nabla Q_*|^2 \, dx   \, , $$
and we only need to show that 
$$\int_{B_\rho \cap \Omega_n} |\nabla \widehat{Q}_{\lambda,r_n}|^2 \, dx \mathop{\longrightarrow}\limits_{n\to\infty} \int_{B_\rho\cap\{x_3>0\}} |\nabla Q_*|^2 \, dx$$ 
  to establish the strong convergence of $\widehat{Q}_{\lambda,r_n}$ in $W^{1,2}(B_\rho)$.  
The rest of the proof is quite similar to the one used for the interior case discussed in Proposition \ref{interiorblowup}. For this reason, we only sketch few differences in the construction of comparison maps when gluing different maps near the boundary.
 
The starting point of the construction is to flatten the boundary $\partial \Omega$ near $x_0$. Assuming $\{r_n\}$ suitably small (depending only on $x_0$ and the curvature of $\partial \Omega$ at $x_0$), there exists a sequence of diffeomorphisms $\{ \Phi_n \}\subset C^2(\overline{B_1};\mathbb{R}^3)$ satisfying the following properties: 
\begin{multline} \label{localstraightening}
   \Omega_n\cap B_{r}  =\Phi_n(B_{r}^+) \, , \, \, \partial \Omega_n\cap B_{r}  =\Phi_n(B_{r} \cap\{ x_3=0\}) \quad \forall 0< r\leq1\,,\\  \hbox{and} \quad \| \Phi_n- {\rm id}\|_{C^2(\overline{B_1})} \mathop{\longrightarrow}\limits_{n\to \infty}0 \, , 
  \end{multline}
 where we set $B_r^+:=B_r\cap\{x_3>0\}$, $0<r\leq 1$.  
 We fix $0<\delta<1/4$ and a competitor $\bar{Q} \in W^{1,2}_{\rm{loc}}( \mathbb{R}^3;\mathbb{S}^4)$ such that $\bar{Q}= Q_*$ a.e. in $\R^3\setminus B^+_{1-\delta}$. Notice that $\widehat{Q}_{\lambda,r_n} \circ \Phi_n \rightharpoonup Q_*$ weakly in $W^{1,2}(B_1^+;\mathbb{S}^4)$ as $n \to \infty$.  In addition, $\widehat{Q}_{\lambda,r_n}(\Phi_n(x))={Q}_{\rm b}(x_0+r_n\Phi_n(x))$ and $\bar{Q}(x)=Q_{\rm b}(x_0)$ for $x\in B_1\cap\{x_3=0\}$ because of \eqref{localstraightening}. Consequently, since $Q_{\rm b} \in C^{1,1}(\partial \Omega; \mathbb{S}^4)$ we get
\[ \lim_{n \to \infty} \int_{ B_1\cap\{x_3=0\} } |\widehat{Q}_{\lambda,r_n} \circ \Phi_n  -\bar{Q}|^2 \, d\mathcal{H}^2 =0 \quad\text{and}\quad \lim_{n\to\infty}\int_{ B_1\cap\{x_3=0\} }  |\nabla_{\rm tan}  (\widehat{Q}_{\lambda,r_n} \circ \Phi_n )|^2 \, d\mathcal{H}^2=0 \, .\] 
Hence we can argue as in the interior case: by Fatou's lemma and Fubini's theorem, extracting a further subsequence if necessary, we can select $\rho \in (1-\delta,1)$ and a constant $C>0$ such that 
\[ \lim_{n \to \infty} \int_{\partial B^+_\rho } |\widehat{Q}_{\lambda,r_n} \circ \Phi_n  -\bar{Q}|^2 \, d\mathcal{H}^2 =0 \quad\text{and}\quad \int_{\partial B^+_{\rho }} \left( |\nabla_{\rm tan}  (\widehat{Q}_{\lambda,r_n} \circ \Phi_n )|^2 +|\nabla_{\rm tan}  \bar{Q}|^2 \right) \, d\mathcal{H}^2  \leq C \, .\] 
We then choose the sequence $\sigma_n\to0$ with $0<\sigma_n<\delta$ as $\sigma_n:= \|\widehat{Q}_{\lambda,r_n} \circ \Phi_n  -\bar{Q}\|_{L^2(\partial B^+_\rho)}^{1/3}$.

Before going further, let us notice that we can argue as in Lemma  \ref{energycompetitor} (using the weak convergence of $\widehat{Q}_{\lambda,r_n}$, its energy minimality on $\Omega_n \cap B_\rho$ , and  \eqref{localstraightening}) to prove the following:  for any bounded sequence $\{v_n\} \subset W^{1,2}(B^+_\rho; \mathbb{S}^4)$ such that $v_n= \widehat{Q}_{\lambda, r_n}\circ \Phi_n$ on $\partial B^+_\rho$, we have
\begin{multline}\label{bdH1lowerbound}
\int_{B_\rho^+}  |\nabla Q_*|^2 \, dx \leq  \liminf_{n\to \infty} \int_{\Omega_n \cap B_\rho}  |\nabla \widehat{Q} _{\lambda,r_n} |^2 \, dx  \leq \limsup_{n\to \infty} \int_{\Omega_n \cap B_\rho} |\nabla \widehat{Q}_{\lambda,r_n}|^2 \, dx  \\  
\leq \limsup_{n\to \infty} \mathcal{E}_{\lambda r_n^2}( \widehat{Q}_{\lambda,r_n},\Omega_n \cap B_\rho) \leq  \limsup_{n\to \infty} \mathcal{E}_{\lambda r_n^2}( v_n \circ \Phi_n^{(-1)},\Omega_n \cap B_\rho)\\
=    \limsup_{n\to \infty} \int_{B_\rho^+}  |\nabla v_n |^2 \, dx \,,
\end{multline}
where the last equality follows from a change of variables and \eqref{localstraightening}.

Now, to construct an effective sequence of comparison maps, it is convenient to introduce a biLipschitz map $\Psi \colon \overline{B_1} \to \overline{B^+_1}$. By means of $\Psi$, the comparison maps can be constructed as in the interior case. 
More precisely, we apply Lemma \ref{luckhaus} to the pair of maps from the two-sphere $\mathbb{S}^2$, namely  $u(\cdot)=\widehat{Q}_{\lambda,r_n} \circ \Phi_n (\rho \Psi(\cdot))$ and $v(\cdot)=\bar{Q}(\rho \Psi(\cdot))$. As in the interior case, the lemma produces a sequence $\{ w_n  \} \subset W^{1,2}(B_1;\mathcal{S}_0)$ satisfying
 \[ w_n(x)=
\begin{cases}
\displaystyle \bar{Q}\bigg(\rho \Psi \big( \frac{x}{1-\sigma_n} \big)\bigg) & \text{if $|x| \leq  1-\sigma_n$} \, , \\[8pt]
\widehat{Q}_{\lambda,r_n} \circ \Phi_n \big(\rho \Psi(x)\big)  & \text{if $|x|=1$} \, , 
\end{cases}
\]
with the estimate
\begin{multline}\label{bdH1competitorbound}
\int_{B_1 \setminus B_{1-\sigma_n}} |\nabla w_n|^2 
 \, dx  \leq C \Big( \sigma_n  \int_{\partial B^+_\rho} 
 \big( |\nabla_{\rm tan}  (\widehat{Q}_{\lambda,r_n} \circ \Phi_n )|^2 +|\nabla_{\rm tan}  \bar{Q}|^2 \big) \, d\mathcal{H}^2\\ 
 + \frac{1}{\sigma_n} \int_{\partial B^+_\rho } |\widehat{Q}_{\lambda,r_n} \circ \Phi_n -\bar{Q}|^2 \, d\mathcal{H}^2 
 \Big) \mathop{\longrightarrow}\limits_{n\to\infty} 0 \,, 
 \end{multline}
and ${\rm dist}(w_n, \mathbb{S}^4) \to 0$ uniformly in $B_1 \setminus B_{1-\sigma_n}$ as $n \to \infty$.

Since $|w_n|\geq 1/2$ for $n$ large enough, we can define a sequence $\{\bar{v}_n\} \subset W^{1,2}(B_1; \mathbb{S}^4)$ by setting 
\begin{equation}
\label{bdenergycompetitor}
\bar{v}_n(x)=
\begin{cases}
\displaystyle \bar{Q}(\rho \Psi \left( \frac{x}{1-\sigma_n} \right)) & \text{if $|x| \leq 1-\sigma_n$} \, , \\[8pt]
\displaystyle \frac{w_n(x)}{|w_n(x)|}  &  \text{if $1-\sigma_n \leq |x|\leq 1$} \, , 
\end{cases}
\end{equation}
and it satisfies
\begin{equation}\label{bdannulusbound} 
 \int_{B_1 \setminus B_{1-\sigma_n}} |\nabla \bar{v}_n|^2 \, dx  \leq C  \int_{B_1 \setminus B_{1-\sigma_n}} |\nabla w_n|^2 \, dx\mathop{\longrightarrow}\limits_{n\to\infty} 0 \, . 
 \end{equation} 
Now we pull-back $\bar{v}_n$ on $B^+_\rho$ by setting $v_n(x)=\bar{v}_n(\Psi^{-1}(x/\rho))$, so that $v_n \in W^{1,2}(B_\rho^+;\mathbb{S}^4)$ and $v_n= \widehat{Q}_{\lambda, r_n}\circ \Phi_n$ on $\partial B^+_\rho$ in the sense of traces. Then, a simple computation using the biLipschitz property of $\Psi$ together with \eqref{bdenergycompetitor} and \eqref{bdannulusbound} yields 
 \begin{multline}\label{bdH1upperbound}
 \limsup_{n\to\infty} \int_{B_\rho^+} |\nabla v_n|^2 \, dx \leq  \limsup_{n\to \infty} \int_{B_\rho^+ \setminus(\rho \Psi(B_{1-\sigma_n})) } |\nabla {v}_n|^2 \, dx +\limsup_{n\to \infty} \int_{\rho\Psi(B_{1-\sigma_n}) } |\nabla {v}_n|^2 \, dx\\
  \leq     \limsup_{n\to \infty} C \int_{B_1 \setminus B_{1-\sigma_n} } |\nabla \bar{v}_n|^2 \, dx +\limsup_{n\to \infty} \int_{\rho \Psi(B_{1-\sigma_n}) } |\nabla \bar{Q}|^2 \, dx
 \leq \int_{B_\rho^+}  |\nabla \bar{Q} |^2 \, dx \, .
\end{multline}
Combining \eqref{bdH1lowerbound} and \eqref{bdH1upperbound} with $\bar{Q} \equiv Q_*$, we infer that  
   $\int_{\Omega_n\cap B_\rho} |\nabla \widehat{Q}_{\lambda,r_n}|^2 \, dx \to \int_{B_\rho^+} |\nabla Q_*|^2 \, dx$,  while for an arbitrary $\bar{Q}$,
it yields $\int_{B_\rho^+} |\nabla Q_*|^2 \, dx\leq \int_{B_\rho^+} |\nabla \bar{Q}|^2 \, dx$. The limiting map $Q_*$ is thus a minimizing harmonic map in  $B^+_\rho$, and the proof is  complete.
\end{proof}

All possible limiting maps $Q_*$ obtained by either Proposition \ref{interiorblowup} or Proposition \ref{boundaryblowup} are often referred to as {\em (minimizing) tangent maps} to $Q_\lambda$ at the given point $x_0$. By the monotonicity formulae and the strong compactness of rescaled maps, triviality (i.e., constancy) of all tangent maps implies smallness of the rescaled energy at sufficiently small scale. In our setting, triviality of tangent maps together with smallness of the scaled energy are established in the following propositions. 

\begin{proposition}
\label{interiorsmallenergy}
If $Q_\lambda$ is a minimizer of $\mathcal{E}_\lambda$ over $\mathcal{A}_{Q_{\rm b}}(\Omega)$, then 
$$\lim_{r\to 0}\frac{1}{r} \mathcal{E}_\lambda (Q_\lambda, B_r(x_0))=0 $$
for every $x_0\in \Omega$. 
\end{proposition}

\begin{proof}
Let us fix an arbitrary point $x_0\in\Omega$ and a sequence $r_n\to 0$. According to Proposition \ref{interiorblowup}, up to a subsequence, the rescaled maps satisfy $Q_{\lambda,r_n} \to Q_*$ strongly in $W^{1,2}_{\rm{loc}}(\R^3)$ as $n\to \infty$ for some $Q_* \in W^{1,2}_{\rm{loc}}(\mathbb{R}^3;\mathbb{S}^4)$. Moreover, $Q_*$ is a degree-zero homogeneous energy minimizing harmonic map, so that  
 there exists a smooth harmonic sphere $\omega: \bbS^2 \to \bbS^4$ such that $Q_*(x)=\omega \big( \frac{x}{|x|}\big)$. On the other hand, according to \cite[Theorem 2.7]{SU3} the map $Q_*$ is smooth. In particular, $Q_*$ is smooth at the origin which implies that $\omega$ must be constant, and thus $Q_*$ itself is a constant map. Then the interior monotonicity formula (see Proposition \ref{corintmonotform}) and the strong $W^{1,2}$-convergence  yield
 $$\lim_{r\to 0}\frac{1}{r} \mathcal{E}_\lambda (Q_\lambda, B_r(x_0))=\lim_{n\to\infty} \mathcal{E}_{\lambda r_n^2} (Q_{\lambda,r_n}, B_1)= \frac{1}{2}\int_{B_1}|\nabla Q_*|^2\,dx=0\,,$$
which completes the proof. 
\end{proof}

\begin{proposition}
\label{bdsmallenergy}
Let $\Omega \subset \mathbb{R}^3$ be a bounded open set with $\partial \Omega$ of class $C^3$ and $Q_{\rm b}\in C^{1,1}(\partial\Omega;\mathbb{S}^4)$. If $Q_\lambda$ is a minimizer of $\mathcal{E}_\lambda$ over $\mathcal{A}_{Q_{\rm b}}(\Omega)$ then  
$$\lim_{r\to 0}\frac{1}{r} \mathcal{E}_\lambda (Q_\lambda, B_r(x_0)\cap\Omega)=0 $$
for every $x_0\in\partial\Omega$. 
\end{proposition}

\begin{proof}
As in the previous proof, by the strong $W^{1,2}$-compactness of rescaled maps, it is enough to prove that any limiting map $Q_*$ obtained from Proposition \ref{boundaryblowup} applied at a point $x_0\in\partial\Omega$ is a constant map, i.e., $Q_*\equiv Q_{\rm b}(x_0)$. Indeed,  by the Boundary Monotonicity Inequality (see Proposition~\ref{corintmonotform}), we have 
\[ \lim_{r\to 0} \frac{1}{r}\mathcal{E}_\lambda(Q_\lambda, B_{r}(x_0) \cap \Omega)=\lim_{n\to \infty}\mathcal{E}_{\lambda r_n^2}(Q_{\lambda,r_n}, B_1\cap \Omega_n)=\frac{1}{2}\int_{B_1\cap\{x_3>0\}}|\nabla Q_*|^2\,dx=0 \, ,  \]
where we have set $\Omega_n:=r_n^{-1}(\Omega-x_0)$.

Let us now consider a degree zero homogeneous map $Q_*\in W^{1,2}_{\rm loc}(\R^3;\mathbb{S}^4)$ which is  an energy minimizing harmonic map in $\{x_3>0\}$, and such that $Q_*=Q_{\rm b}(x_0)=:e_0$ in $\{x_3<0\}$.  Setting $\mathbb{S}^2_+:=\mathbb{S}^2\cap\{x_3>0\}$, the homogeneity of $Q_*$ implies that $Q_*(x)=\omega\big(\frac{x}{|x|}\big)$ in $\{x_3>0\}$ where $\omega\in W^{1,2}(\mathbb{S}^2_+;\mathbb{S}^4)$ is a weakly harmonic map on $\mathbb{S}^2_+$ satisfying $\omega=e_0$ on $\partial\mathbb{S}^2_+$ in the sense of traces. It now suffices to show that $\omega\in C^\infty(\overline{\mathbb{S}^2_+})$. Indeed, by Lemaire rigidity theorem \cite[Theorem 3.2]{Le}, a smooth harmonic map on the (closed) half 2-sphere which is constant on the boundary has to be constant. In other words $\omega\equiv e_0$, whence $Q_*\equiv e_0$. 

The smoothness of $\omega$ in the interior  $\mathbb{S}^2_+$ follows from H\'elein's theorem \cite{He}. Smoothness up to the boundary $\partial\mathbb{S}^2_+$ could be asserted directly from \cite{Qing}, but we prefer to give a short argument illustrating in this simple case the reflection principle in Subsection \ref{subsecextension}. 

Consider the map $\widehat Q_*\in W^{1,2}_{\rm loc}(\R^3;\mathbb{S}^4)$ defined by 
$$\widehat Q_*(x):=\begin{cases}
Q_*(x) & \text{if $x_3>0$}\,,\\
{\bf \Sigma}Q_*(\bar{x}) & \text{if $x_3<0$}\,,
\end{cases}$$
where $\bar{x}=(x_1,x_2,-x_3)$ is the reflection of $x=(x_1,x_2,x_3)$ across the plane $\{x_3=0\}$, and $\boldsymbol{\Sigma}:=2 e_0\otimes e_0-{\rm id}$ is the geodesic reflection on $\mathbb{S}^4$ with respect to the point $e_0$. Following the proof of Proposition~\ref{ELeqExt} with $\lambda=0$ (see also Remark \ref{specifgeomreflec}), we infer that the reflected matrix $A(x)$ is the identity and $\widehat Q_*$ is weakly harmonic in $\R^3$. Since $\widehat Q_*$ clearly inherits  homogeneity from $Q_*$, we have $\widehat Q_*(x)=\widehat\omega\big(\frac{x}{|x|}\big)$ for a weakly harmonic map $\widehat\omega\in W^{1,2}(\mathbb{S}^2;\mathbb{S}^4)$. By H\'elein's theorem \cite{He}, $\widehat\omega$ is smooth on $\mathbb{S}^2$, and the conclusion follows since $\widehat\omega=\omega$ in~$\mathbb{S}^2_+$. 
\end{proof}

\subsection{Full regularity}

Combining the results from the subsections above with the $\eps-$regularity theorem and the higher regularity theorem from Section \ref{sec:eps-reg}, we are finally in the position to prove the first regularity result of the paper. 

\begin{proof}[Proof of Theorem \ref{thm:full-regularity}] 
Let $Q_\lambda$ be a minimizer of $\mathcal{E}_\lambda$ over $\mathcal{A}_{Q_{\rm b}}(\Omega)$. First, we prove interior regularity of $Q_\lambda$ by showing smoothness in a neighborhood of  an arbitrary point $x_0\in \Omega$. In view of Proposition \ref{interiorsmallenergy}, we have $\frac{1}{r} \mathcal
{E}_\lambda(Q_\lambda, B_r(x_0))\to 0$ as $r\to 0$. 
Combining Proposition \ref{corintmonotform} and Lemma~\ref{corolmonotform} (with $Q_{\rm ref}=Q_\lambda$) with Corollary \ref{corolhigherreg}, we infer that  $Q_\lambda \in C^{\omega}(B_\rho(x_0))$ for some radius $\rho>0$ possibly depending on the point $x_0$. Since $x_0 \in \Omega$ is arbitrary, we conclude that $Q_\lambda \in C^{\omega}(\Omega)$. 

To prove boundary regularity, we now fix an arbitrary point $x_0 \in \partial \Omega$. By Proposition \ref{bdsmallenergy}, we have $\frac{1}{r} \mathcal{E}_\lambda (Q_\lambda, B_r(x_0) \cap \Omega)  \to 0$ as $r \to 0$. Then we combine Proposition \ref{corintmonotform} and Lemma~\ref{lemmacontrolscalenergext} (with $Q_{\rm ref}=Q_\lambda$) with Corollary \ref{corolhigherregdbry} to conclude that $Q_\lambda\in C^{1,\alpha}(B_\rho(x_0)\cap\overline\Omega)$ for every $\alpha\in(0,1)$ and some radius $\rho>0$.  
Since $x_0$ is arbitrary, a covering argument yields $Q_\lambda \in C^{1,\alpha}(\overline{\Omega})$ for every $\alpha\in(0,1)$. Under the further assumption that $\partial\Omega$ is of class $C^{k,\beta}$ and $Q_{\rm b} \in C^{k,\beta}(\partial \Omega;\mathbb{S}^4)$ for some $\beta>0$ and $k\geq 2$, then Corollary \ref{corolhigherregdbry} with the same covering argument tells us that $Q_\lambda \in C^{k,\beta}(\overline{\Omega})$. Finally, if $\partial \Omega$ is real-analytic and $Q_{\rm b} \in C^\omega (\partial \Omega; \mathbb{S}^4)$, then Corollary \ref{corolhigherregdbry} again implies that $Q_\lambda \in C^\omega(\overline{\Omega})$.  
\end{proof}

%%%%%%%%%%%%%%%%%%%%%%%%%%%%%%%%%%%%%%%%%%%%%%%%%%%%%%%
%%%%%%%%%%%%%%%%%%%%%%%%%%%%%%%%%%%%%%%%%%%%%%%%%%%%%%%
  
\section{LdG-minimizers in the Lyuksyutov regime}

%%%%%%%%%%%%%%%%%%%%%%%%%%%%%%%%%%%%%%%%%%%%%%%%%%%%%%%
%%%%%%%%%%%%%%%%%%%%%%%%%%%%%%%%%%%%%%%%%%%%%%%%%%%%%%%

The main objective of this section is to prove Theorem \ref{thm:noisotropicphase}, and in particular to prove that  isotropic melting (i.e., presence of the zero phase) is avoided by minimizers of the energy functional $\mathcal{F}_{\lambda,\mu}$
in \eqref{LDGenergy} for values of the parameters in the Lyuksyutov regime $\mu\to\infty$.  More precisely, our main goal is to prove that the pointwise norm of any minimizer $Q_\lambda^\mu$ of $\mathcal{F}_{\lambda,\mu}$  
subject to an $\bbS^4$-valued boundary condition is uniformly bounded from below by a positive constant whenever $\mu$ is large enough (and $\lambda$ of order one). As a consequence we deduce that the radial hedgehog \eqref{radialhedgehog} is not energy minimizing and in Theorem \ref{hedgehoginstability} below we will show that it is not even a stable critical point of the energy functional $\mathcal{F}_{\lambda,\mu}$.

Throughout this section, we assume again that the boundary $\partial\Omega$ is of class $C^3$, and that the boundary condition $Q_{\rm b}$ belongs to $C^{1,1}(\partial\Omega;\mathbb{S}^4)$. Given $\lambda>0$ and $\mu>0$, we shall consider critical point of $\mathcal{F}_{\lambda,\mu}$ over the class $W^{1,2}(\Omega;\mathcal{S}_0)$, including as a particular case solutions of the variational problem 
$$\min\Big\{ \mathcal{F}_{\lambda,\mu}(Q) : Q\in W_{Q_{\rm b}}^{1,2}(\Omega;\mathcal{S}_0)\Big\} $$
whose resolution follows from the direct method of calculus of variations. We may denote by $Q_\lambda^\mu$ a critical point of $\mathcal{F}_{\lambda,\mu}$, or simply by $Q^\mu$ (if no confusion arises) hiding the dependence on the fixed parameter $\lambda$ to simplify the notation. We start with elementary/classical considerations and a priori estimates on~$Q^\mu$.

\subsection{A priori estimates}

In view of the explicit expression \eqref{Newpotential2} of the potential $W$, the Euler-Lagrange equation characterizing a critical point $Q^\mu\in W^{1,2}_{Q_{\rm b}}(\Omega;\mathcal{S}_0)$ reads as follows 
\begin{equation}\label{ELeqQmu}
\begin{cases}
\displaystyle -\Delta  Q^\mu=\lambda\left((Q^\mu)^2-\frac{1}{3}|Q^\mu|^2I-\frac{1}{\sqrt{6}}|Q^\mu|^2Q^\mu\right)+\mu(1-|Q^\mu|^2)Q^\mu & \text{in $\Omega$}\,,\\[8pt]
Q^\mu=Q_{\rm b} & \text{on $\partial\Omega$}\,,
\end{cases}
\end{equation}
 with the term $\frac{1}{3}|Q^\mu|^2I$ due to the traceless constraint.

Let us start the analysis by establishing the regularity of critical points. 

\begin{lemma}
If $Q^\mu$ is a critical point point of $\mathcal{F}_{\lambda,\mu}$ over $W_{Q_{\rm b}}^{1,2}(\Omega;\mathcal{S}_0)$, then $Q^\mu\in C^\omega(\Omega)\cap C^{1,\alpha}(\overline\Omega)$ for every $\alpha\in(0,1)$. In addition, 
\begin{enumerate}
\item[(i)] if $\partial\Omega$ is of class $C^{k,\beta}$ and $Q_{\rm b}\in C^{k,\beta}(\partial\Omega;\mathbb{S}^4)$ for some $\beta>0$ and $k\geq 2$, then $Q^\mu\in C^{k,\beta}(\overline\Omega)$;
\item[(ii)] if $\partial\Omega$ is real-analytic and $Q_{\rm b}\in C^{\omega}(\partial\Omega;\mathbb{S}^4)$, then $Q^\mu\in C^{\omega}(\overline\Omega)$.
\end{enumerate}
\end{lemma}

\begin{proof}
In view of equation \eqref{ELeqQmu}, the fact that $Q^\mu\in C^{1,\alpha}(\overline\Omega)$ for every $\alpha\in(0,1)$ follows exactly as in the proof of Proposition~\ref{monotonprop}, Step 1. Then, a classical bootstrap argument based on Schauder estimates shows that  $Q^\mu\in C^\infty(\Omega)$ (see e.g. \cite[Chapters 6 \& 8]{GilbTrud}), and the standard results in \cite[Chapter 6]{Morrey}  give interior analytic regularity. 
Assuming that $\partial\Omega$ is of class $C^{k,\beta}$ and $Q_{\rm b}\in C^{k,\beta}(\partial\Omega;\mathbb{S}^4)$ with $k\geq 2$, we have $Q_\lambda\in C^{k,\beta}(\overline\Omega)$ by standard elliptic regularity at the boundary, see e.g. \cite[Chapter 6]{GilbTrud}. The corresponding conclusion within the analytic class follows again from the results in \cite[Chapter 6]{Morrey}.
\end{proof}

We now prove an a priori estimate on the modulus and  on the gradient of a critical  point reminiscent from the Ginzburg-Landau theories.   

\begin{lemma}\label{maxprinciplemma}
If $Q^\mu$ is a critical point point of $\mathcal{F}_{\lambda,\mu}$ over $W_{Q_{\rm b}}^{1,2}(\Omega;\mathcal{S}_0)$, then $|Q^\mu|\leq 1$ in $\overline\Omega$. 
\end{lemma}

\begin{proof}
Consider the scalar function $u:=1-|Q^\mu|^2$. In view of the previous lemma and equation \eqref{ELeqQmu}, $u$ is continuous in $\overline{\Omega}$ is a classical solution to 
\begin{equation}\label{eqprincipmax}
-\Delta u +2\mu|Q^\mu|^2u\geq \frac{2\lambda}{\sqrt{6}} \big(|Q^\mu|^4-\sqrt{6}{\rm tr}((Q^\mu)^3)\big) \quad\text{in $\Omega$}\,.
\end{equation}
Let $x_0\in\overline\Omega$ be a minimum point for $u$, and assume by contradiction that $u(x_0)<0$, (in other words, $|Q^\mu(x_0)|>1$). Since $u=1-|Q_{\rm b}|^2\equiv 0$ on $\partial\Omega$, we must have $x_0\in \Omega$. Consequently, $\Delta u(x_0)\geq 0$, and \eqref{eqprincipmax} leads to 
\begin{equation}\label{calccontradicmodulus}
0> |Q^\mu(x_0)|^4-\sqrt{6}{\rm tr}\big((Q^\mu)^3 \big)(x_0)\geq |Q^\mu(x_0)|^3-\sqrt{6}{\rm tr}\big((Q^\mu)^3 \big)(x_0)\,.
\end{equation}
However, \eqref{signedbiaxiality} tells us that the right-hand side of \eqref{calccontradicmodulus} is nonnegative, a contradiction.
\end{proof}

\begin{lemma}\label{estigradLDG}
If $Q^\mu$ is a critical point point of $\mathcal{F}_{\lambda,\mu}$ over $W^{1,2}_{Q_{\rm b}}(\Omega;\mathcal{S}_0)$, then
$$|\nabla Q^\mu|\leq C\big(\sqrt{\lambda+\mu}+1\big) \quad\text{in $\overline\Omega$}\,,$$
for a constant $C$ depending only on $\Omega$ and $Q_{\rm b}$. 
\end{lemma}

\begin{proof}
Consider $H$ to be the harmonic extension of $Q_{\rm b}$ to the domain $\Omega$, i.e., 
$$\begin{cases}
\Delta H=0 & \text{in $\Omega$}\,,\\
H=Q_{\rm b} & \text{on $\partial\Omega$}\,.
\end{cases}$$
By our regularity assumption on $\partial\Omega$ and $Q_{\rm b}$, we have $H\in C^{1,\alpha}(\overline\Omega)\cap C^2(\Omega)$ for every $\alpha\in(0,1)$. Setting $U_\mu:=Q^\mu-H$, we deduce from \eqref{ELeqQmu} and Lemma \ref{maxprinciplemma} that $\|\Delta U_\mu\|_{L^\infty(\Omega)}\leq C(\lambda+\mu)$, and $U_\mu=0$ on $\partial\Omega$. By interpolation (see e.g. \cite[Lemma A.2]{BBH1}) and Lemma \ref{maxprinciplemma} again, we conclude that  
$$\|\nabla U_\mu\|_{L^\infty(\Omega)}\leq C\|\Delta U_\mu\|^{1/2}_{L^\infty(\Omega)}\|U_\mu\|^{1/2}_{L^\infty(\Omega)}\leq C\sqrt{\lambda+\mu} \,, $$
for a constant $C$ depending only on $\Omega$ and $Q_{\rm b}$. Since $\|\nabla Q^\mu\|_{L^\infty(\Omega)}\leq \|\nabla U_\mu\|_{L^\infty(\Omega)}+\|\nabla H\|_{L^\infty(\Omega)}$, the conclusion follows. 
\end{proof}

The last ingredients we need are the following monotonicity formulae.  

\begin{lemma}\label{monotformLDG}
If $Q^\mu$ is a critical point point of $\mathcal{F}_{\lambda,\mu}$ over $W^{1,2}_{Q_{\rm b}}(\Omega;\mathcal{S}_0)$, then
\begin{enumerate}
\item for every $x_0\in\Omega$ and every $0<\rho<r<{\rm dist}(x_0,\partial\Omega)$, we have 
\begin{equation}\label{montonotLDGint}
\frac{1}{\rho}\mathcal{F}_{\lambda,\mu}\big(Q^\mu,B_\rho(x_0)\big) \leq \frac{1}{r}\mathcal{F}_{\lambda,\mu}\big(Q^\mu,B_r(x_0)\big)\,;
\end{equation}
\item  there exist a radius ${\bf r}_\Omega>0$ (depending only on $\Omega$) and  a constant $C^\lambda_{Q_{\rm b}}>0$ depending only $\lambda$, $\Omega$, $Q_{\rm b}$, and on (an upper bound of) $\|\nabla Q^\mu\|_{L^2(\Omega)}$ but independent of $\mu$, such that 
\begin{equation}\label{montonotLDGbdry}
\frac{1}{\rho}\mathcal{F}_{\lambda,\mu}\big(Q^\mu,B_\rho(x_0)\big) \leq \frac{1}{r}\mathcal{F}_{\lambda,\mu}\big(Q^\mu,B_r(x_0)\big)+C^\lambda_{Q_{\rm b}}(r-\rho)
\end{equation}
for every $x_0\in\partial\Omega$ and every $0<\rho<r<{\bf r}_\Omega$. 
\end{enumerate}
\end{lemma}
The proof of this lemma follows word by word the one in Proposition \ref{monotonprop} (Step 2 \& Step 3), and we shall omit it. 
We just observe that the constant $C^\lambda_{Q_{\rm b}}$ in \eqref{montonotLDGbdry} is independent of $\mu$ because $Q_{\rm b}$ has always unit norm on $\partial \Omega$.
\vskip5pt

\subsection{Lyuksyutov regime and absence of isotropic melting}

We now complete the proof of Theorem \ref{thm:noisotropicphase} analyzing the asymptotic behavior as $\mu\to+\infty$ of minimizers of $\mathcal{F}_{\lambda,\mu}$ over the class $W^{1,2}_{Q_{\rm b}}(\Omega;\mathcal{S}_0)$. The heart of the matter is Proposition \ref{propnomelt} below. We emphasize that Proposition~\ref{propnomelt} does not rely on energy minimality but on the a priori strong convergence towards a smooth limiting map. Even if not surprising, this statement allows for some flexibility in its application, and we shall indeed  use it in our companion paper \cite{DMP2} when discussing the Lyuksiutov regime in the class of axially symmetric maps. 

\begin{proposition}\label{propnomelt}
Given a sequence $\mu_n\to+\infty$, consider for each $\mu_n$  a critical point $Q_\lambda^{\mu_n}$ of $\mathcal{F}_{\lambda,\mu_n}$ over $W^{1,2}_{Q_{\rm b}}(\Omega;\mathcal{S}_0)$.  Assume that $Q_\lambda^{\mu_n}\rightharpoonup Q_\lambda$ weakly in $W^{1,2}(\Omega;\mathcal{S}_0)$ as $n\to\infty$ for some $Q_\lambda\in\mathcal{A}_{Q_{\rm b}}(\Omega)\cap C^1(\overline\Omega;\mathbb{S}^4)$, and that 
$$\lim_{n\to\infty}\mathcal{F}_{\lambda,\mu_n}(Q_\lambda^{\mu_n})=\mathcal{E}_\lambda(Q_\lambda)\,. $$
Then, 
\begin{enumerate}
\item $Q_\lambda^{\mu_n}\to Q_\lambda$ strongly in $W^{1,2}(\Omega;\mathcal{S}_0)$; 
\vskip5pt
\item $\displaystyle\mu_n\int_\Omega(1-|Q_\lambda^{\mu_n}|^2)^2\,dx\to 0$; 
\vskip5pt
\item $|Q_\lambda^{\mu_n}|\to 1$ uniformly in $\overline\Omega$.
\end{enumerate}
\end{proposition}

\begin{proof}
{\it Step 1.} We start proving items (1) and (2). First, notice that $Q_\lambda^{\mu_n}\to Q_\lambda$ strongly in $L^4(\Omega)$ by the compact embedding $W^{1,2}(\Omega)\hookrightarrow L^4(\Omega)$. Hence $\int_\Omega W(Q_\lambda^{\mu_n})\,dx\to \int_\Omega W(Q_{\lambda})\,dx$ and by lower semicontinuity of the Dirichlet integral  we get $\mathcal{E}_\lambda(Q_\lambda)\leq\liminf_{n\to\infty} \mathcal{E}_\lambda(Q_\lambda^{\mu_n})$.  

Hence,  we have 
$$\mathcal{E}_\lambda(Q_\lambda)\leq\liminf_{n\to\infty} \mathcal{E}_\lambda(Q_\lambda^{\mu_n})+ \limsup_{n \to \infty} \displaystyle \frac{\mu_n}4\int_\Omega(1-|Q_\lambda^{\mu_n}|^2)^2\,dx  \leq \lim_{n\to\infty}\mathcal{F}_{\lambda,\mu_n}(Q_\lambda^{\mu_n})=\mathcal{E}_\lambda(Q_\lambda)\,. $$
Therefore $\mu_n\int_\Omega(1-|Q_\lambda^{\mu_n}|^2)^2\,dx\to 0$ and $\|\nabla Q_\lambda^{\mu_n}\|^2_{L^2(\Omega)}\to \|\nabla Q_\lambda\|^2_{L^2(\Omega)}$. Combined with the weak $W^{1,2}$-convergence, this latter fact implies that $Q_\lambda^{\mu_n}\to Q_\lambda$ strongly in $W^{1,2}(\Omega)$. 

\vskip5pt

\noindent{\it Step 2.} It now remains to prove that $|Q_\lambda^{\mu_n}|\to 1$ uniformly in $\overline\Omega$. Given $\delta\in(0,1)$ arbitrary, we thus have to prove that $|Q_\lambda^{\mu_n}|>\delta$ on $\overline{\Omega}$ for $n$ large enough. We argue by contradiction assuming that, along a (not relabeled) subsequence, there exists $x_n\in\Omega$ such that $|Q_\lambda^{\mu_n}(x_n)|\leq \delta$. Extracting a further subsequence if necessary, we can assume that $x_n\to x_0$ as $n\to\infty$ for some $x_0\in\overline\Omega$. In view of Lemma \ref{estigradLDG} (and the fact that $|Q^\mu|=1$ on $\partial\Omega$), we can find a constant $\kappa\in(0,1)$ independent of $n$ such that for $r_n:=\kappa\mu^{-1/2}_n \to 0$ and for all $n$ we have 
\begin{equation}\label{contradassumpmelt}
B_{r_n}(x_n)\subset\Omega \quad\text{and}\quad |Q_{\mu_n}|^2\leq \frac{1+\delta^2}{2}\text{ in $B_{r_n}(x_n)$}\,.
\end{equation}
We now distinguish two cases: 
\vskip3pt

\noindent{\it Case 1: $x_0\in\Omega$.} The limiting map $Q_\lambda$ being of class $C^1$, we can find a radius $r_0\in(0,{\rm dist}(x_0,\partial\Omega))$ such that  
$$\frac{1}{r_0}\mathcal{E}_\lambda(Q_\lambda,B_{r_0}(x_0))< \frac{\pi\kappa^2(1-\delta^2)^2}{24}\,. $$
From Step 1, we deduce that for $n$ large enough,
\begin{equation}\label{condcontrameltnlarge}
\frac{1}{r_0}\mathcal{F}_{\lambda,\mu_n}(Q_\lambda^{\mu_n},B_{r_0}(x_0))<  \frac{\pi\kappa^2(1-\delta^2)^2}{24}\,.
\end{equation}
On the other hand, still for  $n$ large enough, we have $|x_n-x_0|<r_0/2$ and $r_n<r_0/2$. Then we infer from \eqref{contradassumpmelt} and \eqref{montonotLDGint} that 
\begin{multline*}
\frac{\pi\kappa^2(1-\delta^2)^2}{12}\leq \frac{\mu_n}{4r_n}\int_{B_{r_n}(x_n)}\big(1-|Q_\lambda^{\mu_n}|^2\big)^2\,dx\leq \frac{1}{r_n}\mathcal{F}_{\lambda,\mu_n}(Q_\lambda^{\mu_n},B_{r_n}(x_n))\\
\leq \frac{2}{r_0}\mathcal{F}_{\lambda,\mu_n}(Q_\lambda^{\mu_n},B_{r_0/2}(x_n))\leq  \frac{2}{r_0}\mathcal{F}_{\lambda,\mu_n}(Q_\lambda^{\mu_n},B_{r_0}(x_0))\,,
\end{multline*}
which contradicts \eqref{condcontrameltnlarge}. 
\vskip3pt

\noindent{\it Case 2: $x_0\in\partial\Omega$.} Once again, since $Q_\lambda\in C^1(\overline\Omega)$ and $\partial\Omega$ is of class $C^3$, we can find  a small radius $r_0\in(0,{\bf r}_\Omega)$ where ${\bf r}_\Omega$ is given by Lemma \ref{monotformLDG} such that the nearest point projection on $\partial\Omega$ is well defined in the $r_0$-tubular neighborhood of $\partial\Omega$, and 
$$\frac{1}{r_0}\mathcal{E}_\lambda(Q_\lambda,B_{r_0}(x_0)\cap\Omega)+C^\lambda_{Q_{\rm b}} r_0< \frac{\pi\kappa^2(1-\delta^2)^2}{48}\,,$$
where the constant $C^\lambda_{Q_{\rm b}}$ is also given by Lemma \ref{monotformLDG} (notice that $\|\nabla Q_\lambda^{\mu_n}\|_{L^2(\Omega)}$ is bounded by Step 1). 
From Step 1, we deduce that for $n$ large enough,
\begin{equation}\label{condcontrameltnlargebdry}
\frac{1}{r_0}\mathcal{F}_{\lambda,\mu_n}(Q_\lambda^{\mu_n},B_{r_0}(x_0)\cap\Omega)+C^\lambda_{Q_{\rm b}} r_0<  \frac{\pi\kappa^2(1-\delta^2)^2}{48}\,.
\end{equation}
If we denote $y_n\in\partial\Omega$ the projection of $x_n$ on $\partial\Omega$, we have for $n$ large enough (by \eqref{contradassumpmelt}), 
$$r_n\leq |y_n-x_n|={\rm dist}(x_n,\partial\Omega)\leq |x_n-x_0| <\frac{r_0}{4}\,,$$
so that $|y_n-x_0|<r_0/2$. Arguing as in Case 1 and setting $d_n:=|y_n-x_n|$, we infer from \eqref{contradassumpmelt} and \eqref{montonotLDGint}-\eqref{montonotLDGbdry} that 
\begin{multline*}
\frac{\pi\kappa^2(1-\delta^2)^2}{12}\leq 
\frac{1}{r_n}\mathcal{F}_{\lambda,\mu_n}(Q_\lambda^{\mu_n},B_{r_n}(x_n))\leq \frac{1}{d_n}\mathcal{F}_{\lambda,\mu_n}(Q_\lambda^{\mu_n},B_{d_n}(x_n))\\
\leq \frac{1}{d_n}\mathcal{F}_{\lambda,\mu_n}(Q_\lambda^{\mu_n},B_{2d_n}(y_n)\cap\Omega)\leq \frac{4}{r_0}\mathcal{F}_{\lambda,\mu_n}(Q_\lambda^{\mu_n},B_{r_0/2}(y_n)\cap\Omega)+C^\lambda_{Q_{\rm b}}r_0\\
\leq  \frac{4}{r_0}\mathcal{F}_{\lambda,\mu_n}(Q_\lambda^{\mu_n},B_{r_0}(x_0)\cap\Omega)+C^\lambda_{Q_{\rm b}}r_0\,,
\end{multline*}
which contradicts \eqref{condcontrameltnlargebdry}. 
\end{proof}

\begin{proof}[Proof of Theorem  \ref{thm:noisotropicphase}]
Let us consider an arbitrary sequence $\mu_n\to+\infty$ and corresponding $Q_\lambda^{\mu_n}$ minimizing $\mathcal{F}_{\lambda,\mu_n}$ over $W^{1,2}_{Q_{\rm b}}(\Omega;\mathcal{S}_0)$. Since the map $\bar Q_{\rm b}\in \mathcal{A}_{Q_{\rm b}}(\Omega)$ is an admissible competitor to the minimality of $Q_\lambda^{\mu_n}$, we have 
\begin{equation}\label{almostdonesect4}
\mathcal{F}_{\lambda,\mu_n}(Q_\lambda^{\mu_n})\leq \mathcal{F}_{\lambda,\mu_n}(\bar Q_{\rm b})=\mathcal{E}_\lambda(\bar Q_{\rm b})\,. 
\end{equation}
Therefore, the sequence $\{Q_\lambda^{\mu_n}\}$ is bounded in $W^{1,2}(\Omega; \mathcal{S}_0)$, and we can extract a (not relabeled) subsequence such that $Q_\lambda^{\mu_n}\rightharpoonup Q_\lambda$ weakly  in $W^{1,2}(\Omega)$ for some $Q_{\lambda}\in W^{1,2}_{Q_{\rm b}}(\Omega;\mathcal{S}_0)$. By the compact embedding $W^{1,2}(\Omega)\hookrightarrow L^4(\Omega)$, we have $\int_{\Omega}(1-|Q_\lambda^{\mu_n}|^2)^2\,dx\to \int_{\Omega}(1-|Q_{\lambda}|^2)^2\,dx$, and it follows from \eqref{almostdonesect4} that
$$ \int_{\Omega}(1-|Q_{\lambda}|^2)^2\,dx=\lim_{n\to\infty} \int_{\Omega}(1-|Q_\lambda^{\mu_n}|^2)^2\,dx\leq \lim_{n\to\infty}\frac{1}{\mu_n} \mathcal{F}_{\lambda,\mu_n}(Q_\lambda^{\mu_n}) = 0\,.$$
Hence $|Q_\lambda|=1$ a.e. in $\Omega$, so that $Q_\lambda\in \mathcal{A}_{Q_{\rm b}}(\Omega)$.  

Since any $Q\in \mathcal{A}_{Q_{\rm b}}(\Omega)$ is in fact admissible to test the minimality of $Q_\lambda^{\mu_n}$, we can proceed as in \eqref{almostdonesect4} and use the lower semicontinuity of $\mathcal{E}_\lambda$ to infer that 
\begin{equation}\label{samdbefterror}
\mathcal{E}_\lambda(Q_\lambda)\leq \liminf_{n\to\infty}\mathcal{E}_\lambda(Q_\lambda^{\mu_n})\leq  \liminf_{n\to\infty}\mathcal{F}_{\lambda,\mu_n}(Q_\lambda^{\mu_n})\leq \mathcal{E}_\lambda(Q)
\end{equation}
for every $Q\in \mathcal{A}_{Q_{\rm b}}(\Omega)$. Hence $Q_\lambda$ is a minimizer of $\mathcal{E}_\lambda$ over $\mathcal{A}_{Q_{\rm b}}(\Omega)$, and we deduce from Theorem~\ref{thm:full-regularity} 
that $Q_\lambda\in C^{1,\alpha}(\overline\Omega)$. In addition, using $Q=Q_\lambda$ as competitor in \eqref{samdbefterror} we obtain that  $\mathcal{F}_{\lambda,\mu_n}(Q_\lambda^{\mu_n})\to \mathcal{E}_\lambda(Q_\lambda)$. The conclusion now follows from Proposition \ref{propnomelt}. 
\end{proof}

\subsection{Instability of the melting hedgehog}
In this subsection, we discuss instability of the melting hedgehog $H^\mu_\lambda$ given in \eqref{radialhedgehog} in the Lyuksyutov regime $\mu \to \infty$. 
The (in)stability property here is similar to the one in \cite{INSZ1}, where the low-temperature regime $a^2\to \infty$ is considered. The main result in \cite{INSZ1} is in fact the stability of the melting 
hedgehog in a different range of parameters. Stability is obtained through a careful spectral decomposition, which also gives as a byproduct  instability as $a^2\to \infty$. Here 
we shall use a  different and more direct perturbation argument. More precisely, the instability property of $H^\mu_\lambda$ will essentially follow  from the corresponding one for the constant norm hedgehog $\bar{H}$ seen as a degree-zero homogeneous harmonic map into $\mathbb{S}^4$. 

First we recall that the constant norm hedgehog 
\[
\bar{H}(x)=\sqrt{\frac32} \left( \frac{x}{|x|} \otimes \frac{x}{|x|}-\frac13 I \right)\, 
\]
satisfies $\bar{H} \in W^{1,2}_{\rm{loc}}(\mathbb{R}^3;\mathbb{R}P^2) \cap C^\infty(\mathbb{R}^3\setminus \{ 0\};\mathbb{R}P^2)$. It is a critical point of $\mathcal{E}_\lambda$ both for $\lambda=0$ (i.e., a weakly harmonic map into $\mathbb{S}^4$), and a critical point for $\lambda>0$ since $\nabla_{\rm tan}W(\bar{H}) \equiv 0$.  
In order to discuss its stability properties, we first set for any $\Phi \in C^\infty_c( B_1; \mathcal{S}_0)$, 
\[ \mathcal{E}^{''}_\lambda(\Phi;\bar{H}):=\left[ \frac{d^2}{dt^2} \mathcal{E}_\lambda \left( \frac{\bar{H}+t\Phi}{|\bar{H}+t\Phi|} \right) \right]_{t=0}  \, . \]
The second variation formula for harmonic maps (see, e.g., \cite[Chapter 1]{LiWa2}) yields
\begin{equation}
\label{secondvarElambda}
 \mathcal{E}^{''}_\lambda(\Phi;\bar{H})=\int_{B_1} |\nabla \Phi_T|^2- |\nabla \bar{H}|^2 |\Phi_T|^2 +\lambda D^2_{\rm tan} W (\bar{H}) \Phi :\Phi \, dx   \, ,
 \end{equation}
where $\Phi_T :=\Phi-\bar{H} (\bar{H}:\Phi)$ is the tangential component of $\Phi$ along $\bar{H}$, and
\begin{multline}
\label{D2W}
D^2_{\rm tan} W (\bar{H}) \Phi :\Phi := \left[ \frac{d^2}{dt^2} W \left( \frac{\bar{H}+t\Phi}{|\bar{H}+t\Phi|} \right) \right]_{t=0} =D^2 W (\bar{H}) \Phi_T :\Phi_T  \\ = \frac{1}{\sqrt{6}} \left(  2 (\bar{H} : \Phi_T)^2 + |\Phi_T|^2 - \sqrt{6}\, {\rm tr} (\bar{H} \Phi_T^2) \right)
= \frac{1}{\sqrt{6}} \left(    |\Phi_T|^2 - \sqrt{6}\, {\rm tr} (\bar{H} \Phi_T^2) \right) \, . 
\end{multline} 

Due to the $O(3)$-equivariance of $\bar{H}$, the second variation $\mathcal{E}^{''}_0(\Phi;\bar{H})$ takes a particularly simple form whenever $\Phi$ is a radial vector field. 

\begin{lemma}
\label{secondvarEzeroradial}
For any $\bar{v} \in \mathbb{S}^4$ and any radial function $\eta \in C^\infty_c (B_1 \setminus \{ 0 \})$, we have
\begin{equation}
\label{D2Ezeroradial}
 \mathcal{E}^{''}_0(\eta \bar{v};\bar{H})=\frac45 \int_{B_1} |\nabla \eta|^2- \frac{3}{|x|^2}|\eta|^2  \, dx \,  . 
\end{equation}
\end{lemma}
\begin{proof}
Let ${\bf i} = (1,0,0)^t$, ${\bf j}=(0,1,0)^t$, ${\bf k} = (0,0,1)^t$ be the canonical basis of $\R^3$. From these vectors, we construct a distinguished orthonormal basis of $\mathcal{S}_0$ by setting 
\begin{gather*}
	e_0 =\sqrt\frac{3}{2} \left({\bf k}  \otimes {\bf k}  - \frac{1}{3}I\right) \,, \quad e_1=\frac{1}{\sqrt{2}}({\bf i} \otimes {\bf k}  + {\bf k}  \otimes {\bf i}) \, , \quad e_2=\frac{1}{\sqrt{2}}({\bf j} \otimes {\bf k}  + {\bf k} \otimes {\bf j} ) \,, \\
	e_3=\frac{1}{\sqrt{2}}({\bf i} \otimes {\bf i}- {\bf j}\otimes {\bf j}), \quad e_4=\frac{1}{\sqrt{2}}({\bf i}  \otimes {\bf j} + {\bf j} \otimes {\bf i}) \,. 
\end{gather*}
In terms of the latitude $\theta \in [0,\pi]$ and of the colatitude $\phi \in [0,2\pi)$ on $\mathbb{S}^2$, the components of $\bar{H}$ with respect to this basis are easily seen to be
\begin{gather*}
	\bar{H}:e_0 = \frac{3}{2}\left(\cos^2\theta - \frac{1}{3}\right) \,, \quad \bar{H}: e_1 = \frac{\sqrt{3}}{2} \sin{2\theta} \cos\phi \,, \quad \bar{H}: e_2= \frac{\sqrt{3}}{2} \sin{2\theta} \sin\phi \,, \\
	\bar{H}:e_3 = \frac{\sqrt{3}}{2} \sin^2{\theta} \cos{2\phi}\,, \quad \bar{H}: e_4 = \frac{\sqrt{3}}{2} \sin^2{\theta} \sin{2\phi} \, .
\end{gather*}
Therefore a straightforward calculation gives 
\begin{equation}
\label{orthogonality}
	\int_{\bbS^2} (\bar{H}:e_i )(\bar{H}: e_j) \,{\rm d}{\rm vol}_{\bbS^2} = \frac{4 \pi}{5} \delta_{ij}
\end{equation}
for any $i, j = 0,\dots, 4$. As a consequence, if we write $\bar{v}=\sum_i \bar{v}_i e_i$ with $|\bar{v}|^2=\sum_i \bar{v}_i^2=1$, then $\bar{h}:=\bar{H}:\bar{v}$ satisfies 
\begin{equation}
\label{L2normhbar}
\int_{\bbS^2} \bar{h}^2 \,{\rm d}{\rm vol}_{\bbS^2} = \sum_{i,j=0}^4\int_{\bbS^2} (\bar{H}:e_i)(\bar{H}:e_j) \bar{v}_i  \bar{v}_j \,{\rm d}{\rm vol}_{\bbS^2}= \sum_{i,j=0}^4 \frac{4 \pi}{5} \delta_{ij}\bar{v}_i  \bar{v}_j =\frac{4 \pi}{5}  \, .
\end{equation}
Next, we notice that $\bar{H}$ is a degree-zero homogeneous harmonic map and $|\nabla \bar{H}|^2= |\nabla_{\rm tan} \bar{H}|^2=\frac{6}{|x|^2}$, hence  
\[
	\Delta_{\bbS^2} \bar{h} = -\abs{\nabla_{\rm tan} \bar{H}}^2 \bar{h} = -6 \bar{h} \, ,
\]
and in view of \eqref{L2normhbar} we obtain

\begin{equation}\label{2nd-var-hbar}
	\int_{\bbS^2} \bar{h}^2 |\nabla_{\rm tan} \bar{H} |^2 \,{\rm d}{\rm vol}_{\bbS^2} = \int_{\bbS^2} |\nabla \bar{h}|^2  \,{\rm d}{\rm vol}_{\bbS^2} =6 \int_{\bbS^2} \bar{h}^2 \,{\rm d}{\rm vol}_{\bbS^2}= \frac{6}{5} \cdot 4 \pi \, . 
\end{equation}

Finally, evaluating $\mathcal{E}''_0$ in \eqref{secondvarElambda} for $\Phi=\eta \bar{v}$ and integrating by parts, since $\eta$ is radial and \eqref{2nd-var-hbar} holds, we conclude that
\begin{gather*}
	\mathcal{E}_0^{''}(\eta \bar{v}; \bar{H}) = \int_{B_1} (1-\bar{h}^2)\abs{\nabla \eta}^2  + \frac{\eta^2}{\abs{x}^2} \left( 2 \abs{\nabla \bar{h}}^2 - (1-\bar{h}^2)\abs{\nabla_{\rm tan} \bar{H}}^2 \right) \,dx \\
	= \left( \int_{\bbS^2} (1-\bar{h}^2) {\rm d}{\rm vol}_{\bbS^2} \right) \int_0^1 (\eta')^2 r^2 dr + \left( \int_{\bbS^2}   2 \abs{\nabla \bar{h}}^2 - (1-\bar{h}^2)\abs{\nabla_{\rm tan} \bar{H}}^2 \, {\rm d}{\rm vol}_{\bbS^2} \right) \int_0^1 \eta^2 dr \\
	= \frac45 \cdot 4\pi \int_0^1 (\eta')^2 r^2 dr -\frac{12}{5} \cdot 4\pi  \int_0^1 \eta^2 dr= \frac45 \int_{B_1} |\nabla \eta|^2- \frac{3}{|x|^2}|\eta|^2  \, dx \,  ,
\end{gather*}
and the proof is complete.
\end{proof}

The instability property of $\bar{H}$ for the Dirichlet energy $\mathcal{E}_0$ along some vector field can be derived from the general instability result for harmonic tangent maps from $\mathbb{R}^3$ in to $\mathbb{S}^4$ proved in \cite{SU3} and \cite{LiWa1}, whence the existence of at least one direction of instability can be obtained through a contradiction argument which yields a negative second variation along {\em some} conformal vector field on $\mathbb{S}^4$ localized on the domain by a radial function. Here, exploiting the $O(3)$-equivariance of $\bar{H}$ and using Lemma \ref{secondvarEzeroradial}, we obtain a stronger and more explicit instability result for $\bar{H}$ as a common critical point of all the functionals $\mathcal{E}_\lambda$  along {\em any} suitably localized conformal vector fields on $\mathbb{S}^4$.

\begin{proposition}
\label{consthedgehoginstability}
Let $\bar{H}$ be the constant norm hedgehog. There exists a radial function $\xi \in C^\infty_c (B_1 \setminus \{ 0\})$ such that  for any vector $\bar{v} \in \mathbb{S}^4$,  $\bar{H}$ 
is a critical point of $\mathcal{E}_0$ which is  unstable along the vector field $\Phi:=\xi \bar{v}$, i.e., $\mathcal{E}^{''}_0(\Phi;\bar{H})<0$.

As a consequence, for each $\lambda>0$ there exists a radial function $\xi_\lambda \in C^\infty_c (B_1 \setminus \{ 0\})$ such that  for any vector $\bar{v} \in \mathbb{S}^4$,  $\bar{H}$ is a critical point of $\mathcal{E}_\lambda$ which is unstable along the vector field $\Phi_\lambda:=\xi_\lambda \bar{v}$, i.e., $\mathcal{E}^{''}_\lambda(\Phi_\lambda;\bar{H})<0$.
\end{proposition} 
\begin{proof}
As already proved in Lemma \ref{secondvarEzeroradial} above, we have
\[ \mathcal{E}^{''}_0(\eta \bar{v};\bar{H})=\frac45 \int_{B_1} |\nabla \eta|^2- \frac{3}{|x|^2}|\eta|^2  \, dx \, \]
for any radial function $\eta \in C^\infty_c (B_1 \setminus \{ 0 \})$. In view of the standard Hardy inequality in $\mathbb{R}^3$,
the quadratic form is not bounded from below and there exists a radial function $\eta \in C^\infty_0 (B_1 \setminus \{ 0 \})$ such that $\mathcal{E}^{''}_0(\eta \bar{v};\bar{H})<0$. Indeed, setting $\eta_n(x):=[\min\{ n |x|, |x|^{-1/2}\}-2 ]_+$,  we have a sequence of radial functions $\eta_n \in \rmLip(B_1)$ compactly supported in $B_1\setminus\{0\}$ satisfying  
\[ \int_{B_1} |\nabla \eta_n|^2 \, dx= \frac14 \int_{\frac1{n^{2/3}}<|x|<\frac14} \frac{dx}{|x|^3}+\mathcal{O}(1) =\frac14 \int_{B_1} \frac{\eta_n^2}{|x|^2} \, dx+\mathcal{O}(1)\quad\text{as $n\to\infty$} \, , \]
whence $\mathcal{E}^{''}_0(\eta_n \bar{v};\bar{H})\to -\infty$ as $n \to \infty$. In particular, $\mathcal{E}^{''}_0(\eta_n \bar{v};\bar{H})<0$ for $n$ large enough.
Finally, as $\eta_n \equiv 0$ for $|x|<1/n$ and $|x|>1/4$, taking $\xi=\eta_n *\rho_\varepsilon$ a regularization by convolution with $\eps<1/n$ and $\{ \rho_\eps \}$ a family of radial mollifiers, we have a (family of) radial function $\xi \in C^\infty_c (B_1 \setminus \{ 0\})$ satisfying $\mathcal{E}^{''}_0(\xi \bar{v};\bar{H})<0$ for $\eps>0$ small enough, which proves the first claim of the theorem.

In order to discuss the case $\lambda>0$, we rescale the radial function $\xi$ above setting $\xi_\delta(x):=\xi(x/\delta)$ for $0<\delta<1$ to be chosen later. Computing the second variation of $\mathcal{E}_\lambda$ along the vector field $\Phi^\delta:=\xi_\delta \bar{v} \in C^\infty_c( B_1; \mathcal{S}_0)$, equation \eqref{secondvarElambda} with $\Phi^\delta_T=\Phi^\delta-\bar{H} (\bar{H}:\Phi^\delta)$ (the tangential component of $\Phi^\delta$ along $\bar{H}$) yields 
\[ \mathcal{E}^{''}_\lambda(\Phi^\delta;\bar{H})=\mathcal{E}^{''}_0(\Phi^\delta;\bar{H})+\lambda \int_{B_1}D^2_{\rm tan} W (\bar{H}) \Phi^\delta :\Phi^\delta \, dx \, .
\]
As $\bar{H}$ is degree-zero homogeneous, a simple rescaling gives
\begin{multline*} \mathcal{E}^{''}_\lambda(\Phi^\delta;\bar{H})=\int_{B_1} |\nabla \Phi^\delta_T|^2- |\nabla \bar{H}|^2 |\Phi^\delta_T|^2 \, + D^2 W (\bar{H}) \Phi^\delta_T : \Phi^\delta_T  dx \\
=\delta \left(  \mathcal{E}^{''}_0(\Phi;\bar{H}) + \lambda \delta^2 \int_{B_1}  D^2 W (\bar{H}) \Phi_T : \Phi_T  dx \right) \, .
 \end{multline*}
 Since by construction $ \mathcal{E}^{''}_0(\Phi;\bar{H})<0$, the conclusion follows for $\delta>0$ small enough.
\end{proof}

Finally we consider the radial hedgehog $H^\mu_\lambda$ as the uniaxial critical point of the functional $\mathcal{F}_{\lambda,\mu}$ of the form \eqref{radialhedgehog} discussed in the introduction. Recall that such critical point is the unique minimizer of $\mathcal{F}_{\lambda,\mu}$ in the class of $O(3)$-equivariant maps in $W^{1,2}(B_1;\mathcal{S}_0)$ which agree with $\bar{H}$ on the boundary (see \cite[Theorem 1.4]{INSZ2}).  Moreover, arguing as in the proof of Theorem \ref{thm:noisotropicphase} above, it is not difficult to show that $H^\mu_\lambda \to \bar{H}$ strongly in $W^{1,2}$ as $\mu \to \infty$ (convergence of minimizers in the class of $O(3)$-equivariant maps). In addition, the convergence is locally uniform away from the origin because $|H^\mu_\lambda|=\sqrt{2/3}\, s^\mu_\lambda \to 1$ locally uniformly away from the origin as $\mu \to \infty$.

Exploiting the aforementioned convergence of $H^\mu_\lambda$ to its constant norm counterpart, we are going to infer the instability property of $H^\mu_\lambda$ from the corresponding one for $\bar{H}$ passing to the limit in the second variations of the energies $\mathcal{F}_{\lambda,\mu}$, and using Proposition~\ref{consthedgehoginstability}. With this respect, we first set for any $\Psi \in C^\infty_c( B_1; \mathcal{S}_0)$,
\[ \mathcal{F}_{\lambda,\mu}^{'}(\Psi;H^\mu_\lambda):=\left[ \frac{d}{dt} \mathcal{F}_{\lambda,\mu} \left( H^\mu_\lambda+t\Psi \right)  \right]_{t=0}  \, , \qquad \mathcal{F}_{\lambda,\mu}^{''}(\Psi;H^\mu_\lambda):=\left[ \frac{d^2}{dt^2} \mathcal{F}_{\lambda,\mu} \left( H^\mu_\lambda+t\Psi \right)  \right]_{t=0}  \, .\]
Simple calculations based on \eqref{LDGenergy} now yield 
\begin{equation}
\label{firstvarFlambdamu}
\mathcal{F}_{\lambda,\mu}^{'}(\Psi;H^\mu_\lambda)=\int_{B_1} \nabla H^\mu_\lambda : \nabla \Psi + \lambda \nabla W (H^\mu_\lambda) :\Psi +\mu (|H^\mu_\lambda|^2-1) H^\mu_\lambda :\Psi \, dx  \, ,
\end{equation}
and
\begin{equation}
\label{secondvarFlambdamu}
 \mathcal{F}_{\lambda,\mu}^{''}(\Psi;H^\mu_\lambda)=\int_{B_1} |\nabla \Psi|^2 +\lambda D^2 W (H^\mu_\lambda) \Psi :\Psi+\mu \left( 2 (H^\mu_\lambda:\Psi)^2 + (|H^\mu_\lambda|^2-1)|\Psi|^2 \right) \, dx   \, .
 \end{equation}

We have the following instability result for the radial hedgehog in the Lyuksyutov regime.

\begin{theorem}
\label{hedgehoginstability}
Let $\lambda>0$ be fixed and for each $\mu>0$, let $H^\mu:=H^\mu_\lambda$ be the radial hedgehog. 
There exists a radial function $\xi \in C^\infty_c (B_1 \setminus \{ 0\})$ such that the following holds. Given a vector $\bar{v} \in \mathbb{S}^4$,  if $\Phi_T$ denotes the tangential part along $\bar{H}$ of the vector field $\Phi=\xi \bar{v}\in C^\infty_c(B_1 \setminus \{ 0\};\mathcal{S}_0)$, then 
$\Phi_T \in C^\infty_c(B_1 \setminus \{ 0\};\mathcal{S}_0) $ and $\mathcal{F}_{\lambda,\mu}^{''}(\Phi_T;H^\mu_\lambda)<0$  for all $\mu$ large enough. 
As a consequence, the radial hedgehog ${H}^\mu_\lambda$ is an unstable critical point of $\mathcal{F}_{\lambda,\mu}$ for all $\mu$ sufficiently large.
\end{theorem}
\begin{proof}
Given $\lambda>0$ and $\bar{v} \in \mathbb{S}^4$, we fix the radial function $\xi \in C^\infty_c (B_1 \setminus \{ 0\})$ as constructed in Proposition \ref{consthedgehoginstability} (which depends on $\lambda$, but not on $\bar{v}$). Then we introduce the vector fields $\Phi:=\xi \bar{v}$ and $\Phi_T:=\Phi-\bar{H} (\bar{H}:\Phi)$. Since $\Phi_T \in C^\infty_c(B_1 \setminus \{ 0\};\mathcal{S}_0)$, it is admissible for the second variation  formula \eqref{secondvarFlambdamu}.

Since by construction $H^\mu_\lambda:\Phi_T\equiv 0$, we obtain
\begin{equation}
\label{tansecondvarFlambdamu}
 \mathcal{F}_{\lambda,\mu}^{''}(\Phi_T;H^\mu_\lambda)=\int_{B_1} |\nabla \Phi_T|^2 +\lambda D^2 W (H^\mu_\lambda) \Phi_T :\Phi_T+\mu   (|H^\mu_\lambda|^2-1)|\Phi_T|^2  \, dx   \, .
 \end{equation}

Recall that $H^\mu_\lambda \to \bar{H}$ strongly in $W^{1,2}(B_1;\mathcal{S}_0)$ and  locally uniformly away from the origin as $\mu \to \infty$. As a consequence, the dominated convergence theorem yields
\begin{equation}
\label{secondvarcontinuityI}
\lim_{\mu \to \infty} \int_{B_1} |\nabla \Phi_T|^2 +\lambda D^2 W (H^\mu_\lambda) \Phi_T :\Phi_T  \, dx=\int_{B_1} |\nabla \Phi_T|^2 +\lambda D^2 W (\bar{H}) \Phi_T :\Phi_T  \, dx \, .
\end{equation}
On the other hand, since $H^\mu_\lambda$ is a critical point of $\mathcal{F}_{\lambda,\mu}$, computing \eqref{firstvarFlambdamu} with the vector field $\Psi:=\frac{|\Phi_T|^2}{|H^\mu_\lambda|^2} H^\mu_\lambda \in C^\infty_c(B_1 \setminus \{ 0\};\mathcal{S}_0)$ yields
\[ 0=\mathcal{F}_{\lambda,\mu}^{'}\left( \Psi;H^\mu_\lambda\right)=\int_{B_1} \nabla H^\mu_\lambda : \nabla \left( \frac{|\Phi_T|^2}{|H^\mu_\lambda|^2} H^\mu_\lambda \right)+ \lambda \frac{|\Phi_T|^2}{|H^\mu_\lambda|^2}  \nabla W (H^\mu_\lambda) : H^\mu_\lambda  +\mu (|H^\mu_\lambda|^2-1) |\Phi_T|^2 \, dx \, ,  \]
whence
\begin{align*}
\int_{B_1} \mu(|H^\mu_\lambda|^2-1) |\Phi_T|^2 \, dx = & -\int_{B_1} |\nabla H^\mu_\lambda|^2  \frac{|\Phi_T|^2}{|H^\mu_\lambda|^2} \, dx \\
& -\int_{B_1} \nabla H^\mu_\lambda : H^\mu_\lambda \, \, \nabla  \frac{|\Phi_T|^2}{|H^\mu_\lambda|^2} 
 + \lambda \frac{|\Phi_T|^2}{|H^\mu_\lambda|^2}  \nabla W (H^\mu_\lambda) : H^\mu_\lambda \, dx  \, .
\end{align*}
Since $\nabla \bar{H}:\bar{H} \equiv 0$, $\nabla W(\bar{H}):\bar{H}=(1-\tilde{\beta}(\bar{H}))/\sqrt{6} \equiv 0$, $H^\mu_\lambda \to \bar{H}$ strongly in $W^{1,2}(B_1;\mathcal{S}_0)$ and uniformly on the support of  $\xi$, letting $\mu \to \infty$ in the previous formula leads to
\begin{equation}
\label{secondvarcontinuityII}
\lim_{\mu \to \infty} \int_{B_1} \mu (|H^\mu_\lambda|^2-1) |\Phi_T|^2 \, dx =-\int_{B_1} |\nabla \bar{H}|^2 |\Phi_T|^2 \, dx \, .
\end{equation}
Combining \eqref{tansecondvarFlambdamu} with \eqref{secondvarcontinuityI}-\eqref{secondvarcontinuityII} and taking into account \eqref{secondvarElambda} and \eqref{D2W}, we infer that
\[ \lim_{\mu \to \infty} \mathcal{F}_{\lambda,\mu}^{''}(\Phi_T;H^\mu_\lambda)=\mathcal{E}^{''}_\lambda(\Phi;\bar{H}) \, , \]
and the conclusion follows, since the right hand side is negative by construction of $\xi$ and $\Phi$.
\end{proof}

\begin{remark}
As $H^\mu_\lambda$ is $O(3)$-equivariant, it is also $\mathbb{S}^1$-equivariant in the sense of condition \eqref{equivariance}. Hence, if we choose $\bar{v}\in\mathbb{S}^4$ such that $R \bar{v} R^t=\bar{v}$ for any $R\in \mathbb{S}^1$, then each map $H^\mu_\lambda+t \xi \bar{v}$ is $\mathbb{S}^1$-equivariant for any $t \in \mathbb{R}$. As a consequence, according to Theorem \ref{hedgehoginstability} the radial hedgehog is an unstable critical point of $\mathcal{F}_{\lambda,\mu}$ also in the restricted class of $\mathbb{S}^1$-equivariant maps (a similar conclusion is valid for $\bar{H}$ as critical point of $\mathcal{E}_\lambda$ in view of Proposition \ref{consthedgehoginstability}).
\end{remark}

\begin{remark}
\label{instcomparison}
Besides the difference in the range of parameters observed in the introduction, our instability result in Theorem 4.8  differs from the one in \cite[Theorem 1.2]{INSZ1} because of the different choice of destabilizing perturbations. Indeed, in \cite{INSZ1} such a perturbation has only one component with respect to the reference moving frame $\{ E_i(\theta, \varphi)\}_{i=0,\ldots ,4}$ used there, which in polar coordinates has specifically the form $w(r) E_3(\theta,\varphi)$. It follows from the definition of $E_3$ that the destabilizing perturbation is a vector field along the image of $H^\mu_\lambda$ in $\mathcal{S}_0$ that at every point is tangent to the sphere passing through $H^\mu_\lambda$ but is however orthogonal to the tangent space of cone over $\R P^2$ at the point $H^\mu_\lambda$. In our case the perturbation along $H^\mu_\lambda$ has instead the form of the tangential component $\Phi_T$ to the sphere thorough $H^\mu_\lambda$ of a radial vector field $\Phi=\xi(r) \bar{v}$ for {\bf any} fixed nonzero constant vector $\bar{v} \in \mathcal{S}_0$. As a consequence our perturbations may have (and for suitable choices of $\bar{v}$ indeed have) nontrivial components along all the vector fields $E_1, \ldots, E_4$ of the frame (but always zero component along $E_0$).    
\end{remark}

In the next remark we discuss the role of the biaxial phase in the instability results.
\begin{remark}
\label{biaxialescape} Let $\Phi \in C^\infty_0(B_1;\mathcal{S}_0)$ be fixed and $\Phi_T$ its tangential part along $\bar{H}$. Simple calculations using \eqref{signedbiaxiality}, \eqref{redpotential}, \eqref{D2W}, and Lemma \ref{lemmaspectrboch}, give
\[ \left. \frac{d^2}{dt^2} \widetilde{\beta} \left( \frac{\bar{H}+t\Phi}{|\bar{H}+t\Phi|} \right) \right|_{t=0} =-3\sqrt{6}\left. \frac{d^2}{dt^2} W \left( \frac{\bar{H}+t\Phi}{|\bar{H}+t\Phi|} \right) \right|_{t=0}= -3 \left(    |\Phi_T|^2 - \sqrt{6}\, {\rm tr} (\bar{H} \Phi_T^2) \right) \leq -\frac32 |\Phi_T|^2\, , \]
and in turn
\[ \left. \frac{d^2}{dt^2} \widetilde{\beta} \left( H^\mu_\lambda+t\Phi \right) \right|_{t=0} =
\left. \frac{d^2}{dt^2} \widetilde{\beta} \left( \frac{\bar{H}+t\frac{\Phi}{|H^\mu_\lambda|}}{|\bar{H}+t\frac{\Phi}{|H^\mu_\lambda|}|} \right) \right|_{t=0}
\leq -\frac32 \frac{|\Phi_T|^2}{|H^\mu_\lambda|^2} \, . \]
Expanding around the value $t=0$ and using stationarity of $\bar{H}$ and $H^\mu_\lambda$ both for $\tilde{\beta}$ and for the energy functionals, as $t\to 0$ we infer
\begin{equation}
\label{Elambdaexpansion}
\widetilde{\beta} \left( \frac{\bar{H}+t\Phi}{|\bar{H}+t\Phi|} \right)\leq 1-\frac34 |\Phi_T|^2 t^2 +o(t^2) \, , \quad 
\mathcal{E}_\lambda \left( \frac{\bar{H}+t\Phi}{|\bar{H}+t\Phi|} \right)=\mathcal{E}_\lambda(\bar{H})+\mathcal{E}^{''}_\lambda(\Phi;\bar{H}) \frac{t^2}2+o(t^2) \, , 
\end{equation}
together with
\begin{equation}
\label{Flambdamuexpansion}
\widetilde{\beta} \left( H^\mu_\lambda+t\Phi \right) \leq 1-\frac34 \frac{|\Phi_T|^2}{|H^\mu_\lambda|^2}  t^2 +o(t^2) \, , \quad 
\mathcal{F}_{\lambda,\mu} \left( H^\mu_\lambda+t\Phi \right)=\mathcal{F}_{\lambda,\mu}(H^\mu_\lambda)+\mathcal{F}^{''}_{\lambda,\mu}(\Phi_T;H^\mu_\lambda) \frac{t^2}2+o(t^2) \, . 
\end{equation}
As a consequence of \eqref{Elambdaexpansion} and \eqref{Flambdamuexpansion}, we see that for $t$ sufficiently small {\em biaxial escape} occurs for the perturbed maps in the set where $\Phi_T \neq 0$. Moreover, if $\Phi=\xi \bar{v}$ with $\bar{v}\in\mathbb{S}^4$ and $\mu$ is large enough, then Proposition \ref{consthedgehoginstability} and Theorem \ref{hedgehoginstability} show that this escape is energetically more favourable because the second variations of the energy functionals in \eqref{Elambdaexpansion} and \eqref{Flambdamuexpansion} are negative.
\end{remark}

As a final remark in this section, we further comment on the actual range of validity of our results in the Lyuksyutov regime \eqref{eq:lyuk-regime}. 
\begin{remark}\label{rmk:rescaling} When studying asymptotic limits from a physical perspective, it is important that all quantities to be compared have the same physical dimensions. Experts often rescale the energy in such a way to recast it in a new fully adimensional form (see, e.g., \cite{Gart}). Our energy \eqref{LDGenergy} is only partially non-dimensionalized, because the terms under integral sign (including the volume element) are not pure numbers. In fact, recalling that $Q$-tensors are adimensional by definition and noticing that $\lambda$ and $\mu$ have the dimension of the inverse of a length square, we see that the resulting energy $\mathcal{F}_{\lambda,\mu}$ has the physical dimensions of a length;  
in addition, the ratio $\mu/\lambda$ is adimensional and we are allowed to compare them in a physically meaningful way, considering in particular the case $\mu \gg \lambda$. Thus, in the Lyuksyutov regime \eqref{eq:lyuk-regime} we are requiring that on a fixed domain $\bf{\Omega}$ the parameter $\lambda$ is constant, hence of the same order of $(\rmdiam \bf{\Omega})^{-2}$, whereas $\mu$ is much larger. 

On the other hand, we could obtain a fully non-dimensionalized energy functional by first choosing a reference length and then rescaling the domain with respect to it. In the present situation there are at least three natural choices of length, namely, $\frac{1}{\sqrt{\lambda}}$, $\frac{1}{\sqrt\mu}$, and $\rmdiam\bf{\Omega}$, where the first two choices, up to an harmless numerical factor, correspond to the {\em biaxial coherence length} and the {\em nematic-isotropic correlation length} respectively, see \cite{PeTr,KVZ,Gart}. Calling $\ell$ the chosen reference length, the original energy functional $\tilde{\mathcal{F}}_{LG}(\bf{Q}, \bf{\Omega})$ under the further rescaling ${\bf x}=\ell x$ turns into the non-dimensionalized functional $\mathcal{F}_{\tilde\lambda,\tilde\mu}(Q, \Omega)$ as in 
 \eqref{LDGenergy}, where ${\bf \Omega}=\ell \Omega$ and the new parameters are given by 
$\tilde{\lambda} = \ell^2 \lambda$, $\tilde{\mu} = \ell^2 \mu$, with $\lambda$ and $\mu$ as in \eqref{eq:lyuk-regime}. Thus, the adimensional energy $\mathcal{F}_{\tilde\lambda,\tilde\mu}$ is formally identical to $\mathcal{F}_{\lambda,\mu}$, and our results continue to hold without any change in the regime $\tilde{\lambda}\sim 1$ and $\tilde{\mu}\gg 1$ on a fixed reference domain $\Omega$. It turns out that the second choice, $\ell=1/\sqrt{\mu}$, amounts to $\mu \sim \lambda$ and $\tilde{\mu} =1$, so it is not covered by our results. However, the first and the third choices both correspond to $\rmdiam{\bf \Omega} \sim 1/\sqrt{\lambda}$ and $\tilde{\mu} \sim \mu/\lambda \gg 1$, i.e., to the following generalization of \eqref{eq:lyuk-regime} to domains of unconstrained size, namely
\begin{equation}
\label{eq:lyuk-regime2}
\rmdiam{\bf \Omega} \sim \frac{1}{\sqrt{\lambda}}= \sqrt{\frac{L}{b^2 s_+}} \, , \qquad \frac{1}{\sqrt{\lambda}} \cdot \left(\frac{1}{\sqrt{\mu}}\right)^{-1}=\sqrt{\frac{a^2}{b^2 s_+}} \gg 1 \, .
\end{equation}
As a consequence, we see that the diameter $\rmdiam{\bf \Omega}$ must be comparable to the the biaxial coherence length, while the nematic correlation length must be negligible with respect to them.
Finally, notice that the second condition in \eqref{eq:lyuk-regime2} holds in particular in the {\em low temperature} limit $a^2 \to \infty$ but in domains ${\bf \Omega}$ of smaller and smaller size because of \eqref{splus}, or, alternatively, in the limit $b\to 0$ but on domains ${\bf \Omega}$  with suitably expanding diameter.
  For a more detailed discussion of this non-dimensionalization procedure and related issues, the interested reader is referred to \cite{Gart} and the references therein. 
\end{remark}

%%%%%%%%%%%%%%%%%%%%%%%%%%%%%%%%%%%%%%%%%%%%%%%%%%%%%%%
%%%%%%%%%%%%%%%%%%%%%%%%%%%%%%%%%%%%%%%%%%%%%%%%%%%%%%%

\section{Topology of minimizers}

%%%%%%%%%%%%%%%%%%%%%%%%%%%%%%%%%%%%%%%%%%%%%%%%%%%%%%%
%%%%%%%%%%%%%%%%%%%%%%%%%%%%%%%%%%%%%%%%%%%%%%%%%%%%%%%

 In this section, we discuss topological properties of field configurations $Q$ satisfying assumptions $(HP_0)-(HP_3)$ and, restricting to energy minimizing configurations, we will obtain as a particular case the proof of Theorem \ref{topology}.

 In connection with assumption $(HP_2)$, we start recalling the following auxiliary result which characterizes simple connectivity of any smooth bounded domain $\Omega \subset \mathbb{R}^3$.
 \begin{lemma}\cite[Thm. 3.2 and Corollary 3.5]{BeFr}
 Let $\Omega \subset \mathbb{R}^3$ be a bounded connected open set with boundary of class  $C^1$. Then $\Omega$ is simply connected if and only if its boundary can be written as $\partial \Omega=\cup_{i=1}^N S_i$ and each surface $S_i$ is diffeomorphic to the standard sphere $\mathbb{S}^2 \subset \mathbb{R}^3$.
 \label{simpledomains}
 \end{lemma}

 As already mentioned in the Introduction, by assumption $(HP_1)$ the maximal eigenvalue $\lambda_{\rm max}(x)$ of the matrix $Q(x)$ is simple for every $x\in\partial \Omega$, and there is a well defined smooth eigenspace map $V_{\max} \colon \partial \Omega \to \mathbb{R}P^2$.  In addition, as $\Omega$ is simply connected and in view of Lemma \ref{simpledomains}, there exists a smooth lifting $v_{\rm max} \in C^1(\partial \Omega; \mathbb{S}^2)$ such that, under the inclusion $\mathbb{R}P^2 \subset \mathbb{S}^4$, we have  $V_{\rm max} (x)=\sqrt{3/2} (v_{\rm max}(x) \otimes v_{\rm max}(x)-\frac13 I)$ for all $x \in \partial \Omega$. 

Notice that, as in \eqref{uniaxialbc}, the case $\bar{\beta}=1$ in $(HP1)$ corresponds to $Q/|Q| \colon \partial \Omega \to \mathbb{R}P^2\subset \mathbb{S}^4$. In this case we have $\lambda_{\rm max} \equiv \sqrt{\frac23}|Q|$ on $\partial \Omega$. Still in view of $(HP_2)$ there exists a map $v' \in C^1(\partial \Omega; \mathbb{S}^2)$ such that $Q =|Q|\sqrt{3/2} (v'\otimes v'-\frac13 I)$  on $\partial \Omega$ (under the inclusion $\mathbb{R}P^2 \subset \mathbb{S}^4$). Hence, under the assumption $\bar{\beta}=1$, one has $Q\equiv |Q| V_{\rm max}$ on $\partial \Omega$.

 Recall also that assumption $(HP_3)$ on the lifting $v_{\rm max}$ of the map $V_{\rm max} \colon \partial \Omega \to \mathbb{R}P^2$, namely  that the total degree $ \deg(v_{\rm max}, \partial \Omega)= \sum_{i=1}^N \deg(v_{\rm max},S_i)$ is odd, does not depend on the chosen lifting. Indeed, since on each spherical component $S_i$ of $\partial\Omega$ the lifting exists by simple connectivity of $S_i$,  and it is unique up a sign, each $\deg(v_{\rm max},S_i)$ may only change by a sign when passing to a different lifting.
 
Now we discuss properties of the biaxiality regions defined in \eqref{biaxialityregions}. The first result below shows that the biaxial escape observed in the introduction is indeed topological in nature,  and that every possible value of the biaxiality is attained. 
\begin{lemma}
\label{notempty}
Let $\Omega \subset \mathbb{R}^3$ be a bounded open set with boundary of class  $C^1$, and $Q:\overline{\Omega}\to\mathcal{S}_0$.  
If $\Omega$ and $Q$ satisfy $(HP_0)-(HP_3)$, then the subset $\{ \beta=-1\}\subset \overline{\Omega}$ 
is not empty. As a consequence, $\{ \beta=t\}\subset \overline{\Omega}$ is not empty for every $t \in [-1,\beta_0]$, where $\beta_0:= \max_{\partial \Omega} \beta$. In particular, if $\bar{\beta}=1$, then the range of $\beta$ is $[-1,1]$.  
\end{lemma}
\begin{proof}
The consequence follows trivially from the definition of $\beta_0$, as the set $\Omega$ (hence $\overline{\Omega}$) is connected, and furthermore $\beta_0=1$ whenever $\bar\beta=1$.

To prove the first statement, we argue by contradiction assuming that $\min_{\overline{\Omega}} \beta > -1$. Then the maximal eigenvalue $\lambda_{\rm max}(x)$ of $Q(x)$ is always simple for every $x\in\overline{\Omega}$, hence of class $C^1$, and there is a well defined eigenspace map $\bar{V} \in C^1(\overline{\Omega};\mathbb{R}P^2)$ which extends $V_{\rm max}$ from the boundary of $\Omega$ to its interior. Since $\overline{\Omega}$ is simply connected this map can be lifted to $\tilde{v} \in C^1( \overline{\Omega} ; \mathbb{S}^2)$ which has to satisfy $\deg(\tilde{v},\partial \Omega)=0$ by Stokes's theorem. On the other hand, as both $v_{\rm max}$ and $\tilde{v}$ are liftings of the same map $V_{\rm max}$ at the boundary, we have $v_{\rm max}=\pm \tilde{v}$ on each $S_i$, whence $\deg (v_{\rm max}, S_i)=\pm \deg(\tilde{v},S_i)$ for all $i=1,\ldots, N$. Summing up over $i$ and passing to $mod \, \, 2$, we have  
\[ \deg(v_{\rm max}, \partial \Omega)= \sum_{i=1}^N \deg(v_{\rm max},S_i) = \sum_{i=1}^N \deg(\tilde{v},S_i)=  0 \quad mod \, \, 2  \, \, ,\]
which contradicts $(HP_3)$.
\end{proof}

 We now further investigate properties of the biaxiality regions $\{\beta\leq t\}$, $\{\beta\geq t\}$.  
 The following lemma and its corollary below represent the key points where the analyticity assumption is used.

\begin{lemma}
\label{analytic}
Let $\Omega \subset \mathbb{R}^3$ be a bounded open set with boundary of class  $C^1$, and $Q:\overline{\Omega}\to\mathcal{S}_0$.  If  $\Omega$ and $Q$ satisfy $(HP_0)-(HP_3)$, then the set of singular (critical) value of $\beta=\widetilde{\beta}\circ Q$ in $(-1, \bar{\beta})$ is at most countable and can accumulate only at $\bar{\beta}$. As a consequence, 
\begin{itemize}
\item[1)] for any $t \in (-1, \bar{\beta})$ there exists a regular value $t' \in (-1,t)$ such that $\{ \beta\geq t \} \subset \overline{\Omega}$ is a deformation retract of $\{ \beta \geq t'\}$;
\vskip3pt
\item[2)] for any $t \in [-1, \bar{\beta})$ there exists a regular value $t' \in (t,\bar{\beta})$ such that $\{ \beta\leq t \} \subset  \Omega$ is a deformation retract of $\{ \beta \leq t'\}$.
\end{itemize}
\end{lemma}
\begin{proof}
Since $\beta=\widetilde{\beta} \circ Q \in C^\omega(\Omega)$, by Sard's theorem for analytic functions (see \cite{SoSo}) the set of singular value is finite on each compact set $K \subset \Omega$, hence all but countably many $t \in (-1,\bar{\beta})$ are regular for $\beta$ in $\Omega$. For such $t$, the level set $\{ \beta =t \}$ is contained in $\Omega$ by definition of $\bar{\beta}$ and it is a finite union of analytic, connected, orientable and boundaryless surfaces. However, since the singular values are finite on compact sets and in view of the definition of $\bar{\beta}$, the only accumulation point for the singular values can be $\bar{\beta}$. Indeed, otherwise there would be a countably many distinct singular value $\beta_n \to \beta_* \in [-1, \bar{\beta})$ and corresponding distinct critical points $x_n \in \{ \beta= \beta_n\} \subset \Omega$ such that up to subsequences $x_n \to x_* \in \{ \beta=\beta_*\}$. Notice that $x_* \in \partial \Omega$, otherwise $x_*$ would be a critical point as well and $\beta_*$ would be a singular value, with coutably many singular values attained in a neighborhood of $x_*$, which contradicts Sard's theorem. Thus $x_* \in  \{ \beta=\beta_*\} \cap \partial \Omega$, which is however impossible by definition of $\bar{\beta}$.
To conclude the proof, we observe that the set of regular value is open. Then, given a regular value $t$, choosing $t'$ sufficiently close to $t$, the conclusion 1) (resp. 2)) follows by a standard retraction following the gradient (resp. negative gradient) flow associated with $\beta$ in $\Omega$ in a neighboorhood of $\{\beta=t \} \subset \Omega$. Actually the same argument applies for any singular value $t$, such value being isolated by the discussion above, and the conclusion follows from real analyticity and the retraction theorem of {\L}ojasiewicz (see \cite[Theorem 5]{Lo}).   
\end{proof}

\begin{corollary}
\label{analyticcor}
Let $\Omega \subset \mathbb{R}^3$ be a bounded open set with boundary of class  $C^1$, and $Q:\overline{\Omega}\to\mathcal{S}_0$.  If  $\Omega$ and $Q$ satisfy $ (HP_0)-(HP_3)$ with $\bar{\beta}=1$ and $Q\in C^\omega(\overline{\Omega}; \mathcal{S}_0)$, then the set of singular (critical) value of $\beta$ in $(-1, 1)$ is finite, and there exists a regular value $t' \in (-1,1)$ such that $\{ \beta=1\} \subset \overline{\Omega}$ is a deformation retract of $\{ \beta \geq t'\} \subset \overline{\Omega}$.
\end{corollary}
\begin{proof}
The proof is similar to the one of Lemma \ref{analytic}, so it will be just sketched. In view of the analytic regularity up to the boundary, the tensor $Q$ has an analytic extension $\widehat{Q}$ (simply by power series) to a larger open set $\widehat{\Omega}\supset\overline{\Omega}$. Then the function $\widehat{\beta}:=\widetilde{\beta}\circ \widehat{Q}$ is analytic in $\widehat{\Omega}$ with finitely many critical values in $\overline{\Omega}$ again by Sard's theorem. Clearly $1$ is a critical value (maximum) of $\widehat\beta$. Hence, choosing a slightly smaller regular value $t'$, the conclusion still follows from \cite{Lo} retracting the set $\{ \beta \geq t'\} \subset \overline{\Omega}$ onto $\{ \beta=1\}$ by the gradient flow of $\widehat\beta$ in $\Omega$. 
\end{proof}

The first information on the topology of the biaxiality regions is contained in the following result. 
\begin{proposition}
\label{nosimplyconnected}
Let $\Omega \subset \mathbb{R}^3$ be a bounded open set with boundary of class  $C^1$, and $Q:\overline{\Omega}\to\mathcal{S}_0$.  If $\Omega$ and $Q$ satisfy $ (HP_0)-(HP_3)$, then the biaxiality regions satisfy
\begin{itemize}
\item[1)] $\{ \beta \geq t \}$ is not simply connected for any $t \in (-1, \bar{\beta})$;
\vskip3pt
\item[2)] $\{ \beta \leq t \}$ is not simply connected for any $ t\in (-1, \bar{\beta})$;
\vskip3pt
\item[3)] the negative uniaxial set $\{  \beta=-1\}$ is not simply connected;
\vskip3pt
\item[4)] $\{ \beta=t\}$ contains a surface of positive genus for any regular value $t\in (-1,\bar{\beta})$  of the function $\beta$;
\vskip3pt
\item[5)] if in addition $\bar{\beta}=1$ and $Q\in C^\omega(\overline{\Omega}; \mathcal{S}_0)$, then the set $\{ \beta=1\} \subset \overline{\Omega}$ is not simply connected.
\end{itemize}
\end{proposition}
\begin{proof}
In view of Lemma \ref{analytic} it is enough to prove claim 1) and 2) for a regular value $t \in (-1, \bar{\beta})$ since (non)simple connectivity passes to deformation retracts. A similar argument applies to claim 3). Indeed, $t=-1$ is a singular value (minimum), and it is isolated by Lemma \ref{analytic}. Hence, combining claim 2) for regular values $t'$ close to $-1$, the set $\{ \beta \leq t' \}$ is not simply connected, and thus its deformation retract $\{ \beta =-1\}$ is also nonsimply connected.

Let us now prove claims 1) and 4). We assume that $t\in (-1, \bar{\beta})$ is a fixed regular value of $\beta \in C^\omega(\Omega)$. Then 
  the set $\{ \beta \geq t\}$ is the closure of the open set $\Omega \cap \{ \beta>t\}$ which is bounded with smooth boundary. In addition, $\{ \beta \geq t\}$ and $\Omega \cap \{ \beta>t\}$ are homotopically equivalent (by inward-retracting both sets along the normal direction in a small neighborhood of the boundary). So it is enough to show that $\widetilde{\Omega}:=\Omega \cap \{ \beta>t\}$ is not simply connected.
Observe that in view of the regularity of $t$ and the smoothness of the boundary, we can write $\partial \widetilde\Omega$ as a disjoint union
\[ \partial \widetilde{\Omega}=\partial \Omega \cup \{ \beta=t\}= \left( \cup_{i=1}^N S_i \right) \cup \left(\cup_{j=1}^M {\widetilde{S}_j} \right)\, , \]
where each $S_i$ is diffeomorphic to $\mathbb{S}^2$ and each $\widetilde{S}_j$ is compact, analytic, connected, orientable and boundaryless surface because $\{ \beta=t\} \subset \Omega$. Now we claim that there exists an index $j_*$ such that the surface $\widetilde{S}_{j_*}$ has positive genus. In other words, claim 4) holds and the open set $\widetilde{\Omega}$ is not simply connected in view of Lemma \ref{simpledomains}, i.e., claim 1) also holds.

To prove the existence of the distinguished surface $\widetilde{S}_{j_*}$, we argue by contradiction assuming that the genus $g(\widetilde{S}_j)=0$ for each $j=1, \ldots, M$.  Hence the Euler characteristic $\chi(\widetilde{S}_j)=2-2g(\widetilde{S}_j)=2$ for each $j=1, \ldots, M$, and we shall derive a contradiction from this fact. Indeed, notice first that the maximal eigenvalue $\lambda_{\rm max}(Q(x))$ is simple for every $x\in\{ \beta \geq t\}\subset \overline{\Omega}$. Therefore, there is a well defined smooth eigenspace map $\widetilde{V} \colon \{ \beta \geq t\} \to \mathbb{R}P^2$, $\widetilde{V}(x)={\rm Ker} \left( Q(x)-\lambda_{\rm max}(Q(x) ) I\right)$. Since each $\widetilde{S}_j$ are assumed to be of zero genus, both $\widetilde{\Omega}$ and $\{ \beta \geq t\}$ are simply connected by Lemma~\ref{simpledomains}. Therefore the map $\widetilde{V} \in C^1( \{ \beta \geq t\} ; \mathbb{R}P^2)$ has a lifting $\widetilde{v} \in C^1( \{ \beta \geq t\};\mathbb{S}^2)$ as in the proof of Lemma \ref{notempty}.
From Stokes' theorem we infer that
\[ \deg(\widetilde{v},\partial \widetilde{\Omega})=\sum_{i=1}^N \deg(\widetilde{v},S_i) +\sum_{j=1}^M \deg(\widetilde{v},\widetilde{S}_j)=0 \, .\]
 Then assumption $(HP_3)$ yields $\sum_{j=1}^M \deg(\widetilde{v},\widetilde{S}_j)\neq 0$, so that there exists $1\leq j_* \leq M$ such that $\deg(\widetilde{v},\widetilde{S}_{j_*}) \neq 0$.

Now consider $F=T \mathbb{S}^2 \to \mathbb{S}^2$ the (real, oriented, rank-two) tangent bundle of $\mathbb{S}^2$ with its Euler class $e(F)\in H^2(\mathbb{S}^2;\mathbb{Z})$. With respect to a normalized volume form on $\mathbb{S}^2$, we can write $e(F)=2 d{\rm vol}_{\mathbb{S}^2} \in H^2_{dR}(\mathbb{S}^2;\mathbb{R})$, and its Euler number (i.e., Euler characteristic) is $\chi(\mathbb{S}^2)=\int_{\mathbb{S}^2} e(F)=2$.

Using the map $\widetilde{v}$ we can consider the pull-back bundle $\widetilde{v}^{*}F \to \widetilde{S}_{j_*} $ which is a smooth real oriented rank-two vector bundle over $\widetilde{S}_{j_*}$. By functoriality of the Euler class (see e.g. \cite{BT}), we have
\[ \int_{\widetilde{S}_{j_*}} e(\widetilde{v}^{*}F)=\int_{\widetilde{S}_{j_*}} \widetilde{v}^* e(F)=2  \int_{\widetilde{S}_{j_*}} \widetilde{v}^* d{\rm vol}_{\mathbb{S}^2}= 2\deg (\widetilde{v},\widetilde{S}_{j_*}) \neq 0 \, , \] 
hence the pull-back bundle $\widetilde{v}^{*}F \to \widetilde{S}_{j_*}$ is nontrivial. On the other hand, since $\widetilde{S}_{j_*} \subset \{ \beta=t\}$ and $t \in (-1,1)$ is a regular value, each eigenvalue $\lambda \in \sigma(Q(x))= \{\lambda_{\rm max}(x), \lambda_{\rm mid}(x), \lambda_{\rm min}(x)\}$ is simple for every $x\in\widetilde{S}_{j_*}$. Therefore there are well defined eigenspace maps $\widetilde{V}_{\rm mid} , \widetilde{V}_{\rm min} \in C^1( \widetilde{S}_{j_*} ; \mathbb{R}P^2)$ and corresponding liftings $\widetilde{v}_{\rm mid}, \widetilde{v}_{\rm min} \in C^1(\widetilde{S}_{j_*};\mathbb{S}^2)$ (since $\widetilde{S}_{j_*}$ simply connected, i.e., $g(\widetilde{S}_{j_*})=0$). By the spectral theorem we have
$ F_{\widetilde{v}(x)}=T_{\widetilde{v}(x)} \mathbb{S}^2=\{\widetilde{v}(x)\}^\perp=\mathbb{R} \widetilde{v}_{\rm mid}(x) \oplus \mathbb{R} \widetilde{v}_{\rm min}(x)$ for every $x \in \widetilde{S}_{j_*}$. 
Hence the bundle $\widetilde{v}^*F\to \widetilde{S}_{j_*}$ is trivial and $\widetilde{v}_{\rm mid}, \widetilde{v}_{\rm min} \in C^1(\widetilde{S}_{j_*};F)$ provides a trivializing frame (up to orientation), a contradiction.

To prove claim 2) we fix a regular value $t \in (-1, \bar{\beta})$, and we recall that $\partial \{ \beta \leq t\}= \{ \beta= t \} \subset \Omega$ is a finite union of  surfaces of class $C^1$ (in fact analytic) which are  disjoint, embedded, connected and boundaryless. Notice that $\partial \{ \beta \geq t \}= \partial \Omega \cup \{ \beta=t\}$ is also a finite union of $C^1$-surfaces which are disjoint, embedded, connected and boundaryless. Moreover, since $\Omega$ is simply connected and $\{ \beta \geq t \}$ is not (because of claim 1)), one of the components of $\{ \beta=t\}$ has positive genus  by Lemma~\ref{simpledomains}. Applying again Lemma \ref{simpledomains} to $\{ \beta<t\} \subset \Omega$, we infer that $\{ \beta<t \}$ is not simply connected because the total genus of its boundary is positive. Hence $\{ \beta \leq t\}$ is also not simply connected since the two sets are homotopically equivalent. 

Finally, the proof of claim 5) follows from claim 1) for regular values $t \in (-1,1)$ combined with the homotopic equivalence property stated in Corollary \ref{analyticcor}.
\end{proof}

As a direct consequence of the previous proposition, we have the linking property between biaxiality sets.

\begin{proposition}
\label{biaxiallinking}
Let $\Omega \subset \mathbb{R}^3$ be a bounded open set with boundary of class  $C^1$, and $Q:\overline{\Omega}\to\mathcal{S}_0$.  Assume  that $\Omega$ and $Q$ satisfy $(HP_0)-(HP_3)$. If $[t_1,t_2]\subset [-1, \bar{\beta})$ is such that $(t_1,t_2)$ contains no singular value of $\beta=\widetilde{\beta}\circ Q$, then $\{ \beta \leq t_1\} \subset \Omega$ and $\{ \beta \geq t_2\} \subset \overline{\Omega}$ are nonempty compact and disjoint subset of $\overline{\Omega}$, and they are mutually linked. 
\end{proposition} 
\begin{proof}
In view of Lemma \ref{notempty} the sets $\{ \beta \leq t_1\} \subset \Omega$ and $\{ \beta \geq t_2\} \subset \overline{\Omega}$ are nonempty compact and disjoint subset of $\overline{\Omega}$. Since $[t_1,t_2]\subset [-1, \bar{\beta})$ we clearly have $\{ \beta \leq t_1\} \subset \overline{\Omega} \setminus \{ \beta \geq t_2\}=\{ \beta<t_2\}$ and $\{ \beta \geq t_2\} \subset \overline{\Omega} \setminus \{ \beta \leq t_1\}=\{ \beta>t_1\}$. As $(t_1,t_2)$ contains no singular value, these two sets are homotopically equivalent to $\{ \beta \leq t_1\}$ and $\{\beta \geq t_2\}$. Indeed, the gradient flow of $\pm \beta$ gives a deformation retract of each larger set onto the corresponding smaller one (this is standard if $t_1$ and $t_2$ are regular values, and otherwise, it follows from \cite[Theorem 5]{Lo} as in Lemma~\ref{analytic} thanks to real analyticity). Thus $\{ \beta \leq t_1\}$ is contractible in $\overline{\Omega} \setminus \{ \beta \geq t_2\}$ if and only if it is contractible, and $\{ \beta \geq t_2\}$ is contractible in  $ \overline{\Omega} \setminus \{ \beta \leq t_1\}$ if and only if it is contractible. On the other hand, the sets $\{ \beta \leq t_1\}$ and $\{ \beta \geq t_2\}$ are not simply connected by Proposition \ref{nosimplyconnected}. Hence both of them are not contractible and therefore mutually linked. 
\end{proof}

In the final result of this section, which contains Theorem \ref{topology} as a particular case, we summarize the topological information obtained as a straightforward combination of Lemma \ref{notempty}, Lemma  \ref{analytic}, Corollary \ref{analyticcor}, Proposition \ref{nosimplyconnected} and Proposition \ref{biaxiallinking}. 

\begin{theorem}
\label{Topology}
Let $\Omega \subset \mathbb{R}^3$ be a bounded open set with boundary of class $C^1$, and $Q:\overline{\Omega}\to\mathcal{S}_0$. Assume that $\Omega$ and $Q$ satisfy $(HP_0)-(HP_3)$ (e.g. $\partial \Omega$ has an odd number of connected components and that $Q(x)=\sqrt{3/2} (\overset{\rightarrow}{n}(x) \otimes \overset{\rightarrow}{n}(x)-\frac13 I)$ on $\partial \Omega$, so that $\bar{\beta}=1$). Then the biaxiality sets satisfy:
\begin{itemize}
\item[1)] the set of singular values of $\beta$ in $[-1,\bar{\beta}]$ is at most countable and can accumulate only at $\bar{\beta}$; moreover, for any regular value $-1<t<\bar{\beta}$,  the set $\{ \beta=t\} \subset \Omega$ is a smooth surface with a connected component of positive genus;
\vskip3pt
\item[2)] for any $-1\leq t_1<t_2 < \bar{\beta}$, the sets $\{ \beta \leq t_1\} \subset \Omega$ and $\{ \beta\geq t_2\}\subset \overline{\Omega}$ are nonempty, compact, and not simply connected; 
\vskip3pt
\item[3)] if in addition $Q \in C^\omega(\overline{\Omega})$ and $\bar{\beta}=1$, then $\{ \beta=1 \}\subset \overline{\Omega}$ is also nonempty, compact, and not simply connected; in particular $\{ \beta=1 \}\cap \Omega$ is not empty;   
\vskip3pt
\item[4)] for any $-1\leq t_1<t_2 < \bar{\beta}$ such that $(t_1,t_2)$ contains no singular value, the sets $\{ \beta\leq t_1\}$ and $\{ \beta \geq t_2\}$ are mutually linked.
\end{itemize}
\end{theorem}

%%%%%%%%%%%%%%%%%%%%%%%%%%%%%%%%%%%%%%%%%%%%%%%%%%%%%%%
%%%%%%%%%%%%%%%%%%%%%%%%%%%%%%%%%%%%%%%%%%%%%%%%%%%%%%%
  
%=======================
% BIBLIOGRAPHY AND INDEX
%=======================

%%%%%%%%%%%%%%%%%%%%%%%%%%%%%%%%%%%%%%%%%%%%%%%%%%%%%%%
%%%%%%%%%%%%%%%%%%%%%%%%%%%%%%%%%%%%%%%%%%%%%%%%%%%%%%%

\end{document}